\documentclass[12pt]{article}
\usepackage{mathrsfs,amsthm,graphicx,color,verbatim,bbm,amsmath,amsfonts,amssymb,newclude,nicefrac,graphicx,enumerate,hyperref,bm,geometry,mathabx}
\usepackage[T1]{fontenc}
\usepackage{todonotes}
\usepackage[capitalise]{cleveref} 

\theoremstyle{plain}
\newtheorem{theorem}{Theorem}[section]
\newtheorem{lemma}[theorem]{Lemma}

\newtheorem{cor}[theorem]{Corollary}

\theoremstyle{remark}

\theoremstyle{definition}

\newcommand{\E}{\mathbb{E}}
\newcommand{\F}{\mathbb{F}}
\renewcommand{\P}{\mathbb{P}}

\newcommand{\R}{\mathbb{R}}
\newcommand{\N}{\mathbb{N}}

\newcommand{\V}{\mathbb{V}}

\newcommand{\X}{\mathbb{X}}

\newcommand{\Borel}{\mathcal{B}}
\newcommand{\cC}{\mathcal{C}}
\newcommand{\cF}{\mathcal{F}}
\newcommand{\cK}{\mathcal{K}}

\newcommand{\cO}{\mathcal{O}}

\newcommand{\cV}{\mathcal{V}}
\newcommand{\cW}{\mathcal{W}}

\newcommand{\Exp}[1]{ \E \! \left[ #1 \right]}

\newcommand{\EXP}[1]{ \E  [ #1 ]}

\newcommand{\norm}[1]{ \left\| #1 \right\| }
\newcommand{\Norm}[1]{ \| #1 \| }
\newcommand{\Normm}[1]{ \big\| #1 \big\| }

\newcommand{\HSnorm}[1]{ \left\vvvert #1 \right\vvvert }
\newcommand{\HSNorm}[1]{
\vvvert #1 \vvvert }
\newcommand{\qandq}{\qquad\text{and}\qquad}

\newcommand{\Exists}{\exists\,}
\newcommand{\Forall}{\forall\,}

\makeatletter
\newcommand{\vast}{\bBigg@{3.5}}
\newcommand{\Vast}{\bBigg@{4}}
\makeatother

\newcounter{AuthorCount}
\stepcounter{AuthorCount}

\begin{document}

\title{On existence and uniqueness \\ 
	properties for solutions of \\ 
	stochastic fixed point equations}

\author{
	Christian Beck$^{\arabic{AuthorCount}}$, 
	\stepcounter{AuthorCount}
	Lukas Gonon$^{\arabic{AuthorCount}\stepcounter{AuthorCount},
		\arabic{AuthorCount}\stepcounter{AuthorCount}}$, 
	Martin Hutzenthaler$^{\arabic{AuthorCount}\stepcounter{AuthorCount}}$, 
	and 
	Arnulf Jentzen$^{\arabic{AuthorCount}\stepcounter{AuthorCount}}$
\bigskip
\setcounter{AuthorCount}{1}
	\\
	\small{$^{\arabic{AuthorCount}
			\stepcounter{AuthorCount}}$ 
		Department of Mathematics, 
		ETH Zurich, 
		Z\"urich,}\\
	\small{Switzerland, 
		e-mail: christian.beck@math.ethz.ch} 
	\\
	\small{$^{\arabic{AuthorCount}
			\stepcounter{AuthorCount}}$ 
		Department of Mathematics, 
		ETH Zurich, 
		Z\"urich,}\\
	\small{Switzerland, 
		e-mail: lukas.gonon@math.ethz.ch}
	\\
	\small{$^{\arabic{AuthorCount}
			\stepcounter{AuthorCount}}$ 
		Faculty of Mathematics and Statistics, University of St.~Gallen,}\\
	\small{St.~Gallen, 
		Switzerland, 
		e-mail:	lukas.gonon@unisg.ch}
	\\
	\small{$^{\arabic{AuthorCount}
			\stepcounter{AuthorCount}}$ 
		Faculty of Mathematics, 
		University of Duisburg-Essen,} \\
	\small{Essen, 
		Germany, 
		e-mail: martin.hutzenthaler@uni-due.de}
	\\
	\small{$^{\arabic{AuthorCount}
			\stepcounter{AuthorCount}}$ 
		Department of Mathematics, 
		ETH Zurich, 
		Z\"urich,}\\
	\small{Switzerland, 
		e-mail: arnulf.jentzen@sam.math.ethz.ch} 
}

\maketitle

\begin{abstract}
The Feynman--Kac formula implies that every suitable classical solution of a semilinear Kolmogorov partial differential equation (PDE) is also a solution of a certain stochastic fixed point equation (SFPE). In this article we study such and related SFPEs. In particular, the main result of this work proves existence of unique solutions of certain SFPEs in a general setting. As an application of this main result we establish the existence of unique solutions of SFPEs associated with semilinear Kolmogorov PDEs with Lipschitz continuous nonlinearities even in the case where the associated semilinear Kolmogorov PDE does not possess a classical solution. 
\end{abstract}

\pagebreak

\tableofcontents

\section{Introduction}
\label{sec:intro}

The Feynman--Kac formula implies that every suitable classical solution of a semilinear Kolmogorov partial differential equation (PDE) is also a solution of a certain stochastic fixed point equation (SFPE). In this article we study such and related SFPEs. The main result of this article, \cref{abstract_existence} in \cref{subsec:abstract_existence_uniqueness} below, shows the existence of unique solutions of certain SFPEs in an abstract setting. 
As an application of \cref{abstract_existence} we establish in \cref{existence_sde_setting}  the existence of unique solutions of SFPEs associated with semilinear Kolmogorov PDEs with Lipschitz continuous nonlinearities even in the case where the associated semilinear Kolmogorov PDE does not possess a classical solution (see, for example, Hairer et al.~\cite{HaHuJe2017_LossOfRegularityKolmogorov}). To illustrate \cref{existence_sde_setting} in more detail we provide in the following result, \cref{introduction:theorem} below, a special case of \cref{existence_sde_setting}.

\begin{theorem}
\label{introduction:theorem}
 Let 
	$d,m\in\N$, 
	$L,T\in (0,\infty)$, 
 let 
 	$\langle\cdot,\cdot\rangle\colon\R^d\times\R^d\to\R$ 
 	be the standard scalar product on $\R^d$, 
 let 
	$\norm{\cdot}\colon \R^d \to [0,\infty)$ 
	be a norm on $\R^d$, 
 let 
	$\HSnorm{\cdot}\colon \R^{d\times m}\to [0,\infty)$ be a norm on $\R^{d\times m}$, 
 let 
	$\mu\colon\R^d\to\R^d$ and 
	$\sigma\colon \R^d\to\R^{d\times m}$
	be locally Lipschitz continuous, 
let 
	$f\in C([0,T]\times\R^d\times\R,\R)$, 
	$g\in C(\R^d,\R)$ 
be at most polynomially growing,  	
assume for all 
	$t\in [0,T]$, 
	$x\in \R^d$, 
	$v,w\in \R$ 
that 
	$\max\{\langle x,\mu(x)\rangle,\HSNorm{\sigma(x)}^2\} 
	\leq L ( 1 + \norm{x}^2 ) $ 
and
	$|f(t,x,v)-f(t,x,w)| \leq L|v-w|$, 
 let 
	$(\Omega,\cF,\P,(\F_t)_{t\in [0,T]})$ be a filtered probability space which satisfies the usual conditions, 
 let 
	$W\colon [0,T]\times\Omega\to\R^m$ be a standard $(\F_t)_{t\in [0,T]}$-Brownian motion, 
and for every 
	$t\in [0,T]$, 
	$x\in \R^d$ 
let 
	$X^{t,x} = (X^{t,x}_s)_{s\in [t,T]}\colon [t,T]\times\Omega \to \R^d$ 
be an $(\F_s)_{s\in [t,T]}$-adapted stochastic process with continuous sample paths which satisfies that for all 
	$s\in [t,T]$ 
it holds $\P$-a.s.~that 
	\begin{equation} 
	X^{t,x}_s = x + \int_t^s \mu(X^{t,x}_r)\,dr + \int_t^s \sigma(X^{t,x}_r)\,dW_r. 
	\end{equation} 
Then there exists a unique at most polynomially growing $u\in C([0,T]\times\R^d,\R)$ such that for all 
		$t\in [0,T]$, 
		$x\in \R^d$ 
it holds that 
	\begin{equation} 
	\label{introduction_theorem:stochastic_fixed_point_equation}
	u(t,x) 
	= 
	\Exp{g(X^{t,x}_{T}) + \int_t^T f\big(s,X^{t,x}_{s},u(s,X^{t,x}_{s})\big)\,ds}
	\!.
	\end{equation}  
\end{theorem}

SFPEs of the form as in \eqref{introduction_theorem:stochastic_fixed_point_equation} 
have a strong connection with semilinear Kolmogorov PDEs and arise, for example, in 
models from the environmental sciences as well as in pricing problems from financial engineering (cf., for example, Burgard \& Kjaer~\cite{burgard2011partial}, Cr\'epey et al.~\cite{crepey2013counterparty}, Duffie et al.~\cite{Duffie1996RecursiveValuation}, and  Henry-Labord\`ere~\cite{henry2012counterparty}). SFPEs such as \eqref{introduction_theorem:stochastic_fixed_point_equation} are also important for full-history recursive multilevel Picard approximation (MLP) methods, which were recently introduced in \cite{hutzenthaler2016multilevel,hutzenthaler2018overcoming}; see also \cite{AllenCahnApproximation2019,EHutzenthalerJentzenKruse2019MLP,hutzenthaler2019overcoming,hutzenthaler2017multi}. 
In \cite{hutzenthaler2018overcoming,hutzenthaler2019overcoming} it has been shown that functions which satisfy SFPEs related to semilinear Kolmogorov PDEs can be approximated by MLP schemes  without the curse of dimensionality. \cref{introduction:theorem} above establishes existence of unique solutions of SFPEs related to semilinear Kolmogorov PDEs with Lipschitz continuous nonlinearities within the class of at most polynomially growing continuous functions.
\cref{introduction:theorem} is an immediate consequence of \cref{existence_of_fixpoint_polynomial_growth} in \cref{subsec:SFEs_and_SDEs} below. \cref{existence_of_fixpoint_polynomial_growth}, in turn, follows from \cref{existence_of_fixpoint_x_dependence_full_space_lyapunov} which itself is a special case of  \cref{existence_sde_setting}.
\cref{existence_sde_setting} is an application of \cref{abstract_existence}, the main result of this article. \cref{existence_sde_setting} shows the existence of unique solutions of SFPEs associated with suitable semilinear Kolmogorov PDEs with Lipschitz continuous nonlinearities within a certain class of continuous functions. 
Related existence and uniqueness results can be found, e.g., in Pazy~\cite[Theorem 6.1.2]{Pazy1983Semigroups}, 
Segal~\cite[Theorem 1]{Segal1963NonlinearSemigroups}, Weissler~\cite[Theorem 1]{Weissler1979SemilinearEvolutionEquations}, 
and Hutzenthaler et al.~\cite[Corollary 3.11]{hutzenthaler2018overcoming}. 

The remainder of this article is organized as follows. In \cref{sec:abstract_study} we investigate SFPEs in an abstract setting. In \cref{abstract_existence} in \cref{subsec:abstract_existence_uniqueness}, the main result of this article, we obtain under suitable assumptions an abstract existence and uniqueness result for solutions of SFPEs. Its proof is based on Banach's fixed point theorem. In \crefrange{subsec:integrability}{subsec:continuity} we establish the well-definedness of the mapping to which Banach's fixed point theorem is applied in the proof of \cref{abstract_existence}. In \cref{subsec:lipschitz} we prove a Lipschitz estimate which establishes the contractivity property of the mapping to which Banach's fixed point theorem is applied in the proof of \cref{abstract_existence}. In \cref{sec:applications} we apply the abstract theory from \cref{abstract_existence} in  \cref{sec:abstract_study} in the context of certain stochastic differential equations (SDEs) to obtain \cref{existence_sde_setting}, the main result of \cref{sec:applications}. In \crefrange{subsec:moment_estimates}{subsec:continuous_dependence} we present several auxiliary results on certain SDEs in order to demonstrate that the hypotheses of \cref{abstract_existence} are satisfied in the setting of \cref{existence_sde_setting}. The article is concluded by means of two simple corollaries of \cref{existence_sde_setting} (see \cref{existence_of_fixpoint_x_dependence_full_space_lyapunov} and \cref{existence_of_fixpoint_polynomial_growth} in \cref{subsec:SFEs_and_SDEs} below).

\section{Abstract stochastic fixed point equations (SFPEs)}
\label{sec:abstract_study}

In this section we study SFPEs from an abstract point of view. This section's main result is \cref{abstract_existence} below. It is an application of Banach's fixed point theorem to a suitable function. 
\cref{stochastic_fixed_point_semilinear_continuity} in \cref{subsec:continuity} establishes the well-definedness of this function. \cref{stochastic_fixed_point_semilinear_continuity} is a direct consequence of \cref{continuity_general_case} which we establish through an approximation argument building upon \cref{convergence_result,continuity_compactly_supported_case}.
The contractivity property of the function to which we apply Banach's fixed point theorem in the proof \cref{abstract_existence} is established in \cref{stochastic_fixed_point_measurable_setting_lipschitz_property} in \cref{subsec:lipschitz} below. 

\subsection{Integrability properties for certain stochastic processes}
\label{subsec:integrability}

\begin{lemma}
	\label{well_definedness_of_the_integrals}
Let 
	$d\in\N$, 
	$T\in (0,\infty)$, 
let 
	$\cO \subseteq \R^d$ 
be a non-empty open set, 
let 
	$(\Omega,\cF,\P)$ 
be a probability space, 
for every 
	$t\in [0,T]$, 
	$x\in \cO$
let 
	$X^{t,x}=(X^{t,x}_s)_{s\in [t,T]}\colon [t,T]\times\Omega \to \cO$ 
	be 
	$(\Borel([t,T])\otimes \cF)$/$\Borel(\cO)$-measurable, 
let 
	$g\colon \cO \to \R$ 
	be $\Borel(\cO)$/$\Borel(\R)$-measurable, 
let 
	$h\colon [0,T]\times \cO \to \R$ 
	be $\Borel([0,T]\times\cO)$/$\Borel(\R)$-measurable, 
let 
	$V\colon [0,T]\times\cO \to (0,\infty)$ 
	be  $\Borel([0,T]\times\cO)$/$\Borel((0,\infty))$-measurable, 
and assume for all 
	$t\in [0,T]$, 
	$s\in [t,T]$,
	$x\in \cO$ 
that 
	$\Exp{V(s,X^{t,x}_s)} \leq V(t,x)$
and 
	$
	\sup_{t\in [0,T]}
	\sup_{x\in \cO} 
	[ 
	\frac{|g(x)|}{V(T,x)} 
	+ 
	\frac{|h(t,x)|}{V(t,x)} 
	] 
	< 
	\infty
	$. 
Then it holds for all 
	$t\in [0,T]$,
	$x\in\cO$
that 
	\begin{equation} 
	\label{well_definedness_of_the_integrals:claim}
	\Exp{|g(X^{t,x}_T)| 
		+
		\int_t^T 
		|h(s,X^{t,x}_s)|\,ds} 
	< \infty.  
	\end{equation}
\end{lemma}

\begin{proof}[Proof of \cref{well_definedness_of_the_integrals}]
Throughout this proof let 
	$c\in [0,\infty)$ 
satisfy for all 
	$t\in [0,T]$, 
	$x\in\cO$ 
that 
	\begin{equation}
	\label{well_definedness_of_the_integrals:growth_condition_quantified}
	|g(x)| \leq c V(T,x) 
	\qandq
	|h(t,x)| \leq c V(t,x). 
	\end{equation}
Observe that the hypothesis that $g\colon\cO\to\R$ is  $\Borel(\cO)$/$\Borel(\R)$-measurable, the hypothesis that $h\colon [0,T]\times\cO \to \R$ is $\Borel([0,T]\times\cO)$/$\Borel(\R)$-measurable, and the hypothesis that for every 
	$t\in [0,T]$,
	$x\in\cO$ 
it holds that $X^{t,x}\colon [t,T]\times\Omega\to\cO$ is $(\Borel([t,T])\otimes\cF)$/$\Borel(\cO)$-measurable ensure that for every 
	$t\in [0,T]$,
	$x\in\cO$
it holds that $\Omega \ni \omega \mapsto g(X^{t,x}_T(\omega))\in \R$ is 
$\cF$/$\Borel(\R)$-measurable and $[t,T]\times\Omega \ni (s,\omega) \mapsto h(s,X^{t,x}_s(\omega))\in\R$ is  $(\Borel([t,T])\otimes\cF)$/$\Borel(\R)$-measurable. The hypothesis that for all 
	$t\in [0,T]$, 
	$s\in [t,T]$, 
	$x\in \cO$ 
it holds that $\EXP{V(s,X^{t,x}_s)}\leq V(t,x)$, Fubini's theorem, and \eqref {well_definedness_of_the_integrals:growth_condition_quantified} hence ensure that for all 
	$t\in [0,T]$,
	$x\in\cO$ 
it holds that 
	\begin{equation}
	\begin{split}
	& 
	\Exp{
		|g(X^{t,x}_T)| 
		+ 
		\int_t^T |h(s,X^{t,x}_s)|\,ds
	} 
	= 
	\Exp{|g(X^{t,x}_T)|} + 
	\int_t^T \Exp{|h(s,X^{t,x}_s)|}\,ds
	\\
	& \leq 
	\Exp{cV(T,X^{t,x}_T)} 
	+ 
	\int_t^T 
	\Exp{cV(s,X^{t,x}_s)} \,ds 
	\leq 
	c V(t,x)
	+ 
	\int_t^T cV(t,x) \,ds 
	\\[1ex]
	& \leq 
	c(1+T) V(t,x) < \infty. 
	\end{split}
	\end{equation}
This demonstrates \eqref{well_definedness_of_the_integrals:claim}. The proof of \cref{well_definedness_of_the_integrals} is thus completed.
\end{proof} 

\subsection{Continuity properties for solutions of SFPEs}

In this section we establish in \cref{convergence_result}, \cref{approximation_result_V1}, and \cref{approximation_result} several elementary convergence and approximation results. The convergence result in \cref{convergence_result} and the  approximation result in \cref{approximation_result} pave the way for \cref{subsec:continuity}. They will together with \cref{continuity_compactly_supported_case} be employed in the proof of \cref{continuity_general_case} in \cref{subsec:continuity}. \cref{continuity_general_case}, in turn, has \cref{stochastic_fixed_point_semilinear_continuity} as a rather direct consequence, which itself is one of the cornerstones of the proof of \cref{abstract_existence}.  

\begin{lemma}
	\label{convergence_result}
Let 
	$d\in\N$,
	$T\in (0,\infty)$, 
let 
	$\norm{\cdot}\colon\R^d\to [0,\infty)$ be a norm on $\R^d$, 
let 
	$\cO\subseteq \R^d$ be a non-empty open set, 
for every 
	$r\in (0,\infty)$ let 
	$O_r \subseteq \cO$ 
	satisfy 
	$O_r
	= 
	\{x\in\cO\colon \norm{x} \leq r
	~\text{and}~   \{
	y\in\R^d\colon \norm{y-x} < \nicefrac{1}{r}
	\}
	\subseteq \cO\}
	$, 
let 
	$(\Omega,\cF,\P)$ be a probability space, 
for every 
	$t\in [0,T]$, 
	$x\in \cO$
let 
	$
	X^{t,x}=
	(X^{t,x}_{s})_{s\in [t,T]}
	\colon [t,T] \times \Omega \to \cO
	$ 
	be 	$(\Borel([t,T])\otimes\cF)$/$\Borel(\cO)$-measurable, 
let 
	$V\in C([0,T]\times\cO,(0,\infty))$ satisfy 
for all 
	$t\in [0,T]$, 
	$s\in [t,T]$, 
	$x\in \cO$
that 
	$
	\EXP{V(s,X^{t,x}_{s})} \leq V(t,x)
	$, 
let 
	$g_{n} \in C(\cO,\R)$, $n\in\N_0$, 
	and 
	$h_{n} \in C([0,T]\times\cO,\R)$, $n\in\N_0$, 
satisfy for all 
	$n\in\N$ 
that 
	$
	\inf_{r\in (0,\infty)} 
	[
	\sup_{t\in [0,T]}
	\sup_{x\in \cO\setminus O_r} 
	(
	\frac{|g_{n}(x)|}{V(T,x)} + \frac{|h_{n}(t,x)|}{V(t,x)}
	)
	]
	= 
	0
	$, 
and assume that 
	\begin{equation} 
	\label{convergence_result:convergence_ass}
	\limsup_{n\to\infty} 
	\left[ 
	\sup_{t\in [0,T]} 
	\sup_{x\in\cO} 
	\left( 
	\frac{|g_{n}(x)-g_{0}(x)|}{V(T,x)} 
	+ 
	\frac{|h_{n}(t,x) - h_{0}(t,x)|}{V(t,x)} 
	\right) 
	\right] 
	= 0.  
	\end{equation}     
Then 
	\begin{enumerate}[(i)]
	\item 
	\label{convergence_result:item0a}
	it holds for every 
		$n\in\N_0$
	that 
	\begin{equation} 
		\sup_{t\in [0,T]} 
		\sup_{x\in \cO} 
		\left[ 
 		 \frac{|g_n(x)|}{V(T,x)} 
		 + 
		 \frac{|h_n(t,x)|}{V(t,x)}
 		\right] 
		< \infty,
	\end{equation} 

	\item \label{convergence_result:item0b}
	it holds for every 
		$n\in\N_0$
	that there exists a unique  
		$u_{n}\colon [0,T]\times\cO\to\R$
	which satisfies for all 
		$t\in [0,T]$, 
		$x\in\cO$ 
that 
	\begin{equation} 
	\label{convergence_result:definition_of_un}
	u_{n}(t,x) 
	= 
	\Exp{g_{n}(X^{t,x}_{T}) + \int_t^T h_{n}(s,X^{t,x}_{s})\,ds}\!, 
	\end{equation}

	\item 
	\label{convergence_result:item1}
	it holds that 
		\begin{equation} 
		\limsup_{n\to\infty} 
		\left[ 
		\sup_{t\in [0,T]}
		\sup_{x\in \cO} 
		\left(        
		\frac{|u_{n}(t,x) - u_{0}(t,x)|}{V(t,x)} 
		\right) 
		\right] 
		= 0,  
		\end{equation}
	and 
	\item 
	\label{convergence_result:item2}
	it holds for every 
		compact set $\cK\subseteq \cO$ 
	that 
		\begin{equation} 
		\limsup_{n\to\infty}
		\left[ 
		\sup_{t\in [0,T]}
		\sup_{x\in \cK} 
		|u_{n}(t,x) - u_{0}(t,x)|	
		\right] 
		= 0. 
		\end{equation}
	\end{enumerate}
\end{lemma}

\begin{proof}[Proof of \cref{convergence_result}]
First, observe that for every 
	$r\in (0,\infty)$ 
it holds that 
	$O_r$
is a compact set. This and the fact that for every 
	$n\in\N_0$ 
it holds that 
	$\cO \ni x \mapsto \frac{g_n(x)}{V(T,x)} \in \R$ 
and 
	$[0,T]\times\cO \ni (t,x) \mapsto \frac{h_n(t,x)}{V(t,x)} \in \R$ 
are continuous imply that for all
	$n\in\N_0$, 
	$r\in (0,\infty)$ 
it holds that 
	\begin{equation} 
		\sup \left(\left\{ \frac{|g_n(x)|}{V(T,x)} + 
		\frac{|h_n(t,x)|}{V(t,x)}  
		\colon t\in [0,T], x\in O_r \right\} \cup\{0\}\right) < \infty. 
	\end{equation}
The hypothesis that for every 
	$n\in\N$
it holds that 
	$\inf_{r\in (0,\infty)} [
	\sup_{t\in [0,T]}
	\sup_{x\in \cO\setminus O_r} 
	(
	\frac{|g_{n}(x)|}{V(T,x)} + \frac{|h_{n}(t,x)|}{V(t,x)}
	)
	]
	= 
	0
	$
hence ensures that for every 
	$n\in\N$ 
it holds that 
	\begin{equation} 
	 \sup_{t\in [0,T]}
 	 \sup_{x\in \cO} 
	 \left[ 
      \frac{|g_n(x)|}{V(T,x)} + \frac{|h_n(t,x)|}{V(t,x)}
 	 \right] 
	 < \infty. 
	\end{equation} 
Combining this with \eqref{convergence_result:convergence_ass} 
demonstrates Item~\eqref{convergence_result:item0a}. 
Next observe that Item~\eqref{convergence_result:item0a} 
and \cref{well_definedness_of_the_integrals} establish
Item~\eqref{convergence_result:item0b}. Next note that the 
hypothesis that for all 
	$t\in [0,T]$, 
	$s\in [t,T]$, 
	$x\in \cO$ 
it holds that 
	$\EXP{V(s,X^{t,x}_s)} \leq V(t,x)$ 
ensures that for all 
	$n\in\N$, 
	$t\in [0,T]$, 
	$x\in\cO$
it holds that 
	\begin{equation}
	\begin{split}
	& 
	\frac{\Exp{|g_{n}(X^{t,x}_{T})-g_{0}(X^{t,x}_{T})|}}{V(t,x)}
	= 
	\Exp{
		\frac{|g_{n}(X^{t,x}_{T})-g_{0}(X^{t,x}_{T})|}{V(T,X^{t,x}_{T})}
		\cdot 
		\frac{V(T,X^{t,x}_{T})}{V(t,x)}}
	\\
	& \qquad \qquad \leq 
	\left[
	\sup_{y\in \cO}
	\left( 
	\frac{|g_{n}(y)-g_{0}(y)|}
	{V(T,y)}
	\right) \right]
	\cdot 
	\frac{\Exp{V(T,X^{t,x}_{T})}}{V(t,x)} 
	\leq 
	\sup_{y\in \cO}
	\left( 
	\frac{|g_{n}(y)-g_{0}(y)|}
	{V(T,y)}
	\right) . 
	\end{split}
	\end{equation}
	This and \eqref{convergence_result:convergence_ass}
	establish that 
	\begin{equation}
	\label{convergence_result:g_convergence}
	\limsup_{n\to\infty} 
	\left[ 
	\sup_{t\in [0,T]}
	\sup_{x\in \cO} 
	\left( 
	\frac{
		|\Exp{g_{n}(X^{t,x}_{T})}
		-
		\Exp{g_{0}(X^{t,x}_{T})}|
	}{V(t,x)}
	\right)
	\right] 
	= 0. 
	\end{equation}
Furthermore, note that the hypothesis that for all 
	$t\in [0,T]$, 
	$s\in [t,T]$, 
	$x\in \cO$ 
it holds that 
	$\EXP{V(s,X^{t,x}_s)} \leq V(t,x)$ 
assures that for all 
	$n\in\N$, 
	$t\in [0,T]$, 
	$x\in\cO$
it holds that 
	\begin{equation}
	\begin{split}
	& 
	\frac{\Exp{ 
			\int_t^T 
			\left| 
			h_{n}(s,X^{t,x}_s)
			- 
			h_{0}(s,X^{t,x}_s) 
			\right| \,ds 
	}}{V(t,x)}
	\\
	& \qquad \qquad 
	=
	\int_t^T 
	\Exp{\frac{|h_{n}(s,X^{t,x}_s) 
			- 
			h_{0}(s,X^{t,x}_s)|}{V(s,X^{t,x}_s)}
		\cdot 
		\frac{V(s,X^{t,x}_s)}{V(t,x)}}
	\,ds 
	\\
	& 
	\qquad \qquad \qquad \qquad
	\leq 
	\int_t^T 
	\left[ 
	\sup_{r\in [0,T]}
	\sup_{y\in \cO}
	\left( 
	\frac
	{|h_{n}(r,y) - h_{0}(r,y)|}{V(r,y)}
	\right) 
	\right]
	\cdot 
	\frac{\Exp{V(s,X^{t,x}_s)}}{V(t,x)} \,ds 
	\\
	& 
	\qquad \qquad \qquad \qquad 
	\qquad \qquad
	\leq 
	T \cdot 
	\left[ 
	\sup_{r\in [0,T]}
	\sup_{y\in \cO}
	\left( 
	\frac
	{|h_{n}(r,y) - h_{0}(r,y)|}{V(r,y)}
	\right) 
	\right]. 
	\end{split}
	\end{equation}
This and \eqref{convergence_result:convergence_ass} imply that 
	\begin{equation}
	\limsup_{n\to\infty} 
	\left[ 
	\sup_{t\in [0,T]}
	\sup_{x\in \cO} 
	\left( 
	\frac{\left|\Exp{\int_t^T h_{n}(s,X^{t,x}_s)\,ds} - 
		\Exp{\int_t^T h_{0}(s,X^{t,x}_s)\,ds}\right|}{V(t,x)}
	\right) 
	\right] 
	= 
	0.
	\end{equation}
The triangle inequality,  \eqref{convergence_result:definition_of_un}, 
and 	\eqref{convergence_result:g_convergence} hence yield that 
	\begin{equation}
	\label{stochastic_fixed_point_continuity:u_convergence}
	\limsup_{n\to\infty} 
	\left[ 
	\sup_{t\in [0,T]}
	\sup_{x\in \cO} 
	\left( 
	\frac{|u_{n}(t,x) - u_{0}(t,x)|}{V(t,x)}
	\right) 
 	\right] 
 	= 0. 
	\end{equation} 
This establishes Item~\eqref{convergence_result:item1}. Moreover, observe that Item~\eqref{convergence_result:item1} and the fact that $V\colon [0,T]\times\cO \to (0,\infty)$ is continuous imply for every 
	compact set $\cK \subseteq \cO$ 
that 
	\begin{multline} 
	\limsup_{n\to\infty} 
	\left[ 
	\sup_{t\in [0,T]}
	\sup_{x\in \cK} 
	|u_{n}(t,x) - u_{0}(t,x)|
	\right] 
	\\
	\leq 
	\limsup_{n\to\infty} 
	\left[ 
	\sup_{t\in [0,T]}
	\sup_{x\in \cK} 
	\left( 
	\frac{|u_{n}(t,x) - u_{0}(t,x)|}{V(t,x)}
	\right) 
	\right]
	\left[  
	\sup_{t\in [0,T]}\sup_{x\in \cK} V(t,x)
	\right] 
	= 0.
	\end{multline} 
This establishes 	Item~\eqref{convergence_result:item2}. The proof of \cref{convergence_result} is thus completed. 
\end{proof}

\begin{lemma}
\label{approximation_result_V1}
Let 
	$d\in\N$, 
	$T\in (0,\infty)$, 
let 
	$\norm{\cdot}\colon\R^d\to [0,\infty)$ be a norm on $\R^d$, 
let 
	$\cO\subseteq\R^d$ be a non-empty open set, 
for every 
	$r\in (0,\infty)$ 
let 
	$O_r\subseteq\cO$ 
satisfy 
	$O_r = \{x\in \cO\colon \norm{x}\leq r~\text{and}~\{y\in\R^d\colon\norm{y-x}<\nicefrac{1}{r}\}\subseteq\cO\}$, 
and let
	$h \in C([0,T]\times\cO,\R)$  
satisfy 
	$
	\inf_{r \in (0,\infty)} 
	[
	\sup_{t\in [0,T]}
	\sup_{x\in \cO\setminus O_r}
	|h(t,x)|
	]
	= 
	0$.
Then there exist compactly supported 
	$\mathfrak{h}_n \in C([0,T]\times\cO, \R)$, $n\in\N$, 
which satisfy that 
	\begin{equation} 
	\label{approximation_result_V1:claim}
	\limsup_{n\to\infty} 
	\left[ 
	\sup_{t\in [0,T]}
	\sup_{x\in \cO} 
	|\mathfrak{h}_n(t,x) - h(t,x)|
	\right]  
	= 0.
	\end{equation}
\end{lemma}

\begin{proof}[Proof of \cref{approximation_result_V1}]
Throughout this proof let 
	$U_n\subseteq\cO$, $n\in\N$, 
be the sets given by 
	$U_n = \{x\in\cO\colon (\Exists 
	z\in O_n\colon\norm{z-x}<\frac{1}{2n})\}$.
Note that for every 
	$n\in\N$ 
it holds that $O_n\subseteq\cO$ is a compact set, 
it holds that $U_n\subseteq\R^d$ is an open set which satisfies $U_n\subseteq\cO$, 
and 
it holds that $O_n \subseteq U_{n}$. Urysohn's lemma (cf., for example, Rudin~\cite[Lemma 2.12]{Rudin1987RealAndComplex}) hence ensures for every 
	$n\in\N$ 
that there exists $\varphi_n\in C([0,T]\times\cO,\R)$ which satisfies for all 
	$t\in [0,T]$,
	$x\in\cO$ 
that 
	$\mathbbm{1}_{[0,T]\times O_ n}(t,x) \leq \varphi_n(t,x) \leq 
	\mathbbm{1}_{[0,T]\times U_{n}}(t,x)$. 
Observe, in particular, that this implies that the functions 
	$\varphi_n\colon [0,T]\times\cO\to\R$, 
	$n\in\N$, 
have compact supports. In the next step we let 
	$\mathfrak{h}_n\colon [0,T]\times\cO \to \R$, $n\in\N$, 
satisfy for all 
	$n\in\N$, 
	$t\in [0,T]$, 
	$x\in \cO$ 
that 
	$\mathfrak{h}_n(t,x) = \varphi_n(t,x)h(t,x)$. 
Note that this and the fact that for every 
	$n\in\N$ 
it holds that $\varphi_n\in C( [0,T]\times\cO,\R)$ is compactly supported imply that for every 
	$n\in\N$ 
it holds that 
	$\mathfrak{h}_n\colon [0,T]\times\cO \to \R$ 
is a compactly supported continuous function. Moreover, observe that 
	\begin{equation} 
	\begin{split}
	&
	\limsup_{n\to\infty} 
	\left[ 
	\sup_{t\in [0,T]}
	\sup_{x\in \cO}
	|\mathfrak{h}_n(t,x) - h(t,x)|
	\right]  
	\\
	& \qquad  
	= 
	\limsup_{n\to\infty} 
	\left[ 
	\sup_{t\in [0,T]}
	\sup_{x\in \cO} 
	\left(
	[1-\varphi_n(t,x)]|h(t,x)|
	\right)
	\right] 
	\leq 
	\limsup_{n\to\infty} 
	\left[ 
	\sup_{t\in [0,T]}
	\sup_{x\in \cO\setminus O_n} 
	|h(t,x)|
	\right] 
	= 
	0. 
	\end{split}
	\end{equation} 
	This establishes \eqref{approximation_result_V1:claim}. The proof of \cref{approximation_result_V1} is thus completed.
\end{proof}

\begin{cor}
\label{approximation_result}
 Let 
 	$d\in\N$, 
 	$T\in (0,\infty)$, 
 let 
 	$\norm{\cdot}\colon\R^d\to [0,\infty)$ be a norm on $\R^d$, 
 let 
 	$\cO\subseteq\R^d$ be a non-empty open set, 
 for every 
 	$r\in (0,\infty)$ 
 let 
 	$O_r\subseteq\cO$ 
 	satisfy 
 	$O_r = \{x\in \cO\colon \norm{x}\leq r~\text{and}~\{y\in\R^d\colon\norm{y-x}<\nicefrac{1}{r}\}\subseteq\cO\}$, 
 let
	$h \in C([0,T]\times\cO,\R)$, 
	$V\in C([0,T]\times\cO,(0,\infty))$, 
 and assume that	
$
	\inf_{r \in (0,\infty)} 
	[
	\sup_{t\in [0,T]}
	\sup_{x\in \cO\setminus O_r}
	(
		\frac{|h(t,x)|}{V(t,x)}
	)
	]
	= 
	0$.
Then there exist compactly supported 
 	$\mathfrak{h}_n \in C([0,T]\times\cO, \R)$, $n\in\N$, 
which satisfy that 
 \begin{equation} 
 \label{approximation_result:claim}
  \limsup_{n\to\infty} 
  \left[ 
   \sup_{t\in [0,T]}
   \sup_{x\in \cO} 
   \left(
   \frac{|\mathfrak{h}_n(t,x) - h(t,x)|}{V(t,x)}
   \right)
  \right]  
  = 0.
 \end{equation}
\end{cor}

\begin{proof}[Proof of \cref{approximation_result}] 
Throughout this proof let
	$g\colon [0,T]\times\cO \to \R$ 
satisfy for all 
	$t\in [0,T]$, 
	$x\in \cO$ 
that 
	$g(t,x) = \frac{h(t,x)}{V(t,x)}$. 
Observe that the assumption that 
 	$h \in C([0,T]\times\cO,\R)$, 
the assumption that 
 	$V \in C([0,T]\times\cO,(0,\infty))$, 
and the assumption that 
 	$\inf_{r\in (0,\infty)} [\sup_{t\in [0,T]} \sup_{x\in \cO \setminus O_r} (\frac{|h(t,x)|}{V(t,x)})] = 0$ 
prove that 
	$g\in C([0,T]\times\cO,\R)$ 
and 
	\begin{equation} 
 	 \inf_{r\in (0,\infty)} \left[\sup_{t\in [0,T]} \sup_{x\in \cO\setminus O_r} |g(t,x)|\right] = 0. 
 	\end{equation} 
\cref{approximation_result_V1} (with $h=g$ in the notation of \cref{approximation_result_V1}) therefore ensures that there exist compactly supported 
 	$\mathfrak{g}_n \in C([0,T]\times\cO,\R)$, $n\in\N$, 
 which satisfy that 
 	\begin{equation} 
 	 \limsup_{n\to\infty} \left[ 
 	  \sup_{t\in [0,T]}\sup_{x\in \cO} |\mathfrak{g}_n(t,x) - g(t,x)|
 	 \right] = 0. 
 	\end{equation}
 Next let 
 	$\mathfrak{h}_n\colon [0,T]\times\cO\to\R$, $n\in\N$, 
 satisfy for all 
 	$n\in\N$, 
 	$t\in [0,T]$, 
 	$x\in \cO$ 
 that 
 	$\mathfrak{h}_n(t,x) = \mathfrak{g}_n(t,x) V(t,x)$. 
 Hence, we obtain that for all 
	$n\in\N$ 
 it holds that 
 	$\mathfrak{h}_n \in C([0,T]\times\cO,\R)$ 
 and 
	\begin{equation} 
     \limsup_{k \to \infty} 
     \left[ 
     \sup_{t\in [0,T]}\sup_{x\in\cO} \left(\frac{|\mathfrak{h}_k(t,x) - h(t,x)|}{V(t,x)} \right) 
     \right] 
     = 
     \limsup_{k \to \infty} 
     \left[ 
     \sup_{t\in [0,T]}\sup_{x\in \cO} \left| \mathfrak{g}_k(t,x) - g(t,x) \right|
     \right] 
     = 0. 
    \end{equation} 
 This establishes \eqref{approximation_result:claim}. The proof of \cref{approximation_result} is thus completed.
\end{proof}

\subsection{Regularity properties for solutions of SFPEs}
\label{subsec:continuity}

In this section we establish \cref{stochastic_fixed_point_semilinear_continuity}, one of the building blocks of the proof of \cref{abstract_existence}. \cref{stochastic_fixed_point_semilinear_continuity} is a rather direct consequence of \cref{continuity_general_case} which, in turn, we prove by means of an argument building upon \crefrange{convergence_result}{continuity_compactly_supported_case}. 

\begin{lemma}
\label{continuity_compactly_supported_case}
Let 
 	$d\in\N$, 
 	$T\in (0,\infty)$, 
let 
 	$\norm{\cdot}\colon\R^d\to [0,\infty)$ 
be a norm on $\R^d$, 
let 
 	$\cO \subseteq  \R^d$ 
be a non-empty open set, 
let 
 	$(\Omega,\mathcal{F},\P)$ 
be a probability space, 
for every 
 	$t\in [0,T]$, 
 	$x\in \cO$ 
let 
 		$X^{t,x} = (X^{t,x}_s)_{s\in [t,T]}\colon [t,T]\times\Omega \to \cO$ 
be 
	$(\Borel([t,T])\otimes\cF)$/$\Borel(\cO)$-measurable, 
let 
	$g\in C(\cO, \R)$, 
	$h\in C([0,T]\times\cO, \R)$ 
be bounded, 
and assume for all 
 	$\varepsilon \in (0,\infty)$, 
 	$s\in [0,T]$ 
and all 
 	$(\mathfrak{t}_n,\mathfrak{x}_n)\in [0,T] \times \cO$, $n\in\N_0$,  
with
 	$\limsup_{n\to\infty} [ |\mathfrak{t}_n-\mathfrak{t}_0| + \norm{\mathfrak{x}_n-\mathfrak{x}_0}] = 0$
that 
 	$\limsup_{n\to\infty} [\P(\Norm{X^{\mathfrak{t}_n,\mathfrak{x}_n}_{\max\{s,\mathfrak{t}_n\}}-X^{\mathfrak{t}_0,\mathfrak{x}_0}_{\max\{s,\mathfrak{t}_0\}}}\geq\varepsilon)] = 0$.  	 
Then 
\begin{enumerate}[(i)] 
	\item \label{continuity_compactly_supported_case:item1}
it holds for all 
	$t\in [0,T]$, 
	$x\in \cO$ 
that 
	\begin{equation} 
	\Exp{\left|g(X^{t,x}_{T})\right|
	+ 
	\int_t^T \left| h(s,X^{t,x}_{s}) \right|\!\,ds } 
	< \infty
	\end{equation} 	
and 
	\item \label{continuity_compactly_supported_case:item2}
it holds that 
	\begin{equation}
	\label{stochastic_fixed_point_continuity_compactly_supported_case:definition_u}
	[0,T]\times\cO 
	\ni (t,x) \mapsto 
	\Exp{g(X^{t,x}_{T}) 
		+ 
		\int_t^T h(s,X^{t,x}_{s})\,ds } 
	\in \R
	\end{equation}
is continuous. 
\end{enumerate}
\end{lemma}

\begin{proof}[Proof of \cref{continuity_compactly_supported_case}]
Throughout this proof let 
 	$(\mathfrak{t}_n,\mathfrak{x}_n) \in [0,T] \times \cO$, $n\in\N_0$, 
satisfy  
 	$\limsup_{n\to\infty} [|\mathfrak{t}_n-\mathfrak{t}_0| + \norm{\mathfrak{x}_n-\mathfrak{x}_0}] 
 	= 0$.  
Note that \cref{well_definedness_of_the_integrals} establishes Item~\eqref{continuity_compactly_supported_case:item1}. Next we prove Item~\eqref{continuity_compactly_supported_case:item2}. For this we intend to show that  
	\begin{equation}
	\label{continuity_compactly_supported_case:sufficient_to_prove}
	\limsup_{n\to\infty} 
	\left[ 
	\Exp{|g(X^{\mathfrak{t}_n,\mathfrak{x}_n}_T) - g(X^{\mathfrak{t}_0,\mathfrak{x}_0}_T)|}
	+ 
	\left|\Exp{\int_{\mathfrak{t}_n}^T h(s,X^{\mathfrak{t}_n,\mathfrak{x}_n}_s)\,ds} 
	- 
	\Exp{\int_{\mathfrak{t}_0}^T h(s,X^{\mathfrak{t}_0,\mathfrak{x}_0}_s)\,ds} \right|
	\right]
	= 0.
	\end{equation}
Next note that the fact that $g\colon\cO\to\R$ and $h\colon [0,T]\times\cO\to\R$ are continuous ensures that for all 
 	$\varepsilon \in (0,\infty)$, 
 	$s\in [0,T]$  
it holds that 
	\begin{equation}
	\label{continuity_compactly_supported_case:stochastic_convergence}
	\limsup_{n\to\infty} 
	\big[ 
	\P(
	| 
	g(X^{\mathfrak{t}_n,\mathfrak{x}_n}_T) - g(X^{\mathfrak{t}_0,\mathfrak{x}_0}_{T}) |
	+ 
	| 
	h(s,X^{\mathfrak{t}_n,\mathfrak{x}_n}_{\max\{s,\mathfrak{t}_n\}}) - h(s,X^{\mathfrak{t}_0,\mathfrak{x}_0}_{\max\{s,\mathfrak{t}_0\}}) 
	|
	\geq \varepsilon
	)
	\big] 
	= 0
	\end{equation}
(cf., for example, Kallenberg~\cite[Lemma 4.3]{Kallenberg2002Foundations}). 
Combining this and the fact that $g\colon\cO\to\R$ and $h\colon [0,T]\times\cO\to\R$ are bounded with Vitali's convergence theorem ensures that for all 
 	$s\in [0,T]$ 
it holds that 
	\begin{equation} 
	\label{continuity_compactly_supported_case:first_convergence}
	\limsup_{n\to\infty}  
	\left(
	\Exp{  
	\left|
	g(X^{\mathfrak{t}_n,\mathfrak{x}_n}_T) - g(X^{\mathfrak{t}_0,\mathfrak{x}_0}_T)
	\right|
	}
	+ 
	\Exp{  
	\big|
	h(s,X^{\mathfrak{t}_n,\mathfrak{x}_n}_{\max\{s,\mathfrak{t}_n\}}) 
	- 
	h(s,X^{\mathfrak{t}_0,\mathfrak{x}_0}_{\max\{s,\mathfrak{t}_0\}})
	\big|
	}
	\right) 
	= 0. 
	\end{equation} 
Lebesgue's dominated convergence theorem   
and the fact that $h\colon [0,T]\times\cO \to \R$ is bounded hence imply that 
\begin{equation} 
	 \limsup_{n\to\infty} 
	 \int_{\mathfrak{t}_0}^{T}
	  \Exp{\big| h(s,X^{\mathfrak{t}_n,\mathfrak{x}_n}_{\max\{s,\mathfrak{t}_n\}})
	  - 
	  h(s,X^{\mathfrak{t}_0,\mathfrak{x}_0}_{\max\{s,\mathfrak{t}_0\}})\big|}\!\,ds
	 = 0. 
\end{equation} 
This yields that 
	\begin{equation} 
	\begin{split}
	& 
	\limsup_{n\to\infty} 
	\left| 
		\Exp{ \int_{\mathfrak{t}_n}^T h(s,X^{\mathfrak{t}_n,\mathfrak{x}_n}_s)\,ds }
	    - 
	    \Exp{\int_{\mathfrak{t}_0}^T h(s,X^{\mathfrak{t}_0,\mathfrak{x}_0}_s)\,ds }
	\right| 
	\\
	& 
	= 
	\limsup_{n\to\infty} 
	\left| 
		\Exp{\int_{\mathfrak{t}_n}^T h(s,X^{\mathfrak{t}_n,\mathfrak{x}_n}_{\max\{s,\mathfrak{t}_n\}})\,ds } 
		- 
		\Exp{\int_{\mathfrak{t}_0}^T h(s,X^{\mathfrak{t}_0,\mathfrak{x}_0}_{\max\{s,\mathfrak{t}_0\}})\,ds }
	\right| 
	\\ 
	& 
	= 
	\limsup_{n\to\infty} 
	\left|
	\Exp{
	  	\int_{\mathfrak{t}_n}^{\mathfrak{t}_0}
	   	h(s,X^{\mathfrak{t}_n,\mathfrak{x}_n}_{\max\{s,\mathfrak{t}_n\}})\,ds
	   	+ 
		\int_{\mathfrak{t}_0}^T
	    \left(
	    	h(s,X^{\mathfrak{t}_n,\mathfrak{x}_n}_{\max\{s,\mathfrak{t}_n \} }) -h(s,X^{\mathfrak{t}_0,\mathfrak{x}_0}_{\max\{s,\mathfrak{t}_0 \} })
	    \right)\!\,ds
	}
	\right|
	\\
	& \leq 
	\limsup_{n\to\infty} 
	\left( 
	\Exp{ 
	\left|
		\int_{\mathfrak{t}_n}^{\mathfrak{t}_0} 
		h(s,X^{\mathfrak{t}_n,\mathfrak{x}_n}_{\max\{s,\mathfrak{t}_n\}})\,ds	
	\right|}
	+
	\Exp{
		\int_{\mathfrak{t}_0}^T \big| h(s,X^{\mathfrak{t}_n,\mathfrak{x}_n}_{\max\{s,\mathfrak{t}_n \}}) - h(s,X^{\mathfrak{t}_0,\mathfrak{x}_0}_{\max\{s,\mathfrak{t}_0\}})
	\big|\,ds} 
	\right) 
	\\
	& 
	\leq 
	\limsup_{n\to\infty} 
	\left( 
	|\mathfrak{t}_n-\mathfrak{t}_0| \!
	\left[ 
	\sup_{s\in [0,T]}\sup_{y\in \cO} |h(s,y)| \right]
	+ 
	\int_{\mathfrak{t}_0}^T  \Exp{\big|h(s,X^{\mathfrak{t}_n,\mathfrak{x}_n}_{\max\{s,\mathfrak{t}_n \} }) - h(s,X^{\mathfrak{t}_0,\mathfrak{x}_0}_{\max\{s,\mathfrak{t}_0\}}) \big|} \!\,ds 
	\right)
	\\
	& = 0. 
	\end{split}
	\end{equation} 
Combining this with \eqref{continuity_compactly_supported_case:first_convergence} demonstrates \eqref{continuity_compactly_supported_case:sufficient_to_prove}. The proof of \cref{continuity_compactly_supported_case} is thus completed.
\end{proof}

\begin{lemma}
	\label{continuity_general_case}
Let 
	$d\in\N$, 
	$T\in (0,\infty)$, 
let 
	$\norm{\cdot}\colon\R^d\to [0,\infty)$  
be a norm on $\R^d$,
let 
	$\cO\subseteq \R^d$ 
be a non-empty open set,
for every 
	$r\in (0,\infty)$ 
let 
	$O_r \subseteq \cO$ 
satisfy  
	$O_r 
	= 
	\{x\in\cO\colon \norm{x} \leq r
	~\text{and}~    \{
	y\in\R^d\colon \norm{y-x} < \nicefrac{1}{r}
	\}
	\subseteq \cO\}
	$, 
let 
	$(\Omega,\mathcal{F},\P)$ 
be a probability space, 
for every 
	$t\in [0,T]$, 
	$x\in \cO$ 
let 
	$X^{t,x} = (X^{t,x}_s)_{s\in [t,T]}\colon [t,T]\times\Omega \to \cO$ 
be $(\Borel([t,T])\otimes\cF)$/$\Borel(\cO)$-measurable, 
assume for all 
	$\varepsilon \in (0,\infty)$, 
	$s\in [0,T]$ 
and all 
	$(t_n,x_n)\in [0,T] \times \cO$, $n\in\N_0$,  
with 
	$\limsup_{n\to\infty} [ |t_n-t_0| + \norm{x_n-x_0}] = 0$ 
that 
	$\limsup_{n\to\infty} [\P(\Norm{X^{t_n,x_n}_{\max\{s,t_n\}}-X^{t_0,x_0}_{\max\{s,t_0\}}}\geq\varepsilon)] = 0$,
let
	$g \in C(\cO,\R)$, 
	$h \in C([0,T]\times\cO,\R)$, 
	$V\in C([0,T]\times\cO,(0,\infty))$ 
	and 
	$u\colon [0,T]\times\cO \to \R$
	satisfy for all 
		$t\in [0,T]$, 
		$s\in [t,T]$, 
		$x\in \cO$ 
	that 
	$
	\Exp{V(s,X^{t,x}_s)} 
	\leq 
	V(t,x)
	$, 
and assume for all 
	$t\in [0,T]$, 
	$x\in \cO$ 
that	
	$
	\inf_{r \in (0,\infty)} 
	[ 
	\sup_{s\in [0,T]}
	\sup_{y\in \cO\setminus O_r}
	( 
	\frac{|g(y)|}{V(T,y)} + \frac{|h(s,y)|}{V(s,y)}
	) 
	]
	= 
	0$
and 
	\begin{equation}
	\label{continuity_general_case:definition_u}
	u(t,x) = 
	\Exp{g(X^{t,x}_T) 
	+ 
	\int_t^T h(s,X^{t,x}_{s})\,ds }
	\end{equation}
(cf.\ Item~\eqref{convergence_result:item0b} of \cref{convergence_result}).
Then 
	\begin{enumerate}[(i)]
	\item
	\label{continuity_general_case:item1}
	it holds that $u\in C([0,T]\times\cO,\R)$ and  
	\item
	\label{continuity_general_case:item2} 
	it holds in the case of 
		$\sup_{r\in (0,\infty)}[ 
		\inf_{t\in [0,T]} 
		\inf_{x\in \cO\setminus O_r} 
		V(t,x)
		]
		= 
		\infty$ that 
		\begin{equation} 
		\lim_{r\to\infty} 
		\left[ 
		\sup_{t\in [0,T]}
		\sup_{x\in  \cO\setminus O_r} 
		\left( 
		\frac{|u(t,x)|}{V(t,x)}
		\right) 
		\right] = 0. 
		\end{equation}
	\end{enumerate}
\end{lemma}

\begin{proof}[Proof of \cref{continuity_general_case}]
Throughout this proof let 
	$\mathfrak{g}_{n}\colon\cO\to\R$, $n\in\N$, 
and 
	$\mathfrak{h}_{n}\colon [0,T]\times\cO\to\R$, $n\in\N$, 
be compactly supported continuous functions which satisfy that 
	\begin{equation}
	\label{continuity_general_case:g_h_approximation}
	\limsup_{n\to\infty} 
	\left[ 
	\sup_{t\in [0,T]}
	\sup_{x\in \cO} 
	\left( 
	\frac{| \mathfrak{g}_{n}(x) - g(x) |}
	{V(T,x)}
	+ 
	\frac{| \mathfrak{h}_{n}(t,x) - h(t,x) |}{V(t,x)} 
	\right) 
	\right] 
	= 0  
	\end{equation}
(cf.~\cref{approximation_result})
and let 
	$\mathfrak{u}_n\colon [0,T]\times\cO \to \R$, $n\in\N$, 
satisfy for all 
	$n\in\N$, 
	$t\in [0,T]$, 
	$x\in \cO$ 
that 
	\begin{equation} 
	\mathfrak{u}_n(t,x) 
	= 
	\Exp{\mathfrak{g}_{n}(X^{t,x}_{T}) 
	+ 
	\int_t^T 
	\mathfrak{h}_{n}(s,X^{t,x}_s)\,ds}
	\end{equation} 
(cf.~\cref{well_definedness_of_the_integrals}). 
Note that \cref{continuity_compactly_supported_case} assures for every 
	$n\in\N$
that $\mathfrak{u}_n\colon [0,T]\times\cO \to \R$ 
is continuous. 
Next observe that the fact that 
	$\mathfrak{g}_n\colon \cO \to \R$, $n\in\N$, 
and 
	$\mathfrak{h}_n\colon [0,T]\times\cO\to\R$, 
	$n\in\N$, 
are compactly supported ensures that for every 
	$n\in\N$ 
there exists 
	$r \in (0,\infty)$ 
which satisfies that for all 
	$t\in [0,T]$, 
	$x\in \cO \setminus O_r$
it holds that 
	$\mathfrak{g}_n(x) = 0 = \mathfrak{h}_n(t,x)$. 
This implies for every 
	$n\in\N$ 
that 
\begin{equation}
 \inf_{r\in (0,\infty)} 
 \left[ 
 \sup_{t\in [0,T]} 
 \sup_{x\in\cO\setminus O_r} 
 \left( 
  \frac{|\mathfrak{g}_n(x)|}{V(T,x)} 
  + 
  \frac{|\mathfrak{h}_n(t,x)|}{V(t,x)}
 \right)  
 \right] 
 = 0. 
\end{equation}
Item~\eqref{convergence_result:item2} of \cref{convergence_result},  \eqref{continuity_general_case:g_h_approximation}, and the fact that $\mathfrak{u}_n\colon [0,T]\times\cO\to\R$, $n\in\N$, are continuous therefore imply that $u\colon [0,T]\times\cO\to\R$ is continuous. This establishes Item~\eqref{continuity_general_case:item1}. In the next step we prove Item~\eqref{continuity_general_case:item2}. For this we assume that 
	$
	\sup_{r\in (0,\infty)} 
	[ 
	\inf_{t\in [0,T]}
	\inf_{x\in \cO\setminus O_r} 
	V(t,x)
	]
	=
	\infty
	$.  
Note that this entails for every 
	$n\in\N$ 
that 
	\begin{equation}
	\begin{split}
	& 
	\limsup_{r\to\infty} 
	\left[ 
	\sup_{t\in [0,T]}
	\sup_{x\in \cO\setminus O_r} 
	\left( 
	\frac{|\mathfrak{u}_{n}(t,x)|}{V(t,x)}
	\right) 
	\right] 
	\\
	&
	\leq 
	\limsup_{r\to\infty} 
	\left( 
	\left[ 
	\sup_{t\in [0,T]}
	\sup_{x\in \cO} 
	 \left|\mathfrak{u}_n(t,x)\right| 
	\right] 
	\left[ 
	\sup_{t\in [0,T]}
	\sup_{x\in \cO\setminus O_r} 
	\left( 
	\frac{1}{V(t,x)}
	\right) 
	\right]
	\right)
	\\
	& 
	\leq 
	\limsup_{r\to\infty} 
	\left( 
	\left[ 
	\left( 
	\sup_{x\in\cO} 
	|\mathfrak{g}_{n}(x)| 
	\right)
	+ 
	T 
	\left(
	\sup_{t\in [0,T]}
	\sup_{x\in \cO} 
	|\mathfrak{h}_{n}(t,x)|
	\right)
	\right]
	\left[ 
	\sup_{t\in [0,T]}
	\sup_{x\in \cO\setminus O_r} 
	\left( 
	\frac{1}{V(t,x)}
	\right) 
	\right] 
	\right) 
	= 0. 
	\end{split}
	\end{equation}
Combining this with Item~\eqref{convergence_result:item1} of 
\cref{convergence_result} yields that  
\begin{equation}
\begin{split}
 &
 \limsup_{r\to\infty} 
 \left[ 
 \sup_{t\in [0,T]}
 \sup_{x\in \cO\setminus O_r} 
 \left( 
 \frac{|u(t,x)|}{V(t,x)}
 \right) 
 \right]
 \\
 & \leq 
 \inf_{n\in\N} 
 \left(
  \limsup_{r\to\infty} 
  \left[ 
   \sup_{t\in [0,T]}
  \sup_{x\in \cO\setminus O_r} 
  \left( 
  \frac{|u(t,x)-\mathfrak{u}_n(t,x)|+|\mathfrak{u}_n(t,x)|}{V(t,x)}
  \right) 
  \right] 
 \right)
 \\
 & = 
 \inf_{n\in\N} 
 \left( 
  \limsup_{r\to\infty}
   \left[ 
  \sup_{t\in [0,T]}
  \sup_{x\in \cO\setminus O_r} 
  \left( 
  \frac{|u(t,x)-\mathfrak{u}_n(t,x)|}{V(t,x)}
  \right) 
  \right]
 \right)
 \\
 & 
 \leq 
 \inf_{n\in\N} 
 \left( 
  \sup_{t\in [0,T]}
 \sup_{x\in \cO} 
 \left( 
 \frac{|u(t,x)-\mathfrak{u}_n(t,x)|}{V(t,x)}
 \right) 
 \right) 
 \\
 & 
 \leq 
 \limsup_{n\to\infty} 
 \left[
 \sup_{t\in [0,T]}
 \sup_{x\in \cO} 
 \left( 
 \frac{|u(t,x)-\mathfrak{u}_n(t,x)|}{V(t,x)}
 \right) 
 \right] 
 = 0. 
\end{split}
\end{equation}
This establishes Item\,\eqref{continuity_general_case:item2}. The proof of \cref{continuity_general_case} is thus completed. 
\end{proof}

\cref{continuity_general_case} allows to infer the next result, \cref{stochastic_fixed_point_semilinear_continuity}, which constitutes an important ingredient of the proof of \cref{abstract_existence}. 

\begin{cor}
\label{stochastic_fixed_point_semilinear_continuity}
Let 
	$d\in\N$, 
	$L,T\in (0,\infty)$, 
let 
	$\norm{\cdot}\colon\R^d \to [0,\infty)$ 
be a norm on $\R^d$,
let 
	$\cO\subseteq \R^d$ 
be a non-empty open set,
for every 
	$r\in (0,\infty)$ 
let 
	$O_r \subseteq \cO$ 
satisfy 
	$O_r 
	= 
	\{x\in\cO\colon \norm{x} \leq r
	~\text{and}~    
	\{
	y\in\R^d\colon \norm{y-x} < \nicefrac{1}{r}
	\}
	\subseteq \cO\}
	$, 
let 
	$(\Omega,\mathcal{F},\P)$ 
be a probability space, 
for every 
	$t\in [0,T]$, 
	$x\in \cO$ 
let 
	$
	X^{t,x}
	=
	(X^{t,x}_{s})_{s\in [t,T]}
	\colon [t,T]\times\Omega \to \cO
	$ 
be $(\Borel([t,T])\otimes\cF)$/$\Borel(\cO)$-measurable, 
assume for all 
	$\varepsilon \in (0,\infty)$, 
	$s\in [0,T]$ 
and all 
	$(t_n,x_n)\in [0,T] \times \cO$, $n\in\N_0$,  
with
	$\limsup_{n\to\infty} [ |t_n-t_0| + \norm{x_n-x_0}] = 0$
that 
	$\limsup_{n\to\infty} [\P(\Norm{X^{t_n,x_n}_{\max\{s,t_n\}}-X^{t_0,x_0}_{\max\{s,t_0\}}}\geq\varepsilon)] = 0$,
let 
	$f\in C([0,T]\times\cO\times\R,\R)$, 
	$g\in C(\cO,\R)$, 
	$u\in C([0,T]\times\cO,\R)$, 
	$V\in C([0,T]\times\cO,(0,\infty))$ 
satisfy for all 
	$t\in [0,T]$, 
	$s\in [t,T]$, 
	$x\in\cO$ 
that 
	$
	\EXP{V(s,X^{t,x}_{s})} 
	\leq 
	V(t,x)
	$, 
and assume for all 
	$t\in [0,T]$, 
	$x\in\cO$, 
	$v,w\in\R$ 
that 
	$\inf_{r\in (0,\infty)} 
	[ 
	\sup_{s\in [0,T]}
	\sup_{y\in\cO\setminus O_r}
	( 
	\frac{|f(s,y,0)|+|u(s,y)|}{V(s,y)}
	+ 
	\frac{|g(y)|}{V(T,y)}
	) 
	] 
	= 
	0
	$
and 
	$
	|f(t,x,v)-f(t,x,w)| 
	\leq 
	L | v - w |
	$.  
Then
	\begin{enumerate}[(i)]
	\item
	\label{stochastic_fixed_point_semilinear_continuity:item0}
	it holds for all 
		$t\in [0,T]$, 
		$x\in \cO$ 
	that 
		\begin{equation}
		\Exp{
		|g(X^{t,x}_{T})| + 
		\int_t^T |f(s,X^{t,x}_{s},u(s,X^{t,x}_{s}))|\,ds}
		< 
		\infty, 
		\end{equation}
	\item 
	\label{stochastic_fixed_point_semilinear_continuity:item1}
	it holds that 
		\begin{equation}
		\label{stochastic_fixed_point_semilinear_continuity:claim}
		[0,T]\times\cO \ni (t,x) 
		\mapsto 
		\Exp{g(X^{t,x}_{T}) 
			+ 
			\int_t^T f\big(s,X^{t,x}_{s},u(s,X^{t,x}_{s})\big)\,ds}
		\in \R
		\end{equation}
		is continuous, and 
	\item
	\label{stochastic_fixed_point_semilinear_continuity:item2} 
	it holds in the case of 
		$
		\sup_{r\in (0,\infty)}
		[ 
		\inf_{t\in [0,T]}
		\inf_{x\in \cO\setminus O_r}
		V(t,x)
		] = \infty$
	that 
		\begin{equation}
		\lim_{r\to\infty} 
		\left[ 
		\sup_{t\in [0,T]}
		\sup_{x\in \cO\setminus O_r} 
		\left( 
		\frac{
			\left|
			\Exp{g(X^{t,x}_{T}) 
				+ 
				\int_t^T f\big(s,X^{t,x}_{s},u(s,X^{t,x}_{s})\big)\,ds
			}
			\right|
		}{V(t,x)}
		\right) 
		\right] = 0.
		\end{equation}
	\end{enumerate}
\end{cor}

\begin{proof}[Proof of \cref{stochastic_fixed_point_semilinear_continuity}]
First, observe that 
	\begin{equation}
	[0,T]\times\cO 
	\ni 
	(t,x) 
	\mapsto 
	f(t,x,u(t,x)) \in \R
	\end{equation}
is a continuous function which satisfies for all 
	$t\in [0,T]$, 
	$x\in \cO$ 
that 
	\begin{equation}
	|f(t,x,u(t,x))| 
	\leq 
	|f(t,x,0)| + L |u(t,x)| . 
	\end{equation}
The hypothesis that 
	$
	\inf_{r\in (0,\infty)} 
	[
	\sup_{t\in [0,T]}
	\sup_{x\in \cO\setminus O_r} 
	( 
	\frac{|f(t,x,0)|+|u(t,x)|}{V(t,x)}
	)
	]
	= 0
	$ 
therefore ensures that 
	\begin{equation}
	\begin{split}
	& \inf_{r\in (0,\infty)} 
	\left[ 
	\sup_{t\in [0,T]}
	\sup_{x\in\cO\setminus O_r}
	\left( 
	\frac{|f(t,x,u(t,x))|}{V(t,x)}
	\right)
	\right] 
	\\
	& 
	\qquad 
	\leq 
	\inf_{r\in (0,\infty)}
	\left[ 
	\sup_{t\in [0,T]}
	\sup_{x\in\cO\setminus O_r}
	\left( 
	\frac{|f(t,x,0)|}{V(t,x)}
	+ 
	L
	\frac{|u(t,x)|}{V(t,x)}
	\right)
	\right] 
	= 
	0. 
	\end{split}
	\end{equation}
\cref{well_definedness_of_the_integrals} and Item  \eqref{convergence_result:item0a} of \cref{convergence_result} hence establish Item~\eqref{stochastic_fixed_point_semilinear_continuity:item0}. Moreover, \cref{continuity_general_case} (with 
	$g=g$, 
	$h=([0,T]\times\cO\ni (t,x) \mapsto f(t,x,u(t,x))\in\R)$, 
	$u=([0,T]\times\cO\ni (t,x) \mapsto 
	\EXP{g(X^{t,x}_T)+\int_t^T f(s,X^{t,x}_s,u(s,X^{t,x}_s))\,ds}\in\R)$
in the notation of \cref{continuity_general_case}) establishes Items\, \eqref{stochastic_fixed_point_semilinear_continuity:item1} and \eqref{stochastic_fixed_point_semilinear_continuity:item2}. The proof of \cref{stochastic_fixed_point_semilinear_continuity}
	is thus completed.
\end{proof}

\subsection{Contractivity properties for SFPEs}
\label{subsec:lipschitz}

In this section we establish an elementary Lipschitz estimate (see  \cref{stochastic_fixed_point_measurable_setting_lipschitz_property} below) which will yield the contractivity needed in the proof of \cref{abstract_existence}. 

\begin{lemma}
	\label{stochastic_fixed_point_measurable_setting_lipschitz_property}
	Let 
	$d\in\N$, 
	$L,T\in (0,\infty)$,
	let 
	$\cO \subseteq  \R^d$ 
	be a non-empty open set, 
let 
	$(\Omega,\mathcal{F},\P)$ 
be a probability space, 
for every 
	$t\in [0,T]$, 
	$x\in\cO$	
let 
	$
	X^{t,x}=(X^{t,x}_s)_{s\in [t,T]}\colon [t,T]\times\Omega\to\cO$ 
be 
	$(\Borel([t,T]) \otimes \mathcal{F})$/$\Borel(\cO)$-measurable, 
let 
	$V\colon [0,T]\times\cO\to (0,\infty) $ 
be 
	$\Borel([0,T]\times\cO)$/$\Borel((0,\infty))$-measurable, 
assume for all 
	$t\in [0,T]$, 
	$s\in [t,T]$, 
	$x\in\cO$ 
that 
	$
	\Exp{V(s,X^{t,x}_{s})}
	\leq 
	V(t,x)
	$, 
let 
	$f\colon [0,T]\times\cO\times\R\to\R$ 
be  $\Borel([0,T]\times\cO\times\R)$/$\Borel(\R)$-measurable,  
assume for all 
	$t\in [0,T]$, 
	$x\in\cO$, 
	$v,w\in \R$ 
that 
	$
	|f(t,x,v)-f(t,x,w)| 
	\leq L | v -  w |$, 
let 
	$v,w\colon [0,T]\times\cO\to\R$ 
be 
	$\Borel([0,T]\times\cO)$/$\Borel(\R)$-measurable, 
and assume that
	\begin{equation}
	\sup_{t\in [0,T]}
	\sup_{x\in \cO} 
	\left[
	\frac{
		|v(t,x)| + |w(t,x)|}{V(t,x)} 
	\right] 
	< 
	\infty. 
	\end{equation}
Then it holds for all 
	$\lambda \in (0,\infty)$, 
	$t\in [0,T]$, 
	$x\in\cO$ 
that 
	\begin{equation}
	\label{stochastic_fixed_point_measurable_setting_lipschitz_property:claim}
	\begin{split}
	& 
	\Exp{
		\int_t^T \left| 
		f\big(s, X^{t,x}_{s}, v(s,X^{t,x}_{s})\big) - 
		f\big(s, X^{t,x}_{s}, w(s,X^{t,x}_{s})\big) 
		\right| \,ds
	}
	\\
	& 
	\qquad \qquad 
	\leq 
	\frac{L}{\lambda} e^{-\lambda t} V(t,x)
	\left[ 
	\sup_{s\in [0,T]}
	\sup_{y\in \cO} 
	\left( 
	\frac{e^{\lambda s} |v(s,y) - w(s,y)|}{V(s,y)}
	\right) 
	\right] . 
	\end{split}
	\end{equation}
\end{lemma}

\begin{proof}[Proof of \cref{stochastic_fixed_point_measurable_setting_lipschitz_property}]
First, note that the fact that 
	$f\colon [0,T]\times\cO\times\R\to\R$ 
is $\Borel([0,T]\times\cO\times\R)$/$\Borel(\R)$-measurable, 
the fact that 
	$v,w\colon [0,T]\times\cO\to\R$ 
are $\Borel([0,T]\times\cO)$/$\Borel(\R)$-measurable, 
and the fact that for all 
	$t\in [0,T]$, 
	$x\in\cO$ 
it holds that 
	$X^{t,x} \colon [t,T]\times\Omega\to\cO$ 
is  $(\Borel([t,T])\otimes\cF)$/$\Borel(\cO)$-measurable ensure that for all 
	$t\in [0,T]$, 
	$x\in\cO$ 
it holds that
	\begin{equation}
	[t,T] \times \Omega 
	\ni (s,\omega)  
	\mapsto 
	\left| 
	f\big( 
	s, X^{t,x}_{s}(\omega),
	v(s, X^{t,x}_{s}(\omega))
	\big) 
	- 
	f\big( 
	s, X^{t,x}_{s}(\omega),
	w(s, X^{t,x}_{s}(\omega))
	\big) 
	\right|\in\R
	\end{equation}
is $(\Borel([t,T]) \otimes \mathcal{F})$/$\Borel(\R)$-measurable. Next observe that the hypothesis that for all 
	$t\in [0,T]$, 
	$s\in [t,T]$, 
	$x\in \cO$ 
it holds that 
	$\EXP{V(s,X^{t,x}_s)} \leq V(t,x)$, 
the hypothesis that for all 
	$t\in [0,T]$, 
	$x\in \cO$, 
	$a,b\in\R$ 
it holds that 
	$
	|f(t,x,a)-f(t,x,b)| \leq L|a-b|
	$, 
and Fubini's theorem ensure that for all
	$\lambda\in (0,\infty)$, 
	$t\in [0,T]$, 
	$x\in\cO$
it holds that 
	\begin{equation}
	\begin{split}
	& 
	\Exp{ 
		\int_t^T \left| 
		f\big(s,X^{t,x}_{s},v(s,X^{t,x}_{s})\big) 
		- 
		f\big(s,X^{t,x}_{s},w(s,X^{t,x}_{s})\big) 
		\right| \,ds
	}
	\\
	& 
	\leq 
	\Exp{ 
		\int_t^T 
		L \left| v(s,X^{t,x}_{s}) - w(s,X^{t,x}_{s})\right| \,ds 
	}
	\\
	& 
	= 
	L 
	\int_t^T \Exp{\frac{e^{\lambda s}|v(s,X^{t,x}_{s})-w(s,X^{t,x}_{s})|}{V(s,X^{t,x}_{s})}V(s,X^{t,x}_{s})} e^{-\lambda s}\,ds
	\\
	& 
	\leq 
	L
	\int_t^T 
	\left[ 
	\sup_{r\in [0,T]}
	\sup_{y\in\cO} 
	\left( 
	\frac{e^{\lambda r}|v(r,y)-w(r,y)|}{V(r,y)}
	\right) 
	\right]
	\Exp{V(s,X^{t,x}_{s})} e^{-\lambda s}\,ds
	\\
	& 
	\leq 
	L \left[ 
	\sup_{r\in [0,T]}
	\sup_{y\in\cO} 
	\left( 
	\frac{e^{\lambda r}|v(r,y)-w(r,y)|}{V(r,y)}
	\right) 
	\right]
	V(t,x)
	\int_t^T 
	e^{-\lambda s}
	\,ds 
	\\
	& 
	\leq \frac{L}{\lambda} 
	\left[ 
	\sup_{r\in [0,T]}
	\sup_{y\in\cO} 
	\left( 
	\frac{e^{\lambda r}|v(r,y)-w(r,y)|}{V(r,y)}
	\right) 
	\right]
	e^{-\lambda t} V(t,x)
	. 
	\end{split}
	\end{equation}
This establishes \eqref{stochastic_fixed_point_measurable_setting_lipschitz_property:claim}. The proof of \cref{stochastic_fixed_point_measurable_setting_lipschitz_property} 
	is thus completed.
\end{proof}

\subsection{Existence and uniqueness properties for solutions of SFPEs}
\label{subsec:abstract_existence_uniqueness}

Combining Banach's fixed point theorem with \cref{stochastic_fixed_point_semilinear_continuity,stochastic_fixed_point_measurable_setting_lipschitz_property} allows to conclude the main result of this section, \cref{abstract_existence} below. 

\begin{theorem}
	\label{abstract_existence}
Let 
	$d\in\N$, 
	$L,T\in (0,\infty)$, 
let 
	$\norm{\cdot}\colon\R^d \to [0,\infty)$ be a norm on $\R^d$, 
let 
	$\cO\subseteq \R^d$ 
	be a non-empty open set,
for every
	$r\in (0,\infty)$ 
	let 
	$O_r \subseteq \cO$ 
	satisfy 
	$O_r 
	= 
	\{x\in\cO\colon \norm{x} \leq r
	~\text{and}~   \{
	y\in\R^d\colon \norm{y-x} < \nicefrac{1}{r}
	\}
	\subseteq \cO\}
	$, 
let 
	$(\Omega,\mathcal{F},\P)$ 
	be a probability space, 
for every 
	$t\in [0,T]$, 
	$x\in \cO$ 
let 
	$
	X^{t,x}
	=
	(X^{t,x}_{s})_{s\in [t,T]}
	\colon [t,T]\times\Omega \to \cO
	$ 
	be $(\Borel([t,T])\otimes\cF)$/$\Borel(\cO)$-measurable, 
assume for all 
	$\varepsilon \in (0,\infty)$, 
	$s\in [0,T]$ 
and all 
	$(t_n,x_n)\in [0,T] \times \cO$, $n\in\N_0$,  
with 
	$\limsup_{n\to\infty} [ |t_n-t_0| + \norm{x_n-x_0}] = 0$
that 
	$\limsup_{n\to\infty} [\P(\Norm{X^{t_n,x_n}_{\max\{s,t_n\}}-X^{t_0,x_0}_{\max\{s,t_0\}}}\geq\varepsilon)] = 0$,
let 
	$f\in C([0,T]\times\cO\times\R,\R)$, 
	$g\in C(\cO,\R)$, 
	$V\in C([0,T]\times\cO,(0,\infty))$ 
satisfy for all 
	$t\in [0,T]$, 
	$s\in [t,T]$, 
	$x\in\cO$, 
	$v,w\in\R$
that 
	$
	\EXP{V(s,X^{t,x}_{s})} \leq V(t,x)
	$
and 
	$
	|f(t,x,v)-f(t,x,w)|
	\leq 
	L |v-w|
	$, 
and assume that 
	$
	\inf_{r\in (0,\infty)}
	[ 
	\sup_{t\in [0,T]}
	\sup_{x\in \cO\setminus O_r} 
	( 
	\frac{|f(t,x,0)|}{V(t,x)}
	+
	\frac{|g(x)|}{V(T,x)}
	) 
	] 
	= 0
	$
	and 
	$
	\sup_{r\in (0,\infty)} 
	[
	\inf_{t\in [0,T]}
	\inf_{x\in \cO\setminus O_r} 
	V(t,x)
	] = \infty
	$. 
Then there exists a unique 
	$u\in C([0,T]\times\cO,\R)$ 
such that 
\begin{enumerate}[(i)]
\item it holds that 
	\begin{equation}
	\limsup_{r\to\infty} 
	\left[ 
	\sup_{t\in [0,T]}
	\sup_{x\in \cO\setminus O_r} 
	\left( 
	\frac{|u(t,x)|}{V(t,x)}
	\right) 
	\right] 
	= 0
	\end{equation}
and 
\item it holds for all 
	$t\in [0,T]$, 
	$x\in\cO$
that
	\begin{equation}
	u(t,x) 
	= 
	\Exp{ 
		g(X^{t,x}_T) 
		+ 
		\int_t^T 
		f\big( s, X^{t,x}_{s}, u(s,X^{t,x}_{s})\big) \,ds 
	}\!. 
	\end{equation} 
\end{enumerate}
\end{theorem}

\begin{proof}[Proof of \cref{abstract_existence}]
 Throughout this proof let 
	$\cV$ 
 be the set given by
	\begin{equation}
	\cV = \left\{
	u\in C([0,T]\times\cO,\R)\colon 
	\limsup_{r\to\infty} 
	\left[ \sup_{t\in [0,T]}\sup_{x\in\cO\setminus O_r} \left( \frac{|u(t,x)|}{V(t,x)} \right) \right] = 0
	\right\}\!,  
	\end{equation}
let 
 	$\cW_1$ and $\cW_2$ 
be the sets given by 
	\begin{equation}
	\cW_1 = \left\{ u \in C([0,T]\times\cO,\R)\colon 
	\sup_{t\in [0,T]} \sup_{x\in\cO} |u(t,x)|  < \infty \right\}  
	\end{equation} 
and 
	\begin{equation} 
	\cW_2 = \left\{ u \in C([0,T]\times\cO,\R)\colon 
	\limsup_{r\to\infty} \left[ \sup_{t\in [0,T]} \sup_{x\in\cO \setminus O_r} |u(t,x)| \right] = 0 \right\}\!, 
	\end{equation} 
let 
 	$\norm{\cdot}_{\lambda}\colon \cV \to [0,\infty)$, $\lambda\in \R$, 
satisfy for every 
 	$\lambda\in \R$, $v\in\cV$ 
that 
	\begin{equation}
	\|v\|_{\lambda} 
	= 
	\sup_{t\in [0,T]}
	\sup_{x\in \cO} 
	\left( 
	\frac{e^{\lambda t}|v(t,x)|}{V(t,x)}
	\right)  
	\end{equation}
(see Item~\eqref{convergence_result:item0a} of \cref{convergence_result}), 
and let 
	$\norm{\cdot}_{\cW_i}\colon\cW_i \to [0,\infty) $, $i\in\{1,2\}$, 
satisfy for every 
	$i\in\{1,2\}$, 
	$w\in \cW_i$ 
that 
	\begin{equation} 
	 \norm{w}_{\cW_i} = \sup_{t\in [0,T]} \sup_{x\in\cO} |w(t,x)|. 
	\end{equation}
Recall that 
	$(\cW_1,\norm{\cdot}_{\cW_1})$ is an $\R$-Banach space. 
Combining this with the fact that $\cW_2$ is a closed subset of $(\cW_1,\norm{\cdot}_{\cW_1})$ (see \cref{convergence_result}) implies that $(\cW_2,\norm{\cdot}_{\cW_2})$ is an $\R$-Banach space.  
Moreover, observe that 
	$(\cV,\norm{\cdot}_{\lambda})$, $\lambda\in\R$, are normed $\R$-vector spaces. 
In the next step we show that 
	$(\cV,\norm{\cdot}_{0})$
is complete. For this let 
	$v_n\in\cV$, $n\in\N$,
satisfy
	$\limsup_{n\to\infty} [\sup_{m\geq n} \norm{v_n-v_m}_{0}] = 0$. 
This implies that 
	$\frac{v_n}{V}\colon [0,T]\times\cO \to \R$, $n\in\N$, 
is a Cauchy sequence in $(\cW_2,\norm{\cdot}_{\cW_2})$. Thus, there exists $\phi\in \cW_2$ which satisfies that 
	$
	\limsup_{n\to\infty} 
	[
	\sup_{t\in [0,T]} \sup_{x\in\cO} 
    | \frac{v_n(t,x)}{V(t,x)} - \phi(t,x) | 
    ]
    = 
    0$. 
Hence, we obtain that 
	$\phi V = ([0,T]\times\cO \ni (t,x) \mapsto \phi(t,x) V(t,x) \in \R) \in \cV$ 
and 
	$\limsup_{n\to\infty} 
	\norm{v_n - \phi V}_0 = 0$. 
This demonstrates that 
	$(\cV,\norm{\cdot}_{0})$ is an $\R$-Banach space. 
Combining this with the fact that for every 
	$\nu\in\R$, 
	$\lambda\in [\nu,\infty)$, 
	$v\in \cV$ 
it holds that 
	$\norm{v}_{\nu} \leq \norm{v}_{\lambda} \leq e^{(\lambda-\nu) T}\norm{v}_{\nu} $ 
shows that for every 
	$\lambda\in\R$ 
it holds that 
	$(\cV,\norm{\cdot}_{\lambda})$ is an $\R$-Banach space.
Next note that \cref{stochastic_fixed_point_semilinear_continuity} yields that there exists a unique  $\Phi\colon\cV\to\cV$ which satisfies for all 
	$t\in [0,T]$, 
	$x\in\cO$, 
	$v\in\cV$ 
that 
	\begin{equation}
	\left[\Phi(v)\right]\!(t,x) 
	= 
	\Exp{ 
		g(X^{t,x}_T) 
		+ 
		\int_t^{T} 
		f\big(s,X^{t,x}_{s},v(s,X^{t,x}_{s})\big)\,ds
	}\!.
	\end{equation}
Moreover, observe that \cref{stochastic_fixed_point_measurable_setting_lipschitz_property} ensures for all 
	$\lambda\in (0,\infty)$, 
	$v,w\in\cV$ 
that 
	\begin{equation}
	\| \Phi(v) - \Phi(w) \|_{\lambda} 
	\leq 
	\frac{L}{\lambda} 
	\| v - w \|_{\lambda}.  
	\end{equation}
Hence, we obtain for all 
	$\lambda \in [2L,\infty)$, 
	$v,w\in \cV$ 
that 
	\begin{equation}
	\| \Phi(v) - \Phi(w) \|_{\lambda} 
	\leq 
	\frac12 
	\| v - w \|_{\lambda}. 
	\end{equation}
Banach's fixed point theorem therefore demonstrates that there exists a unique $u\in\cV$ which satisfies $\Phi(u)=u$. 
The proof of \cref{abstract_existence} is thus completed.
\end{proof}

\section{SFPEs associated with stochastic differential equations (SDEs)}
\label{sec:applications}

In this section we apply the abstract existence and uniqueness result which we obtained in the previous section (see \cref{abstract_existence} in \cref{sec:abstract_study} above) to certain SDEs (see \cref{subsec:SFEs_and_SDEs} below). In \crefrange{subsec:moment_estimates}{subsec:continuous_dependence} we present, for the reader's convenience and for the sake of completeness, some elementary and essentially well-known results on SDEs. These results are employed to show that the hypotheses of \cref{abstract_existence} are indeed satisfied in the setting of \cref{existence_sde_setting} (cf.\ \cref{integral_estimate_through_lyapunov_function,stochastic_continuity_lemma}). 

\subsection{A priori estimates for solutions of SDEs} 
\label{subsec:moment_estimates}

The following well-known result, 
\cref{integral_estimate_through_lyapunov_function} below (cf., for example, Gy\"ongy \& Krylov~\cite{GyoengyKrylov1995_existenceStrong}), 
can be seen as an extension of moment bounds for solutions 
of SDEs in the presence 
of a Lyapunov function or, in other words, a non-negative supersolution of the corresponding Kolmogorov PDE.  

\begin{lemma}
\label{integral_estimate_through_lyapunov_function}
Let 
	$d,m\in\N$, 
	$T \in (0,\infty)$, 
let 
	$\cO \subseteq  \R^d$ be an open set, 
let 
	$\langle\cdot,\cdot\rangle\colon \R^d \times \R^d \to \R$ be the standard scalar product on $\R^d$, 
let
	$
	\mu
	\in C([0,T]\times\cO,\R^d)
	$, 
	$
	\sigma\in C([0,T]\times\cO,\R^{d\times m})
	$,
	$V\in C^{1,2}([0,T]\times \cO, [0,\infty))$
satisfy for all 
	$t \in [0,T]$, 
	$x\in\cO$ 
that
	\begin{equation}
	\label{integral_estimate_through_lyapunov_function:ass1}
	(\tfrac{\partial V}{\partial t})(t,x) 
	+ 
	\tfrac12 
	\operatorname{Trace}\!\big( 
	\sigma(t,x)[\sigma(t,x)]^{*} 
	(\operatorname{Hess}_x V)(t,x)
	\big) 
	+ 
	\langle 
	\mu(t,x),(\nabla_x V)(t,x) \rangle
	\leq 0,
	\end{equation}
let 
	$ ( \Omega, \mathcal{F}, \P, (\mathbb{F}_t )_{t\in [0,T]}) $ 
be a filtered probability space which satisfies the usual conditions, 
let 
	$ W \colon [0,T] \times \Omega \to \R^m $ 
be a standard $(\mathbb{F}_t)_{t\in [0,T]}$-Brownian motion, 
let 
	$\tau\colon \Omega\to [0,T]$ be an 
	$(\mathbb{F}_t)_{t\in [0,T]}$-stopping time, 
and let 
	$X
	\colon 
	[0,T] \times \Omega \to \cO$ 
be an $(\mathbb{F}_t)_{t\in [0,T]}$-adapted stochastic process with continuous sample paths which satisfies that for all 
	$t\in [0,T]$ 
it holds $\P$-a.s.\,that
	\begin{equation}
	\label{integral_estimate_through_lyapunov_function:ass2}
	X_t 
	= 
	X_{0}
	+ 
	\int_{0}^t \mu(s,X_s)\,ds 
	+ 
	\int_{0}^t \sigma(s,X_s)\,dW_s. 
	\end{equation}
Then it holds that 
	\begin{equation}
	\Exp{V(\tau,X_{\tau})} 
	\leq 
	\Exp{V(0,X_{0})}
	. 
	\end{equation}
\end{lemma}

\begin{proof}[Proof of \cref{integral_estimate_through_lyapunov_function}]
Throughout this proof let 
	$\norm{\cdot}\colon \R^d \to [0,\infty)$ 
be the standard norm on $\R^d$, 
let 
	$\HSnorm{\cdot}\colon \R^{d\times m} \to [0,\infty)$ 
be the Frobenius norm on $\R^{d\times m}$,  
for every 
	$ r \in (0,\infty) $ 
let 
	$ O_r \subseteq \cO $ 
satisfy 
	$ O_r = \{ x\in\cO \colon \norm{x} \leq r~\text{and}~\{y \in \R^d \colon \norm{y-x} < \nicefrac{1}{r} \} \subseteq \cO \} $, 
let 
	$Y\colon [0,T]\times\Omega \to \R$ 
be an $(\F_t)_{t\in [0,T]}$-adapted stochastic  process with continuous sample paths which satisfies that for all 
	$t\in [0,T]$ 
it holds $\P$-a.s.~that 
	\begin{equation}
	\label{integral_estimate_through_lyapunov_function:definition_of_Y}
	Y_t = \int_{0}^t 
	\left\langle (\nabla_x V)(s,X_s), 
	\sigma(s,X_s)\,dW_s 
	\right\rangle, 
	\end{equation}
and let 
	$\rho_n\colon \Omega\to [0,T]$, $n\in\N$, 
be the
	$(\mathbb{F}_t)_{t\in [0,T]}$-stopping times 
which satisfy for all 
	$n\in\N$ 
that 
	\begin{equation}
	\label{integral_estimate_through_lyapunov_function:definition_stopping_time}
	\rho_n 
	= 
	\inf\!\left( \left\{t\in [0,T]\colon 
	X_t \notin O_n
	\right\} \cup \{T\} \right). 
	\end{equation}
Observe that the fact that $X$ has continuous sample paths and the fact 
that $[0,T]$ is compact ensure that for all 
	$\omega\in\Omega$ 
it holds that 
	$\{X_t(\omega)\colon t\in [0,T]\}$ 
is compact. Combining this with the fact that 
	$\R^d \ni x \mapsto \norm{x}\in[0,\infty)$ 
and 
	$\R^d \ni x \mapsto \inf(\{1\}\cup\{\norm{x-y}\colon y\in\R^d\setminus\cO\}) \in [0,1] $
are continuous implies that for every 
	$\omega \in \Omega$ 
there exist
	$\varepsilon,r \in (0,\infty)$ 
such that for all 
	$t\in [0,T]$ 
it holds that 
	$
	\{y\in \R^d\colon \Norm{y-X_t(\omega)}<\varepsilon \} \subseteq \cO 
	$
and 
	$\sup_{t\in [0,T]} \norm{X_t(\omega)} \leq r$. 
Combining this with the fact that for all 
	$\varepsilon,r \in (0,\infty)$ 
there exists 
	$n \in \N$ 
such that for all 
	$k \in \N$ 
with 
	$k \geq n$ 
it holds that  
	$r\leq k$ 
and 
	$\nicefrac{1}{k} \leq \varepsilon $
implies that for every 
	$\omega\in\Omega$ 
there exists 
	$n \in \N$ 
such that for all 
	$k \in \N$ 
with 
	$k \geq n$ 
it holds that  
	$\rho_k(\omega)=T$. 
Next note that the assumption that 
	$\nabla_x V\colon [0,T]\times\cO \to \R^d$ 
and 
	$\sigma\colon [0,T]\times\cO \to \R^{d\times m}$ 
are continuous implies that for all 
	$n\in\N$ 
it holds that 
	\begin{equation} 
	\sup \Big(
	\left\{ 
	\norm{(\nabla_x V)(t,x)} + \HSnorm{\sigma(t,x)} 
	\colon 
	t \in [0,T], 
	x \in O_n
	\right\} \cup \{ 0 \}
	\Big)
	< \infty. 
	\end{equation} 
This yields for all 
	$n\in\N$ 
that 
	\begin{equation} 
	\begin{split}
	 & 
	 \Exp{\int_0^T 
	 	\norm{(\nabla_x V)(s,X_s)}^2 \HSNorm{\sigma(s,X_s)\mathbbm{1}_{\{0<s\leq\min\{\tau,\rho_n\}\}}}^2	 	
	 \,ds} 
 	 \\
 	 & \qquad \leq 
 	 T \left[ 
 	 \sup \Big(
 	 \left\{ 
 	 \norm{(\nabla_x V)(t,x)} + \HSnorm{\sigma(t,x)} 
 	 \colon 
 	 t \in [0,T], 
 	 x \in O_n
 	 \right\} \cup \{ 0 \}
 	 \Big)\right]^4 
 	 < \infty.
 	\end{split}  
	\end{equation} 
Combining \eqref{integral_estimate_through_lyapunov_function:definition_of_Y} and \eqref{integral_estimate_through_lyapunov_function:definition_stopping_time} hence assures for all 
	$n\in\N$ 
that 
	\begin{equation}
	\label{integral_estimate_through_lyapunov_function:zero_expectation}
	\Exp{Y_{\min\{\tau,\rho_n\}}} 
	= 
	\Exp{ 
		\int_0^T \langle (\nabla_x V)(s,X_s), \sigma(s,X_s) 
		\mathbbm{1}_{\{0<s\leq\min\{\tau,\rho_n\}\}}\,dW_s
		\rangle
	}
	= 
	0 . 
	\end{equation}
Next note that It\^o's formula ensures that for all 
	$t\in [0,T]$ 
it holds $\P$-a.s.~that
	\begin{equation}
	\begin{split}
	& 
	V(t,X_t) 
	\\
	& =
	V(0,X_{0}) 
	+ 
	\int_{0}^t 
	(\tfrac{\partial V}{\partial t})(s,X_s) \,ds
	+ 
	\int_{0}^{t}
	\langle 
	(\nabla_x V)(s,X_{s}) ,
	\sigma(s,X_{s})\,dW_s
	\rangle
	\\
	& \quad
	+ \int_{0}^t
	\left[ 
	\langle 
	(\nabla_x V)(s,X_s), 
	\mu(s,X_{s})
	\rangle
	+
	\tfrac12 
	\operatorname{Trace}\!\left(
	\sigma(s,X_{s})
	\!\left[
	\sigma(s,X_{s})\right]^{*}
	\!(\operatorname{Hess}_x V)(s,X_{s})
	\right)
	\right]\!
	\,ds
	\\
	& 
	= 
	V(0,X_{0}) 
	+ 
	\int_{0}^t 
	(\tfrac{\partial V}{\partial t})(s,X_s) \,ds
	+ 
	Y_t
	\\
	& \quad
	+ \int_{0}^t
	\left[ 
	\langle 
	\mu(s,X_{s}), 
	(\nabla_x V)(s,X_s)
	\rangle
	+
	\tfrac12 
	\operatorname{Trace}\!\left(
	\sigma(s,X_{s})
	\!\left[
	\sigma(s,X_{s})\right]^{*}
	\!(\operatorname{Hess}_x V)(s,X_{s})
	\right)
	\right]\!
	\,ds
	.  
	\end{split}
	\end{equation}
This and the fact that $X$ has continuous sample paths imply that for all 
	$n\in\N$ 
it holds $\P$-a.s.~that 
	\begin{equation}
	\begin{split}
	& V(\min\{\tau,\rho_n\}, X_{\min\{\tau,\rho_n\}}) 
	\\
	& 
	=
	V(0, X_{0}) 
	+ 
	Y_{\min\{\tau,\rho_n\}}
	+ 
	\int_{0}^{\min\{\tau,\rho_n\}} 
	(\tfrac{\partial V}{\partial t})(s,X_s) \,ds 
	+  
	\int_{0}^{\min\{\tau,\rho_n\}}
	\langle 
	\mu(s,X_{s}), 
	(\nabla_x V)(s,X_s)
	\rangle
	\,ds
	\\
	& \quad 
	+ \int_{0}^{\min\{\tau,\rho_n\}}
	\tfrac12 
	\operatorname{Trace}\!\left(
	\sigma(s,X_{s})
	\!\left[
	\sigma(s,X_{s})\right]^{*}
	\!(\operatorname{Hess}_x V)(s,X_{s})
	\right)\!
	\,ds. 
	\end{split} 
	\end{equation}
This and \eqref{integral_estimate_through_lyapunov_function:ass1} guarantee that for all 
	$ n \in \N$ 
it holds $\P$-a.s.\,that 
	\begin{equation}
	\label{integral_estimate_through_lyapunov_function:ae_inequality}
	\begin{split}
	V(\min\{\tau,\rho_n\}, X_{\min\{\tau,\rho_n\}})
	& \leq 
	V(0,X_{0})  
	+ 
	Y_{\min\{\tau,\rho_n\}}. 
	\end{split}
	\end{equation}
Combining this and  \eqref{integral_estimate_through_lyapunov_function:zero_expectation} yields for all 
	$n\in\N$ 
that
	\begin{equation}
	\begin{split}
	& \Exp{ 
		V(
		\min\{\tau,\rho_n\},
		X_{\min\{\tau,\rho_n\}}
		)
	}
	\leq 
	\Exp{V(0,X_{0})}. 
	\end{split}
	\end{equation}
Fatou's lemma hence ensures that 
	\begin{equation}
	\begin{split}
	\Exp{V(\tau,X_{\tau})} 
	& = 
	\Exp{\liminf_{n\to\infty}  V(\min\{\tau,\rho_n\}, X_{\min\{\tau,\rho_n\}}) }
	\\
	& \leq 
	\liminf_{n\to\infty} 
	\Exp{V(\min\{\tau,\rho_n\}, X_{\min\{\tau,\rho_n\}})} 
	\leq 
	\Exp{V(0, X_{0})} .
	\end{split}
	\end{equation}
The proof of \cref{integral_estimate_through_lyapunov_function} is thus completed.
\end{proof}    

The next elementary result, \cref{moment_estimate_special_lyapunov} below,  provides a way to construct from a supersolution of a suitable elliptic PDE a  supersolution of a Kolmogorov PDE (cf.~\cref{integral_estimate_through_lyapunov_function}  above). Later we will employ \cref{moment_estimate_special_lyapunov} to infer \cref{existence_of_fixpoint_x_dependence_full_space_lyapunov} from \cref{existence_sde_setting}. 

\begin{lemma}
	\label{moment_estimate_special_lyapunov}
Let 
	$d,m\in\N$, 
	$T\in (0,\infty)$, 
	$\rho \in \R$, 
let 
	$\langle\cdot,\cdot\rangle\colon \R^d \times \R^d \to \R$ be the standard scalar product on $\R^d$, 
let 
	$\cO \subseteq  \R^d$ 
	be a non-empty open set, 
let
	$\mu\in C([0,T]\times\cO,\R^d)$, 
	$\sigma\in C([0,T]\times\cO,\R^{d\times m})$, 
	$V\in C^2(\cO,(0,\infty))$ 
satisfy for all 
	$t\in [0,T]$, 
	$x\in\cO$ 
that
	\begin{equation}
	\label{moment_estimate_special_lyapunov:ass1}
	\tfrac12 
	\operatorname{Trace}\!\big( 
	\sigma(t,x)[\sigma(t,x)]^{*} 
	(\operatorname{Hess} V)(x)
	\big) 
	+ 
	\langle 
	\mu(t,x),(\nabla V)(x)
	\rangle
	\leq \rho V(x), 
	\end{equation}
and let 
	$
	\mathbb{V}
	\colon 
	[0,T]\times\cO
	\to 
	(0,\infty)
	$
satisfy for all 
	$t\in [0,T]$, 
	$x\in \cO$
that 
	\begin{equation}
	\mathbb{V}(t,x) = e^{-\rho t} V(x) . 
	\end{equation} 
Then  
	\begin{enumerate}[(i)]
		\item
		\label{moment_estimate_special_lyapunov:item1}
		it holds that $\mathbb{V}\in C^2([0,T]\times\cO,(0,\infty))$ and
		\item 
		\label{moment_estimate_special_lyapunov:item2}
		it holds for all 
		$t\in [0,T]$, 
		$x\in\cO$ 
		that
		\begin{equation}
		(\tfrac{\partial \mathbb{V}}{\partial t})(t,x) 
		+ 
		\tfrac12\operatorname{Trace}\!\big( 
		\sigma(t,x)[\sigma(t,x)]^{*}
		(\operatorname{Hess}_x \mathbb{V})(t,x)
		\big) 
		+ 
		\langle 
		\mu(t,x), 
		(\nabla_x \mathbb{V})(t,x) 
		\rangle
		\leq 0 . 
		\end{equation}
	\end{enumerate}
\end{lemma}

\begin{proof}[Proof of \cref{moment_estimate_special_lyapunov}]
First, note that the chain rule and the fact that $V\in C^2(\cO,(0,\infty))$ ensure for all 
	$t\in [0,T]$, 
	$x\in\cO$ 
that 
	\begin{enumerate}[(I)]
		\item
		\label{moment_estimate_special_lyapunov:item_V_twice_diff}
		$\mathbb{V}\in C^2([0,T]\times\cO,\R)$,
		\item 
		\label{moment_estimate_special_lyapunov:item_V_partial_t}
		$(\tfrac{\partial \mathbb{V}}{\partial t})(t,x) 
		= 
		-\rho e^{-\rho t} V(x) 
		= 
		-\rho \mathbb{V}(t,x)$, 
		\item
		\label{moment_estimate_special_lyapunov:item_V_nabla_x}
		$(\nabla_x \mathbb{V})(t,x) 
		= e^{-\rho t}(\nabla V)(x)$, and
		\item 
		\label{moment_estimate_special_lyapunov:item_V_Hess_x}
		$(\operatorname{Hess}_x \mathbb{V})(t,x) 
		= e^{-\rho t}(\operatorname{Hess} V)(x)$. 
	\end{enumerate}
Note that Item~\eqref{moment_estimate_special_lyapunov:item_V_twice_diff} establishes Item~\eqref{moment_estimate_special_lyapunov:item1}. Moreover, combining \eqref{moment_estimate_special_lyapunov:ass1} with  Items~\eqref{moment_estimate_special_lyapunov:item_V_partial_t}--\eqref{moment_estimate_special_lyapunov:item_V_Hess_x} yields for all 
	$t\in [0,T]$, 
	$x\in\cO$ 
that 
	\begin{equation}
	\begin{split}
	& 
	(\tfrac{\partial \mathbb{V}}{\partial t})(t,x) 
	+ 
	\tfrac12
	\operatorname{Trace}\!\big( 
	\sigma(t,x)[\sigma(t,x)]^{*}(\operatorname{Hess}_x \mathbb{V})(t,x)
	\big) 
	+ 
	\langle 
	\mu(t,x),(\nabla_x \mathbb{V})(t,x) 
	\rangle
	\\
	&
	= e^{-\rho t} 
	\left( -\rho V(x) 
	+ 
	\tfrac12 
	\operatorname{Trace}\!\big( 
	\sigma(t,x)[\sigma(t,x)]^{*} 
	(\operatorname{Hess} V)(x)
	\big) 
	+ 
	\langle 
	\mu(t,x),(\nabla V)(x)
	\rangle
	\right) 
	\leq 0.  
	\end{split}
	\end{equation}
This establishes Item~\eqref{moment_estimate_special_lyapunov:item2}. The proof of \cref{moment_estimate_special_lyapunov} is thus completed.
\end{proof}

The next elementary result, \cref{polynomials_lyapunov} below, establishes in conjunction with \cref{moment_estimate_special_lyapunov} above that under certain coercivity and linear growth conditions (see \eqref{polynomials_lyapunov:ass1} in \cref{polynomials_lyapunov}) Lyapunov-type functions with polynomial growth are available (cf.\ also Grohs et al.~\cite[Lemma 2.21]{grohs2018proof}).
\cref{polynomials_lyapunov} will later on allow to infer \cref{existence_of_fixpoint_polynomial_growth} from \cref{existence_of_fixpoint_x_dependence_full_space_lyapunov}.

\begin{lemma}
	\label{polynomials_lyapunov}
Let 
	$d,m\in\N$, 
	$c,T,p,\rho\in (0,\infty)$ 
satisfy 
	$\rho = \tfrac{pc}{2} \max\{p+1,3\} 
	$, 
let 
	$\langle\cdot,\cdot\rangle\colon \R^d \times \R^d \to \R$ be the standard scalar product on $\R^d$, 
let 
	$\norm{\cdot}\colon\R^d\to [0,\infty)$ be the  standard norm on $\R^d$,
let 
	$\HSnorm{\cdot}\colon\R^{d\times m}\to [0,\infty)$ be the Frobenius norm on 
	$\R^{d\times m}$,
let 
	$\cO \subseteq  \R^d$ be a non-empty open set, 
and let 
	$\mu\colon [0,T]\times\cO\to\R^d$, 
	$\sigma\colon [0,T]\times\cO\to\R^{d\times m}$, 
	$V\colon \cO \to (0,\infty)$ 
satisfy for all 
	$t\in [0,T]$, 
	$x\in\cO$ 
that 
	\begin{equation}
	\label{polynomials_lyapunov:ass1}
	\max\!\left\{
	\langle x,\mu(t,x)\rangle
	, 
	\HSnorm{\sigma(t,x)}^2
	\right\}
	\leq c ( 1 + \norm{x}^2)
	\qandq 
	V(x) = \left( 1 + \norm{x}^2 \right)^{\nicefrac{p}{2}}. 
	\end{equation}
Then 
	\begin{enumerate}[(i)]
		\item it holds that 
		\label{polynomials_lyapunov:item1}
		$V\in C^{\infty}(\cO,(0,\infty))$ and 
		\item
		\label{polynomials_lyapunov:item2}
		it holds for all 
		$t\in [0,T]$, 
		$x\in\cO$ 
		that 
		\begin{equation}
		\tfrac12 
		\operatorname{Trace}\!
		\big( 
		\sigma(t,x)[\sigma(t,x)]^{*}(\operatorname{Hess} V)(x)
		\big) 
		+ 
		\langle 
		\mu(t,x), 
		(\nabla V)(x) 
		\rangle
		\leq 
		\rho 
		V(x) .
		\end{equation}
	\end{enumerate}
\end{lemma}

\begin{proof}[Proof of \cref{polynomials_lyapunov}] 
Throughout this proof let 
	$\sigma_{i,j}\colon [0,T]\times\cO \to \R$, 
	$i\in\{1,2,\ldots,d\}$, 
	$j\in\{1,2,\ldots,m\}$, 
satisfy for all 
	$t\in [0,T]$, 
	$x\in \cO$ 
that
	\begin{equation} 
	\sigma(t,x) = \begin{pmatrix}
	\sigma_{1,1}(t,x) 
	& 
	\sigma_{1,2}(t,x) 
	& 
	\ldots
	& 
	\sigma_{1,m}(t,x) \\
	\sigma_{2,1}(t,x) 
	&
	\sigma_{2,2}(t,x)
	&
	\ldots 
	&
	\sigma_{2,m}(t,x) \\
	\vdots 
	& 
	\vdots 
	& 
	\ddots 
	& 
	\vdots \\
	\sigma_{d,1}(t,x) 
	& 
	\sigma_{d,2}(t,x) 
	& 
	\ldots 
	& 
	\sigma_{d,m}(t,x) 
	\end{pmatrix} \in\R^{d\times m}.  
	\end{equation}
Note that the chain rule, the fact that $\R^d\ni x \mapsto 1 + \norm{x}^2 \in (0,\infty)$ is infinitely  often differentiable, and   
the fact that 
	$ (0,\infty) \ni s \mapsto s^{\frac{p}{2}} \in (0,\infty)$ 
is infinitely often differentiable establish Item~\eqref{polynomials_lyapunov:item1}.  It thus remains to prove  Item~\eqref{polynomials_lyapunov:item2}. For this, we observe for all 
	$x\in\cO$, 
	$i,j\in \{1,2,\ldots,d\}$ 
that 
	\begin{equation}
	(\nabla V)(x) 
	= 
	\tfrac{p}{2} 
	\left( 1 + \norm{x}^2 \right)^{\frac{p}{2}-1} \cdot (2x) 
	= p V(x) \tfrac{x}{1 + \norm{x}^2}
	\end{equation}
and 
	\begin{equation}
	\begin{split}
	(\tfrac{\partial^2 V}{\partial x_i\partial x_j})(x) 
	& 
	=
	\tfrac{\partial}{\partial x_i}
	\left[ 
	p V(x) \tfrac{x_j}{1+\norm{x}^2}
	\right]
	=
	p \cdot (\tfrac{\partial V}{\partial x_i})(x) \cdot \tfrac{x_j}{1 + \norm{x}^2} 
	+ 
	p V(x) \cdot \left[\tfrac{\partial}{\partial x_i}\!\left( \tfrac{x_j}{1+ \norm{x}^2}\right)\right]
	\\
	& 
	= 
	p^2 V(x) \tfrac{x_i x_j}{(1+\norm{x}^2)^2 } 
	+ 
	p V(x) 
	\tfrac{\delta_{ij}(1 + \norm{x}^2) - 2x_i x_j}{( 1 + \norm{x}^2)^2}
	\\
	& 
	= 
	p(p-2) V(x) \tfrac{x_i x_j}{(1+\norm{x}^2)^2} 
	+  p V(x) \tfrac{\delta_{ij}}{1 + \norm{x}^2}  
	\\
	& 
	= 
	p V(x)
	\left[ 
	(p-2) 
	\tfrac{x_i x_j}
	{(1+\norm{x}^2)^2} 
	+   
	\tfrac{\delta_{ij}}
	{1 + \norm{x}^2}
	\right]\!. 
	\end{split}
	\end{equation}
This yields for all 
	$t\in [0,T]$, 
	$x\in\cO$ 
that 
	\begin{equation} \label{polynomials_lyapunov:some_calculation}
	\begin{split}
	& \tfrac12 
	\operatorname{Trace}\!
	\big( 
	\sigma(t,x)[\sigma(t,x)]^{*}(\operatorname{Hess} V)(x)
	\big) 
	+ 
	\langle 
	\mu(t,x), 
	(\nabla V)(x) 
	\rangle
	\\
	& = 
	\tfrac12 \left[
	\sum_{k=1}^m\sum_{i,j=1}^d \sigma_{i,k}(t,x)\sigma_{j,k}(t,x) (\tfrac{\partial^2 V}{\partial x_i\partial x_j})(x) 
	\right]
	+ 
	\left\langle 
	\mu(t,x), p V(x) \tfrac{x}{1+ \norm{x}^2} 
	\right\rangle
	\\
	& = 
	\tfrac{p}{2} 
	\left[\left(\sum_{k=1}^m\sum_{i,j=1}^d 
	\sigma_{i,k}(t,x) \sigma_{j,k}(t,x) 
	\left(
	(p-2) \tfrac{x_i x_j}{( 1 + \norm{x}^2)^2 } 
	+
	\tfrac{\delta_{ij}}{1+\norm{x}^2}
	\right) 
	\right)
	+ \tfrac{2\langle \mu(t,x),x\rangle
	}{1+\norm{x}^2}
	\right] V(x)
	\\
	& = 
	\tfrac{p}{2} 
	\left[
	\tfrac{(p-2)}{(1+\norm{x}^2)^2}  
	\left(
	\sum_{k=1}^m 
	\left[
	\sum_{i=1}^d \sigma_{i,k}(t,x)x_i
	\right]^2
	\right)
	+ 
	\tfrac{\HSnorm{\sigma(t,x)}^2}{1+\norm{x}^2}
	+ 
	\tfrac{2\langle\mu(t,x),x\rangle}{1+\norm{x}^2}
	\right] V(x) . 
	\end{split}
	\end{equation}
Next note that for all 
	$t\in [0,T]$, 
	$x\in \cO$ 
it holds that 
	\begin{equation}
	\begin{split}
	\sum_{k=1}^m \left[\sum_{i=1}^d \sigma_{i,k}(t,x) x_i\right]^2 
	& 
	\leq 
	\sum_{k=1}^m 
	\left( \sum_{i=1}^d |\sigma_{i,k}(t,x)|^2 \right) 
	\left( \sum_{i=1}^d |x_i|^2 \right) 
	= 
	\HSnorm{\sigma(t,x)}^2
	\norm{x}^2
	\\
	& \leq 
	c ( 1 + \norm{x}^2 ) \norm{x}^2 
	\leq 
	c \left[ 1 + \norm{x}^2 \right]^2.
	\end{split} 
	\end{equation}
Combining this with \eqref{polynomials_lyapunov:some_calculation} shows that for all 
	$t\in [0,T]$, 
	$x\in\cO$ 
it holds that
	\begin{equation}
	\begin{split}
	& 
	\tfrac12 
	\operatorname{Trace}\!
	\big( 
	\sigma(t,x)[\sigma(t,x)]^{*}(\operatorname{Hess} V)(x)
	\big) 
	+ 
	\langle 
	\mu(t,x), 
	(\nabla V)(x) 
	\rangle
	\\
	& 
	\leq 
	\tfrac{p}{2} 
	\left[
	\max\{p-2,0\} c + 3c
	\right]  
	V(x) 
	= 
	\tfrac{pc}{2}
	\max\{p+1,3\} V(x)
	= \rho V(x) . 
	\end{split}
	\end{equation}
This establishes Item~\eqref{polynomials_lyapunov:item2}. 
The proof of \cref{polynomials_lyapunov} is thus completed.
\end{proof}    

\subsection{Locality properties for solutions of SDEs}

In this section we present two elementary results concerning the local 
behaviour of solutions to SDEs. These results, 
\cref{sde_compactly_supported,coinciding_until_stopping_times} below, 
are used in the proof of \cref{stochastic_continuity_lemma} 
(see \cref{subsec:SFEs_and_SDEs} below). \cref{sde_compactly_supported} 
asserts, loosely speaking, that a particle whose movements are governed 
by a  SDE with sufficiently regular 
coefficients is almost surely at rest when it finds itself in a region 
away from the supports of the coefficients. 

\begin{lemma}
	\label{sde_compactly_supported}
Let 
	$d,m\in\N$, 
	$T\in (0,\infty)$, 
let 
	$\norm{\cdot}\colon\R^d\to [0,\infty)$ be the standard norm on $\R^d$, 
let 
	$\HSnorm{\cdot}\colon \R^{d\times m}\to [0,\infty)$ 
	be the Frobenius norm on $\R^{d\times m}$, 
let 
	$\mu\in C([0,T]\times\R^d,\R^d)$, 
	$\sigma\in C([0,T]\times\R^d,\R^{d\times m})$ 
satisfy for all 
	$r\in (0,\infty)$
that 
	\begin{equation}
	\label{sde_compactly_supported:local_lipschitz_ass}
	\sup_{t\in [0,T]}
	\sup_{\substack{
			x,y\in\R^d, \\
			x\neq y, \\
			\norm{x},\norm{y}\leq r} }
	\left[ 
	\frac{
		\norm{ 
			\mu(t,x) - \mu(t,y) 
		} 
		+ 
		\HSnorm{
			\sigma(t,x) 
			- 
			\sigma(t,y) } 
	}{\norm{ x - y }}
	\right]
	< \infty, 
	\end{equation}
let 
	$\cO \subseteq  \R^d$ be an open set which satisfies $\operatorname{supp}(\mu)\cup \operatorname{supp}(\sigma) \subseteq  [0,T]\times \cO$,   
let 
	$(\Omega,\mathcal{F},\P,(\mathbb{F}_t)_{t\in [0,T]})$ be a filtered probability space which satisfies the usual conditions, 
let 
	$W\colon [0,T]\times\Omega\to\R^m$ be a standard $(\mathbb{F}_t)_{t\in [0,T]}$-Brownian motion, 
and let 
	$X \colon [0,T]\times\Omega\to\R^d$ 
be an $(\mathbb{F}_t)_{t\in [0,T]}$-adapted stochastic process with continuous sample paths which satisfies that for all 
	$t\in [0,T]$ 
it holds $\P$-a.s.~that 
	\begin{equation} \label{sde_compactly_supported:sde_ass}
	X_t 
	= 
	X_0 
	+ 
	\int_0^t 
	\mu(s,X_s)\,ds 
	+ 
	\int_0^t 
	\sigma(s,X_s) \,dW_s. 
	\end{equation}
Then 
	\begin{enumerate}[(i)]
		\item
		\label{sde_compactly_supported:item1} 
		it holds that 
		$
		\Big[
		\big(
		\P(X_0 \notin \cO) = 1
		\big)
		\Rightarrow 
		\big( 
		\P(\Forall t\in [0,T]\colon X_t = X_0) = 1
		\big)
		\Big]
		$
		and
		\item
		\label{sde_compactly_supported:item2} 
		it holds that 
		$
		\Big[ 
		\big( 
		\P(X_0 \in \cO) = 1
		\big)
		\Rightarrow 
		\big(
		\P(\Forall t\in [0,T]\colon 
		X_t \in \overline{\cO}) = 1
		\big)
		\Big]
		$. 
	\end{enumerate}
\end{lemma}

\begin{proof}[Proof of \cref{sde_compactly_supported}]
	We first prove Item~\eqref{sde_compactly_supported:item1}. 
	For this we assume that $\P(X_0 \notin \cO) = 1$. Observe that this implies 
	$
	\P(\Forall t\in [0,T]\colon \norm{\mu(t,X_0)} 
	+ 
	\HSnorm{\sigma(t,X_0)} 
	= 0
	)
	= 
	1
	$. 
Therefore, we obtain that
	\begin{equation}
	\label{sde_compactly_supported_definition_of_Y}
	Y = ( [0,T]\times\Omega \ni (t,\omega)\mapsto X_0(\omega) \in \R^d )
	\end{equation}
is an $(\mathbb{F }_t)_{t\in [0,T]}$-adapted stochastic process with continuous sample paths which satisfies that for all $t\in [0,T]$ it holds $\P$-a.s.~that 
	\begin{equation}
	\begin{split}
	Y_t 
	& = 
	X_0
	= 
	X_0 
	+ 
	\int_0^t 0 \,ds 
	+ 
	\int_0^t 0 \,dW_s 
	= 
	X_0 
	+ 
	\int_0^t \mu(s,X_0)\,ds 
	+ 
	\int_0^t \sigma(s,X_0)\,dW_s
	\\
	& = 
	X_0 
	+ 
	\int_0^t \mu(s,Y_s)\,ds 
	+ 
	\int_0^t \sigma(s,Y_s)\,dW_s . 
	\end{split}
	\end{equation}
Karatzas \& Shreve~\cite[Theorem 5.2.5]{KaSh1991_BrownianMotionAndStochasticCalculus} and \eqref{sde_compactly_supported:local_lipschitz_ass}--\eqref{sde_compactly_supported_definition_of_Y} hence assure that 
	\begin{equation}
	\P\left( 
	\Forall t\in [0,T]\colon
	\,X_t = X_0 
	\right) 
	=
	\P\left( 
	\Forall t\in [0,T]\colon\,X_t = Y_t 
	\right) 
	= 1. 
	\end{equation}
This establishes Item~\eqref{sde_compactly_supported:item1}. 
Next we prove Item~\eqref{sde_compactly_supported:item2}. For this we assume that 
	$\P(X_0\in\cO)=1$ 
	and let 
	$\tau\colon\Omega\to [0,T]$ 
	satisfy  
	$
	\tau
	=
	\inf( 
	\{t\in [0,T]\colon~X_t \notin \overline{\cO} \}
	\cup
	\{T\}
	)
	$.  
Note that $\tau$ is an $(\mathbb{F}_t)_{t\in [0,T]}$-stopping time. Let $Y\colon [0,T]\times\Omega\to\R^d$ satisfy for all 
	$t\in [0,T]$, 
	$\omega\in\Omega$ 
that 
	$ Y_t(\omega) = X_{\min\{t,\tau(\omega)\}}(\omega)$. 
Observe that $Y\colon [0,T]\times\Omega\to\R^d$ is an $(\mathbb{F}_t)_{t\in [0,T]}$-adapted stochastic process with continuous sample paths. Moreover, note that for all 
	$t\in [0,T]$ 
it holds $\P$-a.s.~that 
	\begin{equation}
	\begin{split}
	Y_t 
	& = 
	X_{\min\{t,\tau\}} 
	= 
	X_0 
	+ 
	\int_0^{\min\{t,\tau\}} 
	\mu(s,X_s)\,ds 
	+ 
	\int_0^{\min\{t,\tau\}}
	\sigma(s,X_s)\,dW_s \\
	& = 
	X_0 
	+ 
	\int_0^t \mathbbm{1}_{\{0<s\leq \tau\}} 
	\mu(s,X_s)\,ds 
	+ 
	\int_0^t \mathbbm{1}_{\{0<s\leq \tau\}}
	\sigma(s,X_s)\,dW_s
	.
	\end{split} 
	\end{equation}
Combining this with the fact that for all 
	$t\in [0,T]$ 
it holds that 
	$
	\mathbbm{1}_{\{t\leq\tau\}} X_t 
	=
	\mathbbm{1}_{\{t\leq\tau\}} Y_t 
	$
and 
	$
	\mathbbm{1}_{\{\tau < t\}} 
	[ 
	\norm{ \mu(t,Y_{t}) } 
	+ 
	\HSnorm{ \sigma(t,Y_{t})} 
	] 
	= 0$ 
we obtain that for all 
	$t\in [0,T]$ 
it holds $\P$-a.s.~that 
	\begin{equation}
	\begin{split}
	Y_t 
	& = 
	X_0 + \int_0^t \mu(s,Y_{s}) \,ds
	+ 
	\int_0^t \sigma(s,Y_{s})\,dW_s . 
	\end{split}
	\end{equation}
Karatzas \& Shreve~\cite[Theorem 5.2.5]{KaSh1991_BrownianMotionAndStochasticCalculus}, 
\eqref{sde_compactly_supported:local_lipschitz_ass}, 
and 
\eqref{sde_compactly_supported:sde_ass} 
hence demonstrate that 
	\begin{equation}
	\P\left( \Forall t\in [0,T]\colon~X_t=Y_t\right) 
	= 
	1 . 
	\end{equation}
	This establishes Item~\eqref{sde_compactly_supported:item2}. The proof of \cref{sde_compactly_supported} is thus completed. 
\end{proof}

The next result, \cref{coinciding_until_stopping_times} below, 
basically asserts that the solutions of SDEs 
coincide as long as the trajectories stay in a domain in which the drift 
and diffusion coefficients are the same. 

\begin{lemma}
	\label{coinciding_until_stopping_times}
Let 
	$d,m\in\N$, 
	$T\in (0,\infty)$, 
let 
	$\norm{\cdot}\colon\R^d\to [0,\infty)$ 
be the standard norm on $\R^d$, 
let 
	$\HSnorm{\cdot}\colon\R^{d\times m}\to [0,\infty)$ 
be the Frobenius norm on $\R^{d\times m}$, 
let 
	$\cO \subseteq  \R^d$ 
be an open set, 
for every 
	$r\in (0,\infty)$ 
let 
	$O_r\subseteq\cO$ 
satisfy 
	$O_r = \{x\in\cO\colon \norm{x}\leq r~\text{and}~\{y\in\R^d\colon\norm{y-x}<\nicefrac{1}{r}\}\subseteq\cO \}$, 
let 
	$\cC \subseteq [0,T]\times\R^d$ 
be a closed set which satisfies 
	$\cC \subseteq [0,T]\times\cO$, 
let    
	$\mu_1,\mu_2 \in C([0,T]\times\cO,\R^d)$,
	$\sigma_1, \sigma_2
	\in C([0,T]\times\cO,\R^{d\times m})$ 
satisfy for all
	$r\in (0,\infty)$ 
that
	$\mu_1|_{\cC}=\mu_2|_{\cC}$, 
	$\sigma_1|_{\cC}=\sigma_2|_{\cC}$, 
	and 
	\begin{equation}
	\sup\!\left(
	\left\{ 
	\tfrac{
		\norm{\mu_1(t,x) - \mu_1(t,y)} 
		+ 
		\HSnorm{\sigma_1(t,x) - \sigma_1(t,y)} 
	}{\norm{x-y}}\colon 
	t\in [0,T], 
	x,y\in O_r, 
	x\neq y
	\right\}
	\cup \{T\} 
	\right)
	< \infty
	,
	\end{equation} 		
let 
	$(\Omega,\cF,\P,(\F_t)_{t\in [0,T]})$ 
be a filtered probability space which satisfies the usual conditions, 
let 
	$W\colon [0,T]\times\Omega\to\R^m$ 
be a standard $(\F_t)_{t\in [0,T]}$-Brownian motion, 
let 
	$X^{(i)}=(X^{(i)}_t)_{t\in [0,T]}\colon [0,T]\times\Omega\to\cO$, 
	$i\in\{1,2\}$,
be $(\F_t)_{t\in [0,T]}$-adapted stochastic processes with continuous 
sample paths which satisfy that for every 
	$i\in\{1,2\}$, 
	$t\in [0,T]$ 
it holds $\P$-a.s.~that 
	\begin{equation}
	\label{coinciding_until_stopping_times:sde_ass}
	X^{(i)}_t = X^{(i)}_0 
	+ \int_0^t \mu_i(s,X^{(i)}_s)\,ds 
	+ \int_0^t \sigma_i(s,X^{(i)}_s)\,dW_s,  
	\end{equation}
	assume that 
	$X^{(1)}_0=X^{(2)}_0$, 
	and let 
	$\tau \colon \Omega\to [0,T]$ 
	satisfy  
	$
	\tau = \inf(\{t\in [0,T]\colon (t,X^{(1)}_t) \notin\cC~\text{or}~(t,X^{(2)}_t)\notin\cC \}\cup\{T\})
	$.
Then it holds that 
	\begin{equation}
	\label{coinciding_until_stopping_times:claim}
	\P\left( 
	\Forall t\in [0,T]\colon 
	\mathbbm{1}_{\{t \leq \tau\}}
	\Normm{
		X^{(1)}_t - X^{(2)}_t } = 0
	\right) = 1 . 
	\end{equation}
\end{lemma}

\begin{proof}[Proof of  \cref{coinciding_until_stopping_times}] 
Throughout this proof let 
	$\rho_n\colon \Omega \to [0,T]$, 
	$n\in\N$, 
satisfy for all 
	$n\in\N$ 
that 
	$
	\rho_n 
	= 
	\inf\!\big(
	\{ 
	t\in [0,T]\colon 
	X^{(1)}_t \in \cO \setminus O_n~\text{or}~X^{(2)}_t \in \cO \setminus O_n 
	\}
	\cup \{T\} 
	\big)
	$
and let 
	$L_n\in [0,\infty)$, $n\in\N$, 
be real numbers which satisfy for all 
	$t\in [0,T]$, 
	$x,y\in O_n$ 
that 
	\begin{equation}
	\label{coinciding_until_stopping_times:local_lipschitz_constants}
	\norm{\mu_1(t,x) - \mu_1(t,y)} 
	+ 
	\HSnorm{\sigma_1(t,x) - \sigma_1(t,y)} 
	\leq L_n \norm{x-y}  
	. 
	\end{equation}
Observe that for all 
	$n\in\N$ 
it holds that $\tau$ and $\rho_n$ are $(\F_t)_{t\in [0,T]}$-stopping times. 
Moreover, note that for every $K \subseteq \cO$ compact there exists $n \in \N$ such that $K \subseteq O_n$. This and the fact that $X^{(1)}$ and $X^{(2)}$ have continuous sample paths ensure that for all 
	$\omega \in \Omega$ 
there exists 
	$n\in\N$ 
such that for all 
	$k\in\N$ 
with 
	$k\geq n$ 
it holds that 
	$\rho_k(\omega) = T$. 
Next note that the assumption that 
	$X^{(1)}$ 
and 
	$X^{(2)}$ 
have continuous sample paths and the fact that 
	$O_n$, $n\in\N$, 
are compact imply that for all 
	$n \in \N$, 
	$\omega \in \{\rho_n > 0\}$, 
	$i \in \{1,2\}$
it holds that 
	$X^{(i)}_{\rho_n(\omega)}(\omega) \in O_n$. 
Combining this with the assumption that 
	$X^{(1)}_0 = X^{(2)}_0$
assures that for all 
	$n\in\N$, 
	$t\in [0,T]$ 
it holds that 
\begin{equation} 
	\Normm{X^{(1)}_{\min\{t,\tau,\rho_n\}}-X^{(2)}_{\min\{t,\tau,\rho_n\}}}
	\leq 2n
	.
\end{equation} 
This ensures for every 
	$n\in\N$ 
that 
	\begin{equation}
	\label{coinciding_until_stopping_times:l2_integrability_of_X1_and_X2}
	\sup_{t\in [0,T]} 
	\left( 
	\Exp{
		\norm{
			X^{(1)}_{\min\{t,\tau,\rho_n\}}
			- 
			X^{(2)}_{\min\{t,\tau,\rho_n\}}
		}^2
	} 
	\right)  
	< 
	\infty. 
	\end{equation}
Next note that the fact that for all 
	$s \in (0,T]$ 
it holds that 
	$
	\mathbbm{1}_{\{s\leq\tau\}}
	\big[ 
	\Norm{\mu_1(s,X^{(2)}_s)-\mu_2(s,X^{(2)}_s)}
	+ 
	\HSNorm{\sigma_1(s,X^{(2)}_s)-\sigma_2(s,X^{(2)}_s)}
	\big]=0
	$, 
the assumption that 
	$ X^{(1)}_0 = X^{(2)}_0 $, 
and
	\eqref{coinciding_until_stopping_times:sde_ass}
ensure that for all 
	$n\in\N$,
	$t\in [0,T]$ 
it holds $\P$-a.s.~that 
	\begin{equation}
	\begin{split}
	X^{(1)}_{\min\{t,\tau,\rho_n\}} - X^{(2)}_{\min\{t,\tau,\rho_n\}} 
	& 	= 
	\int_0^{\min\{t,\tau,\rho_n\}}
	\left[ 
	\mu_1(s,X^{(1)}_{s}) 
	- 
	\mu_2(s,X^{(2)}_{s})
	\right]\! 
	\,ds
	\\
	& \qquad 
	+ 
	\int_0^{\min\{t,\tau,\rho_n\}}
	\left[ 
	\sigma_1(s,X^{(1)}_{s})
	- 
	\sigma_2(s,X^{(2)}_{s})
	\right]\!
	\,dW_s 
	\\
	& 
	=
	\int_0^t \mathbbm{1}_{\{0<s\leq\min\{\tau,\rho_n\}\}} 
	\left[ 
	\mu_1(s,X^{(1)}_{s}) 
	- 
	\mu_2(s,X^{(2)}_{s})
	\right]\!
	\,ds
	\\
	& \qquad 
	+ 
	\int_0^{t}
	\mathbbm{1}_{\{0<s\leq\min\{\tau,\rho_n\}\}} 
	\left[ 
	\sigma_1(s,X^{(1)}_{s})
	- 
	\sigma_2(s,X^{(2)}_{s})
	\right]\! 
	\,dW_s 
	\\
	& 
	= 
	\int_0^t \mathbbm{1}_{\{0<s\leq\min\{\tau,\rho_n\}\}} 
	\left[ 
	\mu_1(s,X^{(1)}_{s}) 
	- 
	\mu_1(s,X^{(2)}_{s})
	\right]\!
	\,ds
	\\
	& \qquad 
	+ 
	\int_0^{t}
	\mathbbm{1}_{\{0<s\leq\min\{\tau,\rho_n\}\}} 
	\left[ 
	\sigma_1(s,X^{(1)}_{s})
	- 
	\sigma_1(s,X^{(2)}_{s})
	\right]\! 
	\,dW_s 
	. 
	\end{split}
	\end{equation}
This implies that for all 
	$n\in\N$,
	$t\in [0,T]$ 
it holds $\P$-a.s.~that 
	\begin{equation}
	\begin{split} 
	& X^{(1)}_{\min\{t,\tau,\rho_n\}} - X^{(2)}_{\min\{t,\tau,\rho_n\}}
	\\ 
	& \qquad  
	= 
	\int_0^t \mathbbm{1}_{\{0<s\leq \min\{\tau,\rho_n\} \}} 
	\left[ 
	\mu_1\big(s,X^{(1)}_{\min\{s,\tau,\rho_n\}}\big) 
	- 
	\mu_1\big(s,X^{(2)}_{\min\{s,\tau,\rho_n\}}\big)
	\right]\! 
	\,ds
	\\
	& \qquad \qquad	+ 
	\int_0^{t}
	\mathbbm{1}_{\{0<s\leq \min\{\tau,\rho_n\} \}} 
	\left[ 
	\sigma_1\big(s,X^{(1)}_{\min\{s,\tau,\rho_n\}}\big)
	- 
	\sigma_1\big(s,X^{(2)}_{\min\{s,\tau,\rho_n\}}\big)
	\right]\! 
	\,dW_s 
	. 
	\end{split}
	\end{equation}
Minkowski's inequality, It\^o's isometry, and  \eqref{coinciding_until_stopping_times:local_lipschitz_constants}--\eqref{coinciding_until_stopping_times:l2_integrability_of_X1_and_X2} hence yield for all
	$n\in\N$,
	$t\in [0,T]$ 
that 
	\begin{equation}
	\begin{split}
	& \Exp{ 
		\norm{
			X^{(1)}_{\min\{t,\tau,\rho_n\}} 
			- 
			X^{(2)}_{\min\{t,\tau,\rho_n\}}
		}^2
	}^{\nicefrac12}
	\\
	& \qquad \leq 
	\int_0^t 
	\Exp{\mathbbm{1}_{
			\{ 0<s\leq \min\{\tau,\rho_n\} \}
		}
		\norm{\mu_1(s,X^{(1)}_{\min\{s,\tau,\rho_n\}}) 
		- 
		\mu_1(s,X^{(2)}_{\min\{s,\tau,\rho_n\}})}^2
	}^{\nicefrac12}\,ds
	\\
	& \qquad \qquad  
	+ 
	\Exp{\norm{\int_0^{t}
			\mathbbm{1}_{
				\{ 0<s\leq \min\{\tau,\rho_n\} \}
			} 
			\left[ 
			\sigma_1(s,X^{(1)}_{\min\{s,\tau,\rho_n\}})
			- 
			\sigma_1(s,X^{(2)}_{\min\{s,\tau,\rho_n\}})
			\right]\! 
			\,dW_s }^2}^{\nicefrac12} 
	\\
	& \qquad \leq 
	\int_0^t 
	\Exp{
		\norm{\mu_1(s,X^{(1)}_{\min\{s,\tau,\rho_n\}}) 
		- 
		\mu_1(s,X^{(2)}_{\min\{s,\tau,\rho_n\}})}^2
	}^{\nicefrac12}\,ds
	\\
	& \qquad \qquad
	+ 
	\left[\int_0^t \Exp{\HSnorm{\sigma_1(s,X^{(1)}_{\min\{s,\tau,\rho_n\}})
			- 
			\sigma_1(s,X^{(2)}_{\min\{s,\tau,\rho_n\}})}^2}\!\,ds
	\right]^{\nicefrac12} 
	\\
	& \qquad \leq 
	L_n
	\int_0^t \Exp{\norm{ 
			X^{(1)}_{\min\{s,\tau,\rho_n\}} 
			- 
			X^{(2)}_{\min\{s,\tau,\rho_n\}}
		}^2
	}^{\nicefrac12} \,ds 
	\\
	& \qquad \qquad + 
	L_n
	\left[ \int_0^t 
	\Exp{\norm{
			X^{(1)}_{\min\{s,\tau,\rho_n\}} 
			- 
			X^{(2)}_{\min\{s,\tau,\rho_n\}}}^2
	}\!\,ds
	\right]^{\nicefrac12}. 
	\end{split}
	\end{equation}
The fact that for all 
	$a,b\in [0,\infty)$ 
it holds that 
	$(a+b)^2 \leq 2a^2+2b^2$ 
and H\"older's inequality hence demonstrate for all 
	$n\in\N$,
	$t\in [0,T]$ 
that 
	\begin{equation}
	\Exp{ 
		\norm{
			X^{(1)}_{\min\{t,\tau,\rho_n\}}
			- 
			X^{(2)}_{\min\{t,\tau,\rho_n\}}
		}^2
	}
	\leq 
	2(L_n)^2(T+1) 
	\int_0^t \Exp{\norm{
			X^{(1)}_{\min\{s,\tau,\rho_n\}} 
			- 
			X^{(2)}_{\min\{s,\tau,\rho_n\}}}^2
	}\,ds . 
	\end{equation}
Combining this with Gronwall's inequality and \eqref{coinciding_until_stopping_times:l2_integrability_of_X1_and_X2} implies for all 
	$n\in\N$,
	$t\in [0,T]$ 
that 
	\begin{equation}
	\Exp{ 
		\norm{
			X^{(1)}_{\min\{t,\tau,\rho_n\}} 
			- 
			X^{(2)}_{\min\{t,\tau,\rho_n\}} 
		}^2
	}
	= 0.
	\end{equation}
The fact that $X^{(1)}$ and $X^{(2)}$ have continuous sample paths hence  ensures for all 
	$n\in\N$
that 
	\begin{equation}
	\P\!\left( \Forall t\in [0,T]\colon \mathbbm{1}_{\{ t\leq \min\{\tau,\rho_n\} \}} 
	\Normm{X^{(1)}_{t} - X^{(2)}_{t}} = 0\right) = 1. 
	\end{equation}
Therefore we obtain that 
	\begin{equation}
	\P\!\left( \Forall n\in\N~ \Forall t\in [0,T]\colon \mathbbm{1}_{\{ t\leq \min\{\tau,\rho_n\} \}} 
	\Normm{X^{(1)}_{t} - X^{(2)}_t} = 0\right) = 1. 
	\end{equation}
This implies that 
	\begin{equation}
	\P\!\left( \Forall t\in [0,T]\colon \mathbbm{1}_{\{ t \leq \tau \} } 
	\Normm{X^{(1)}_{t} - X^{(2)}_{t}} = 0\right) = 1. 
	\end{equation}
This establishes \eqref{coinciding_until_stopping_times:claim}. The proof of \cref{coinciding_until_stopping_times} is thus completed.
\end{proof}

\subsection{Continuity properties for solutions of SDEs}
\label{subsec:continuous_dependence}

The well-known \cref{stability_for_sdes} below (cf.~also Stroock~\cite[Theorem I.2.2]{Stroock1982LecturesOnSDEs}) estimates the difference between two solutions to the same SDE that start at different times and different places. \cref{stability_for_sdes} is a crucial ingredient in the proof of \cref{stochastic_continuity_lemma}, where it is used to show that the solution to an auxiliary SDE evaluated at a certain time is stochastically continuous as a function of the initial values. 
\begin{lemma} \label{stability_for_sdes}
Let 
	$d,m\in\N$, 
	$L,T\in (0,\infty)$, 
let 
	$\norm{\cdot}\colon\R^d\to [0,\infty)$ 
be the standard norm on $\R^d$, 
let 
	$\HSnorm{\cdot}\colon\R^{d\times m}\to [0,\infty)$ 
be the Frobenius norm on 
	$\R^{d\times m}$,    
let 
	$\cO \subseteq  \R^d$ be a non-empty open set, 
let 
	$\mu\in C([0,T]\times\cO,\R^d)$, 
	$\sigma\in C([0,T]\times\cO,\R^{d\times m})$ 
be compactly supported functions which satisfy for all 
	$t\in [0,T]$, 
	$x,y\in \cO$
that 
	\begin{equation}
	\label{stability_for_sdes:lipschitz_constant}
	\norm{
		\mu(t,x) 
		- 
		\mu(t,y) } 
	+ 
	\HSnorm{ 
		\sigma(t,x) 
		- 
		\sigma(t,y)} 
	\leq L \norm{ x - y }
	, 
	\end{equation} 
let 
	$(\Omega,\mathcal{F},\P,(\mathbb{F}_t)_{t\in [0,T]})$ be a 
filtered probability space which satisfies the usual conditions, 
let 
	$W\colon [0,T]\times\Omega\to\R^m$ be a standard $(\mathbb{F}_t)_{t\in [0,T]}$-Brownian motion, 
and for every 
	$t\in [0,T]$, 
	$x\in \cO$ 
let 
	$X^{t,x} = (X^{t,x}_{s})_{s\in [t,T]}\colon [t,T]\times\Omega\to\cO$ 
be an $(\mathbb{F}_s)_{s\in [t,T]}$-adapted stochastic process with continuous sample paths which satisfies that for all 
	$s\in [t,T]$ 
it holds $\P$-a.s.~that 
	\begin{equation}
	\label{stability_for_sdes:sde}
	X^{t,x}_{s} 
	= 
	x 
	+ 
	\int_t^s 
	\mu(r, X^{t,x}_{r} ) \,dr
	+ 
	\int_t^s \sigma(r, X^{t,x}_{r} ) \,dW_r . 
	\end{equation}
Then it holds for all 
	$t\in [0,T]$, 
	$\mathfrak{t}\in [t,T]$, 
	$s\in [\mathfrak{t},T]$, 
	$x,\mathfrak{x}\in \cO$
that 
	\begin{equation}
	\label{stability_for_sdes:claim}
	\begin{split}
	&
	\Exp{\norm{X^{t,x}_{s} - X^{\mathfrak{t},\mathfrak{x}}_{s}}^2}
	\\
	& \leq 
	9 
	\left[ 
	\norm{x-\mathfrak{x}} + |t-\mathfrak{t}|^{\nicefrac{1}{2}} 
	\right]^2 
	\left[ 1 + \sqrt{T} \sup_{r\in [0,T]} \norm{\mu(r,\mathfrak{x})} 
	+ \sup_{r\in [0,T]} \HSnorm{\sigma(r,\mathfrak{x})} \right]^2 
	\exp\!\left(6L^2T(T+1)\right).
	\end{split} 
	\end{equation} 
\end{lemma}

\begin{proof}[Proof of 
	\cref{stability_for_sdes}]
Throughout this proof let 
	$\mathfrak{m}\colon [0,T]\times\R^d \to \R^d$ 
and 
	$\mathfrak{s}\colon [0,T]\times\R^d \to \R^{d\times m}$ 
satisfy for all 
	$t\in [0,T]$, 
	$x\in \R^d$ 
that 
	\begin{equation}
	\mathfrak{m}(t,x) 
	= 
	\begin{cases} 
	\mu(t,x) & \colon x\in\cO \\
	0 		 & \colon x\in\R^d\setminus\cO
	\end{cases} 
	\qandq 
	\mathfrak{s}(t,x)
	= 
	\begin{cases} 
	\sigma(t,x) & \colon x\in\cO \\
	0 		    & \colon x\in\R^d\setminus\cO. 
	\end{cases}
	\end{equation}
Observe that \eqref{stability_for_sdes:lipschitz_constant}
ensures that 
	$\mathfrak{m} \colon [0,T]\times\R^d \to \R^d$ 
and 
	$\mathfrak{s}\colon [0,T]\times\R^d \to \R^{d\times m}$ 
are compactly supported continuous functions which satisfy for all 
	$ t \in [0,T] $, 
	$ x,y \in \R^d $ 
that 
	\begin{equation}
	\norm{\mathfrak{m}(t,x)-\mathfrak{m}(t,y)} 
	+ 
	\HSNorm{\mathfrak{s}(t,x)-\mathfrak{s}(t,y)} 
	\leq 
	L \norm{x-y} . 
	\end{equation}
Karatzas \& Shreve~\cite[Theorem 
5.2.9]{KaSh1991_BrownianMotionAndStochasticCalculus} hence guarantees 
for every 
	$t\in [0,T]$, 
	$x\in \cO$
that there exists an $(\F_s)_{s\in [t,T]}$-adapted stochastic process 
	$
	\mathfrak{X}^{t,x}=(\mathfrak{X}^{t,x}_s)_{s\in [t,T]}\colon [t,T]\times\Omega \to \R^d 
	$
with continuous sample paths 
\begin{enumerate}[(A)] 
	\item 
which satisfies that 
	$ \sup_{s\in [t,T]} \Exp{\Norm{\mathfrak{X}^{t,x}_s}^2} < \infty $
and 
	\item 
which satisfies that for all 
	$
	s\in [t,T]
	$
it holds $\P$-a.s.~that 
	\begin{equation} 
	\mathfrak{X}^{t,x}_s 
	= 
	x 
	+ 
	\int_t^s 
	\mathfrak{m}(r,\mathfrak{X}^{t,x}_r)\,dr 
	+ 
	\int_t^s 
	\mathfrak{s}(r,\mathfrak{X}^{t,x}_r)\,dW_r. 
	\end{equation} 
\end{enumerate}
This, Karatzas \& Shreve~\cite[Theorem 
5.2.5]{KaSh1991_BrownianMotionAndStochasticCalculus}, and 
\eqref{stability_for_sdes:sde} ensure that for all 
	$t\in [0,T]$, 
	$x\in \cO$ 
it holds that 
	\begin{equation} 
	\P\!\left( 
	\forall s\in[t,T]\colon X^{t,x}_s = \mathfrak{X}^{t,x}_s 
	\right) = 1. 
	\end{equation}
Combining this with the fact that for all 
	$ t \in [0,T] $, 
	$ x \in \cO $ 
it holds that 	
	$ \sup_{s\in [t,T]} \EXP{\Norm{\mathfrak{X}^{t,x}_s}^2} < \infty $
implies that for all 
	$ t \in [0,T] $, 
	$ x \in \cO $
it holds that 
	\begin{equation} \label{stability_for_sdes:l2_integrability}
	\sup_{s\in [t,T]}
	\Exp{\norm{X^{t,x}_s}^2} 
	< \infty . 
	\end{equation}
Next note that \eqref{stability_for_sdes:sde} ensures that for all 
	$t\in [0,T]$, 
	$\mathfrak{t}\in [t,T]$, 
	$s\in [\mathfrak{t},T]$, 
	$x,\mathfrak{x}\in \cO$ 
it holds $\P$-a.s.~that
	\begin{equation}
	\label{stability_for_sdes:difference_equation}
	\begin{split}
	X^{t,x}_{s}
	-
	X^{\mathfrak{t},\mathfrak{x}}_{s} 
	&
	= 
	X^{t,x}_{\mathfrak{t}} - \mathfrak{x} 
	+ 
	\int_{\mathfrak{t}}^s 
	\left( 
	\mu( r, X^{t,x}_{r} ) 
	-
	\mu( r, X^{\mathfrak{t},\mathfrak{x}}_{r} ) 
	\right)\! \,dr
	+ 
	\int_{\mathfrak{t}}^s 
	\left( 
	\sigma(r, X^{t,x}_r ) 
	- 
	\sigma(r, X^{\mathfrak{t},\mathfrak{x}}_r ) 
	\right)\!\,dW_r . 
	\end{split}
	\end{equation}
Minkowski's inequality hence yields that for all 
	$t\in [0,T]$, 
	$\mathfrak{t}\in [t,T]$, 
	$s\in [\mathfrak{t},T]$, 
	$x,\mathfrak{x}\in\cO$ 
it holds that 
	\begin{equation} 
	\begin{split}
	\left|
	\Exp{\norm{X^{t,x}_{s}
			-
			X^{\mathfrak{t},\mathfrak{x}}_{s}}^2}
	\right|^{\nicefrac12}
	\leq 
	\left|
	\Exp{\norm{X^{t,x}_{\mathfrak{t}} - \mathfrak{x}}^2}
	\right|^{\nicefrac12}
	& + 
	\int_{\mathfrak{t}}^s 
	\left|\Exp{
		\norm{\mu( r, X^{t,x}_{r} ) 
			-
			\mu( r, X^{\mathfrak{t},\mathfrak{x}}_{r} )}^2}
	\right|^{\nicefrac12}
	\,dr
	\\
	& + 
	\left|\Exp{\norm{
			\int_{\mathfrak{t}}^s 
			\left( 
			\sigma(r, X^{t,x}_r ) 
			- 
			\sigma(r, X^{\mathfrak{t},\mathfrak{x}}_r ) 
			\right)\!\,dW_r}^2}\right|^{\nicefrac12} .
	\end{split}
	\end{equation}
It\^o's isometry and \eqref{stability_for_sdes:lipschitz_constant} therefore ensure that for all 
	$t\in [0,T]$, 
	$\mathfrak{t}\in [t,T]$, 
	$s\in [\mathfrak{t},T]$, 
	$x,\mathfrak{x}\in\cO$ 
it holds that  
	\begin{equation}
	\begin{split}
	\left|
	\Exp{\norm{X^{t,x}_{s}
			-
			X^{\mathfrak{t},\mathfrak{x}}_{s}}^2}
	\right|^{\nicefrac12}
	\leq 
	\left|
	\Exp{\norm{X^{t,x}_{\mathfrak{t}} - \mathfrak{x}}^2}
	\right|^{\nicefrac12}
	& + 
	L\int_{\mathfrak{t}}^s 
	\left|\Exp{
		\norm{X^{t,x}_{r} - X^{\mathfrak{t},\mathfrak{x}}_{r} }^2}
	\right|^{\nicefrac12}
	\,dr
	\\
	& + 
	\left|
	\int_{\mathfrak{t}}^s
	\Exp{\HSnorm{ 
			\left( 
			\sigma(r, X^{t,x}_r ) 
			- 
			\sigma(r, X^{\mathfrak{t},\mathfrak{x}}_r ) 
			\right)}^2}\!\,dr
	\right|^{\nicefrac12}.
	\end{split}
	\end{equation}
This and 
\eqref{stability_for_sdes:lipschitz_constant} imply 
for all 
	$t\in [0,T]$, 
	$\mathfrak{t}\in [t,T]$, 
	$s\in [\mathfrak{t},T]$, 
	$x,\mathfrak{x}\in \cO$ 
that 
	\begin{equation} 
	\begin{split} 
	& 
	\left|
	\Exp{\norm{X^{t,x}_{s}
			-
			X^{\mathfrak{t},\mathfrak{x}}_{s}}^2}
	\right|^{\nicefrac12}
	\\
	& \leq 
	\left|
	\Exp{\norm{X^{t,x}_{\mathfrak{t}} - \mathfrak{x}}^2}
	\right|^{\nicefrac12}
	+ 
	L\int_{\mathfrak{t}}^s 
	\left|\Exp{
		\norm{X^{t,x}_{r} - X^{\mathfrak{t},\mathfrak{x}}_{r} }^2}
	\right|^{\nicefrac12}
	\,dr
	+ 
	L 
	\left|
	\int_{\mathfrak{t}}^s
	\Exp{\norm{ 
			X^{t,x}_r - X^{\mathfrak{t},\mathfrak{x}}_r }^2}\!\,dr
	\right|^{\nicefrac12}.
	\end{split}
	\end{equation} 
The fact that for all 
	$a,b,c\in [0,\infty)$ 
it holds that 
	$(a+b+c)^2 \leq 3 (a^2+b^2+c^2)$ 
and H\"older's inequality therefore ensure that for all
	$t\in [0,T]$, 
	$\mathfrak{t}\in [t,T]$, 
	$s\in [\mathfrak{t},T]$, 
	$x,\mathfrak{x}\in \cO$ 
it holds that  
	\begin{equation} 
	\begin{split}
	& \Exp{\norm{X^{t,x}_{s}-X^{\mathfrak{t},\mathfrak{x}}_{s}}^2}
	\\
	& \leq 
	3
	\Exp{\norm{X^{t,x}_{\mathfrak{t}} - \mathfrak{x}}^2}
	+ 
	3L^2T\int_{\mathfrak{t}}^s 
	\Exp{
		\norm{X^{t,x}_{r} - X^{\mathfrak{t},\mathfrak{x}}_{r}}^2}\!\,dr
	+ 
	3L^2 
	\int_{\mathfrak{t}}^s
	\Exp{\norm{ 
			X^{t,x}_r - X^{\mathfrak{t},\mathfrak{x}}_r }^2}\!\,dr
	\\
	& = 
	3
	\Exp{\norm{X^{t,x}_{\mathfrak{t}} - \mathfrak{x}}^2}
	+ 
	3L^2(T+1)
	\int_{\mathfrak{t}}^s
	\Exp{\norm{ 
			X^{t,x}_r - X^{\mathfrak{t},\mathfrak{x}}_r }^2}\!\,dr
	.
	\end{split}
	\end{equation} 
Gronwall's inequality and \eqref{stability_for_sdes:l2_integrability} hence imply for all 
	$t\in [0,T]$, 
	$\mathfrak{t}\in [t,T]$, 
	$s\in [\mathfrak{t},T]$, 
	$x,\mathfrak{x}\in \cO$ 
that 
	\begin{equation}
	\label{stability_for_sdes:first_estimate}
	\Exp{\norm{X^{t,x}_{s}-X^{\mathfrak{t},\mathfrak{x}}_{s}}^2}
	\leq 
	3
	\Exp{\norm{X^{t,x}_{\mathfrak{t}} - \mathfrak{x}}^2}
	\exp\!\left(3L^2T(T+1)\right).
	\end{equation}
In the next step we observe that \eqref{stability_for_sdes:sde} guarantees that for all 
	$t\in [0,T]$, 
	$\mathfrak{t}\in [t,T]$, 
	$x,\mathfrak{x}\in\cO$ 
it holds $\P$-a.s.~that 
	\begin{equation}
	X^{t,x}_{\mathfrak{t}} - \mathfrak{x} 
	= 
	x - \mathfrak{x} 
	+ 
	\int_t^{\mathfrak{t}} 
	\mu(r, X^{t,x}_r ) \,dr
	+ 
	\int_t^{\mathfrak{t}}
	\sigma(r, X^{t,x}_r ) \,dW_r. 
	\end{equation}
Minkowski's inequality hence demonstrates that for all 
	$t\in [0,T]$, 
	$\mathfrak{t}\in [t,T]$, 
	$x,\mathfrak{x}\in \cO$ 
	it holds that 
	\begin{equation}
	\left|\Exp{\norm{X^{t,x}_{\mathfrak{t}} - \mathfrak{x}}^2}\right|^{\nicefrac12}
	\leq 
	\norm{x - \mathfrak{x}}
	+ 
	\int_t^{\mathfrak{t}} 
	\left|\Exp{\norm{\mu(r, X^{t,x}_r )}^2}\right|^{\nicefrac12} \,dr
	+ 
	\left|\Exp{\norm{\int_t^{\mathfrak{t}}
			\sigma(r, X^{t,x}_r ) \,dW_r}^2}\right|^{\nicefrac12}. 
	\end{equation} 
It\^o's isometry therefore implies that for all 
	$t\in [0,T]$, 
	$\mathfrak{t}\in [t,T]$, 
	$x,\mathfrak{x}\in \cO$ 
it holds that 
	\begin{equation}
	\left|\Exp{\norm{X^{t,x}_{\mathfrak{t}} - \mathfrak{x}}^2}\right|^{\nicefrac12}
	\leq 
	\norm{x - \mathfrak{x}}
	+ 
	\int_t^{\mathfrak{t}} 
	\left|\Exp{\norm{\mu(r, X^{t,x}_r )}^2}\right|^{\nicefrac12} \,dr
	+ 
	\left|\int_t^{\mathfrak{t}}\Exp{\HSnorm{
			\sigma(r, X^{t,x}_r ) }^2}\!\,dr\right|^{\nicefrac12}.
	\end{equation}
This, Minkowski's inequality, and \eqref{stability_for_sdes:lipschitz_constant} yield that for all 
	$t\in [0,T]$, 
	$\mathfrak{t}\in [t,T]$, 
	$x,\mathfrak{x}\in \cO$ 
it holds that 	
	\begin{equation} 
	\begin{split}
	\left|\Exp{\norm{X^{t,x}_{\mathfrak{t}} - \mathfrak{x}}^2}\right|^{\nicefrac12}
	& \leq 
	\norm{x - \mathfrak{x}}
	+ 
	\int_t^{\mathfrak{t}} 
	\norm{\mu(r,\mathfrak{x})}\,dr 
	+ 
	L
	\int_t^{\mathfrak{t}}
	\left|\Exp{\norm{X^{t,x}_r-\mathfrak{x}}^2}\right|^{\nicefrac12} \,dr
	\\
	& + 
	\left|
	\int_t^{\mathfrak{t}} \HSnorm{\sigma(r,\mathfrak{x})}^2\,dr
	\right|^{\nicefrac12}
	+ 
	L 
	\left|
	\int_t^{\mathfrak{t}}
		\Exp{
		\norm{X^{t,x}_r-\mathfrak{x}}^2}\!\,dr\right|^{\nicefrac12}.
	\end{split}
	\end{equation}
The fact that for all 
	$a,b,c\in [0,\infty)$
it holds that 
	$(a+b+c)^2 \leq 3(a^2+b^2+c^2)$
hence implies for all 
	$t\in [0,T]$, 
	$\mathfrak{t}\in [t,T]$, 
	$x,\mathfrak{x}\in\cO$ 
that
	\begin{equation} 
	\begin{split}
	& \Exp{\norm{X^{t,x}_{\mathfrak{t}} - \mathfrak{x}}^2}
	\\
	& \leq 
	3
	\left[ 
	\norm{x - \mathfrak{x}}
	+ 
	\int_t^{\mathfrak{t}} 
	\norm{\mu(r,\mathfrak{x})}\,dr 
	+ 
	\left|
	\int_t^{\mathfrak{t}} \HSnorm{\sigma(r,\mathfrak{x})}^2\,dr
	\right|^{\nicefrac12}
	\right]^2
	+ 
	3L^2(T+1)
	\int_t^{\mathfrak{t}}
	\Exp{\norm{X^{t,x}_r-\mathfrak{x}}^2}\!\,dr.   
	\end{split}
	\end{equation}
This demonstrates for all 
	$t\in [0,T]$, 
	$\mathfrak{t}\in [t,T]$, 
	$x,\mathfrak{x}\in \cO$ 
that 
	\begin{multline}
	\Exp{\norm{X^{t,x}_{\mathfrak{t}} - \mathfrak{x}}^2}
	\leq 
	3 \left[ 
	\norm{x-\mathfrak{x}} 
	+ 
	|\mathfrak{t}-t|
	\sup_{r\in [0,T]} \norm{\mu(r,\mathfrak{x})}
	+ 
	|\mathfrak{t}-t|^{\nicefrac{1}{2}}
	\sup_{r\in [0,T]} \HSnorm{\sigma(r,\mathfrak{x})}
	\right]^2
	\\+ 
	3L^2(T+1)
	\int_t^{\mathfrak{t}}
	\Exp{\norm{X^{t,x}_r-\mathfrak{x}}^2}\!\,dr .
	\end{multline}
Hence, we obtain for all 
	$t\in [0,T]$, 
	$\mathfrak{t}\in [t,T]$, 
	$x,\mathfrak{x}\in\cO$
that 
	\begin{multline} 
	\Exp{\norm{X^{t,x}_{\mathfrak{t}} - \mathfrak{x}}^2}
	\leq 
	3 
	\left[ 
	\norm{x-\mathfrak{x}} + |t-\mathfrak{t}|^{\nicefrac{1}{2}} 
	\right]^2 
	\left[ 1 + \sqrt{T} \sup_{r\in [0,T]} \norm{\mu(r,\mathfrak{x})} 
	+ \sup_{r\in [0,T]} \HSnorm{\sigma(r,\mathfrak{x})} \right]^2 
	\\
	+ 3L^2(T+1)
	\int_t^{\mathfrak{t}}
	\Exp{\norm{X^{t,x}_r-\mathfrak{x}}^2}\!\,dr .
	\end{multline} 
	Gronwall's inequality and 
\eqref{stability_for_sdes:l2_integrability} hence ensure that for all 
	$t\in [0,T]$, 
	$\mathfrak{t}\in [t,T]$, 
	$x,\mathfrak{x}\in\cO$ 
	it holds that 
	\begin{equation} 
	\begin{split}
	&
	\Exp{\norm{X^{t,x}_{\mathfrak{t}} - \mathfrak{x}}^2}
	\\
	& \leq 
	3 
	\left[ 
	\norm{x-\mathfrak{x}} + |t-\mathfrak{t}|^{\nicefrac{1}{2}} 
	\right]^2 
	\left[ 1 + \sqrt{T} \sup_{r\in [0,T]} \norm{\mu(r,\mathfrak{x})} 
	+ \sup_{r\in [0,T]} \HSnorm{\sigma(r,\mathfrak{x})} \right]^2 
	\exp\!\left(3L^2T(T+1)\right). 
	\end{split}
	\end{equation} 
Combining this with  \eqref{stability_for_sdes:first_estimate} demonstrates \eqref{stability_for_sdes:claim}. The proof of \cref{stability_for_sdes} is thus completed. 
\end{proof}

\subsection{Existence and uniqueness properties for solutions of SFPEs associated with SDEs}
\label{subsec:SFEs_and_SDEs}

In this section we provide the announced application of \cref{abstract_existence} (see \cref{existence_sde_setting} below). The next essentially well-known result, \cref{stochastic_continuity_lemma} below (cf., for example, Liu \& R\"ockner~\cite[Proposition 3.2.1]{LiuRoeckner2015SPDEs}), ascertains that the stochastic continuity hypothesis of \cref{abstract_existence} is satisfied in the setting of \cref{existence_sde_setting}. 

\begin{lemma}
\label{stochastic_continuity_lemma}
Let 
	$d,m\in\N$, 
	$T\in (0,\infty)$, 
let 
	$\langle\cdot,\cdot\rangle\colon\R^d\times\R^d\to\R$ 
	be the standard scalar product on $\R^d$, 
let 
	$\norm{\cdot}\colon\R^d\to [0,\infty)$ 
	be the standard norm on $\R^d$, 
let 
	$\HSnorm{\cdot}\colon\R^{d\times m}\to [0,\infty)$ be the Frobenius norm on $\R^{d\times m}$, 
let 
	$\cO\subseteq\R^d$ be a non-empty open set, 
for every 
	$r\in (0,\infty)$ 
let 
	$O_r\subseteq\cO$ 
	satisfy 
	$O_r = \{x\in\cO\colon\norm{x}\leq r~\text{and}~\{y\in\R^d\colon\norm{y-x}<\nicefrac{1}{r}\}\subseteq\cO\}$, 
let 
	$\mu\in C([0,T]\times\cO,\R^d)$, 
	$\sigma\in C([0,T]\times\cO,\R^{d\times m})$ 
	satisfy for all 
		$r\in (0,\infty)$ 
	that
\begin{equation}\label{stochastic_continuity_lemma:locally_lipschitz}
\sup\left(
\left\{
\frac{
	\norm{ 
		\mu(t,x) - \mu(t,y) } 
	+ 
	\HSnorm{
		\sigma(t,x) - \sigma(t,y) }
}
{\norm{ x - y }}
\colon 
t\in [0,T], 
x,y\in O_r, 
x\neq y
\right\}
\cup \{0\}
\right)
< 
\infty, 
\end{equation}
let 
	$V\in C^{1,2}([0,T]\times\cO,(0,\infty))$ 
satisfy for all 
	$t\in [0,T]$, 
	$x\in \cO$ 
that  
\begin{equation}
(\tfrac{\partial V}{\partial t})(t,x) 
+ 
\tfrac12 
\operatorname{Trace}\!\left( 
\sigma(t,x)[\sigma(t,x)]^{*}
(\operatorname{Hess}_x V)(t,x)
\right) 
+ 
\langle 
\mu(t,x),
(\nabla_x V)(t,x) 
\rangle
\leq 0,  
\end{equation}
assume that 
	$
	\sup_{r\in (0,\infty)} 
	[\inf_{t\in [0,T]} 
	\inf_{x\in \cO\setminus O_r} 
	V(t,x)
	] 
	= 
	\infty
	$, 	
let 
	$(\Omega,\mathcal{F},\P,
	(\mathbb{F}_t)_{t\in [0,T]})$ 
be a filtered probability space which satisfies the usual conditions, 
let 
	$W\colon [0,T]\times\Omega\to\R^m$ 
be a standard 
	$(\mathbb{F}_t)_{t\in [0,T]}$-Brownian motion,
and for every 
	$t\in [0,T]$, 
	$x\in\cO$ 
let 
	$ 
	X^{t,x} = 
	(X^{t,x}_s)_{s\in [t,T]} \colon [t,T]\times\Omega\to\cO$ 
be an $(\mathbb{F}_s)_{s\in [t,T]}$-adapted stochastic process with continuous sample paths which satisfies that for all 
	$s\in [t,T]$ 
it holds $\P$-a.s.\ that 
	\begin{equation}
	X^{t,x}_s 
	= 
	x 
	+
	\int_t^s 
	\mu\big( r, X^{t,x}_r \big) 
	\,dr 
	+ 
	\int_t^s 
	\sigma\big( r, X^{t,x}_{r} \big) \,dW_r.   
	\end{equation}
Then it holds for all 
	$\varepsilon\in (0,\infty)$, 
	$s\in [0,T]$ 
and all
	$(t_n,x_n)\in [0,T]\times\cO$, $n\in\N_0$, 
with
	$\limsup_{n\to\infty}[|t_n-t_0|+\norm{x_n-x_0}]=0$ 
that 
	\begin{equation} 
	 \limsup_{n\to\infty} 
	 \left[ 
	 \P\!\left(  
	 	\Normm{X^{t_n,x_n}_{\max\{s,t_n\}} - X^{t_0,x_0}_{\max\{s,t_0\}}}
		\geq
		\varepsilon
	 \right)
	 \right] = 0. 
	\end{equation} 
\end{lemma}

\begin{proof}[Proof of \cref{stochastic_continuity_lemma}]
Throughout this proof let 
 	$(\mathfrak{t}_n,\mathfrak{x}_n)\in [0,T]\times\cO$, $n\in\N_0$, 
satisfy 
	$\limsup_{n\to\infty} [|\mathfrak{t}_n-\mathfrak{t}_0| + \norm{\mathfrak{x}_n-\mathfrak{x}_0}]=0$. 
Note that for the proof of~\cref{stochastic_continuity_lemma} it is 
sufficient to demonstrate that for all 
	$\varepsilon \in (0,\infty)$, 
	$s\in [0,T]$ 
it holds that
	\begin{equation}
	\label{stochastic_continuity_lemma:suffices_to_show}
	\limsup_{n\to\infty} 
	\left[
	\P\!\left(\Normm{X^{\mathfrak{t}_n,\mathfrak{x}_n}_{\max\{s,\mathfrak{t}_n\}} - X^{\mathfrak{t}_0,\mathfrak{x}_0}_{\max\{s,\mathfrak{t}_0\}}} \geq \varepsilon \right)
	\right] = 0. 
	\end{equation} 
Next note that the assumption that it holds that 
	$\sup_{r\in (0,\infty)} \inf_{t\in [0,T]} \inf_{x\in\cO\setminus O_r} V(t,x) = \infty$ 
ensures that for every 
	$n\in\N$
there exists 
	$r\in (0,\infty)$ 
such that 
	$\inf_{t\in [0,T]} \inf_{x\in\cO \setminus O_r} V(t,x) > n$. 
This yields that for every 
	$n\in\N$ 
there exists 
	$r\in (0,\infty)$ 
such that 
	$\{ V \leq n \} \subseteq [0,T]\times O_r$. 
Hence, we obtain for every 
	$n\in\N$ 
that 
	$\{V \leq n\}$ is a bounded set. 
Combining this with the fact that 
	$V\colon [0,T]\times\cO \to (0,\infty)$ 
is continuous demonstrates that for every 
	$n\in\N$ 
it holds that 
	$\{ V \leq n \}$ 
is a compact set. 
Lang~\cite[Theorem II.3.7]{Lang1999FundamentalsOfDifferentialGeometry}
therefore ensures that there exist 
	$\varphi_n\in C^{\infty}_c([0,T]\times\cO,\R)$,  $n\in\N$, 
which satisfy for all 
	 $n\in\N$, 
	 $t\in [0,T]$, 
	 $x\in \cO$  
that
 	\begin{equation} \label{stochastic_continuity_lemma:varphi_n}
	 \mathbbm{1}_{\{V\leq n\}}(t,x)
	 \leq 
	 \varphi_n(t,x) 
	 \leq 
	 \mathbbm{1}_{\{V < n+1\}}(t,x)
	 .
	 \end{equation} 
Next let
	 $\mathfrak{m}_{n}\colon [0,T]\times\R^d \to \R^d$, $n\in\N$, 
and 
     $\mathfrak{s}_{n}\colon [0,T]\times\R^d \to \R^{d\times m}$, $n\in\N$,
satisfy for all 
	$n\in \N$,
	$t\in [0,T]$, 
	$x\in \R^d$ 
that 
	\begin{equation} 
	\mathfrak{m}_n(t,x) 
	= 
	\begin{cases}
	\varphi_n(t,x)\mu(t,x) & \colon x\in\cO \\
	0 & \colon x\in\R^d\setminus\cO
	\end{cases}
	\quad\text{and}\quad 
	\mathfrak{s}_n(t,x)
	= 
	\begin{cases}
	\varphi_n(t,x)\sigma(t,x) & \colon x\in\cO \\
	0 & \colon x\in\R^d\setminus\cO. 
	\end{cases}   
	\end{equation} 
This, \eqref{stochastic_continuity_lemma:locally_lipschitz}, and \eqref{stochastic_continuity_lemma:varphi_n} assure that 
	$\mathfrak{m}_n\colon [0,T]\times\R^d\to\R^d$, $n\in\N$, 
and 
	$\mathfrak{s}_n\colon [0,T]\times\R^d\to\R^{d\times m}$, $n\in\N$, 
are compactly supported continuous functions which satisfy that 
\begin{enumerate}[(A)]
	\item
	\label{stochastic_continuity_lemma:lipschitz_approximations_itemA}
	 for all 
	$n\in\N$ 
	it holds that 
	\begin{equation}
	\sup_{t\in [0,T]}
	\sup_{\substack{x,y\in \R^d\\x\neq y}}
	\left[ 
	\frac{
		\norm{ 
			\mathfrak{m}_n(t,x) 
			- 
			\mathfrak{m}_n(t,y) 
		} 
		+ 
		\HSnorm{ 
			\mathfrak{s}_n(t,x) 
			- 
			\mathfrak{s}_n(t,y) 
		}
	}
	{\norm{ x - y }}
	\right] 
	<
	\infty, 
	\end{equation}
	\item \label{stochastic_continuity_lemma:lipschitz_approximations_itemB}
	for all 
	$n\in\N$, 
	$t\in [0,T]$, 
	$x\in \cO$
	it holds that 
	\begin{equation}
	\left[ 
	\norm{ 
		\mathfrak{m}_n(t,x) 
		-
		\mu(t,x) } 
	+ 
	\HSnorm{ 
		\mathfrak{s}_n(t,x) 
		- 
		\sigma(t,x)
	}
	\right]
	\mathbbm{1}_{\{V\leq n\}}(t,x)
	= 
	0,  
	\end{equation}
	and
	\item 
	\label{stochastic_continuity_lemma:compactly_supported_approximations_itemC}
	for all 
	$n\in\N$, 
	$t\in [0,T]$, 
	$x\in\cO$
	it holds that 
	\begin{equation}
	\left[ 
	\norm{ 
		\mathfrak{m}_n(t,x) 
	}
	+
	\HSnorm{ 
		\mathfrak{s}_n(t,x) 
	}
	\right]
	\mathbbm{1}_{\{V\geq n+1\}}(t,x)
	= 
	0.  
	\end{equation}
\end{enumerate} 
Note that Karatzas \& Shreve~\cite[Theorem 5.2.9]{KaSh1991_BrownianMotionAndStochasticCalculus} (cf.~also Gy\"ongy \& Krylov~\cite[Corollary 2.6]{GyoengyKrylov1995_existenceStrong} and Liu \& R\"ockner~\cite[Theorem 3.1.1]{LiuRoeckner2015SPDEs}) and Item~\eqref{stochastic_continuity_lemma:lipschitz_approximations_itemA} yield that for every 
	$n\in\N$, 
	$t\in [0,T]$, 
	$x\in\cO$ 
there exists an $(\F_s)_{s\in [t,T]}$-adapted  stochastic process with continuous sample paths $\X^{n,t,x}=(\X^{n,t,x}_{s})_{s\in [t,T]}\colon [t,T] \times \Omega \to \R^d$ which satisfies that for all 
	$s\in [t,T]$ 
it holds $\P$-a.s.~that 
	\begin{equation} 
	\X^{n,t,x}_{s} 
	= 
	x 
	+ 
	\int_t^s 
	\mathfrak{m}_n\big( r, \X^{n,t,x}_{r}\big) \,dr 
	+
	\int_t^s 
	\mathfrak{s}_n\big( r, \X^{n,t,x}_{r} \big) \,dW_r . 
	\end{equation}
Moreover, note that Item~\eqref{stochastic_continuity_lemma:compactly_supported_approximations_itemC} ensures for all 
	$n\in\N$ 
that 
	$\operatorname{supp}(\mathfrak{m}_n)\cup\operatorname{supp}(\mathfrak{s}_n) \subseteq \{V \leq n+1\}$. 
The fact that for every 
	$n\in\N$ 
there exists 
	$r\in (0,\infty)$ 
such that 
	$\{ V \leq n \} \subseteq [0,T]\times O_r$ 
hence implies that for every
	$n\in\N$
there exists 
	$r\in (0,\infty)$ 
such that
	$
	\operatorname{supp}(\mathfrak{m}_n) \cup \operatorname{supp}(\mathfrak{s}_n) 
	\subseteq [0,T]\times O_r
	$. 
Furthermore, observe that Item~\eqref{sde_compactly_supported:item1} of \cref{sde_compactly_supported} ensures that for all 
	$n\in\N$, 
	$r\in (0,\infty)$, 
	$m\in\N\cap (r,\infty)$, 
	$t\in [0,T]$, 
	$x\in \cO\setminus\{y\in\cO\colon (\Exists z\in O_r\colon\norm{y-z}<\nicefrac{1}{m}) \}$
with 
	$ 
	\operatorname{supp}(\mathfrak{m}_n) \cup \operatorname{supp}(\mathfrak{s}_n) 
	\subseteq [0,T]\times O_r$
it holds that 
	$\P(\Forall s\in [t,T]\colon \X^{n,t,x}_s=x)=1$. 
Combining this with the fact that for all 
	$r\in (0,\infty)$ 
it holds that 
	$O_r = \cap_{m\in\N\cap(r,\infty)} \{y\in\R^d\colon(\Exists x\in O_r\colon\norm{x-y}<\nicefrac{1}{m})\}$	
implies that for all 
	$n\in\N$, 
	$r\in (0,\infty)$, 
	$t\in [0,T]$,
	$x\in \cO\setminus O_r$
with 
	$ 
	\operatorname{supp}(\mathfrak{m}_n) \cup \operatorname{supp}(\mathfrak{s}_n) 
	\subseteq [0,T]\times O_r$
it holds that 
	\begin{equation} \label{stochastic_continuity_lemma:outside_Or}
			\P(\Forall s\in [t,T]\colon \X^{n,t,x}_s=x)=1. 
	\end{equation}
Next we observe that Item~\eqref{sde_compactly_supported:item2}  of \cref{sde_compactly_supported} ensures that for all 
	$n\in\N$, 
	$r\in (0,\infty)$, 
	$m\in\N\cap (r,\infty)$, 
	$t\in [0,T]$,
	$x\in \{y\in\cO\colon (\Exists z\in O_r\colon\norm{y-z}<\nicefrac{1}{m})\}$
with 
	$
	\operatorname{supp}(\mathfrak{m}_n)\cup\operatorname{supp}(\mathfrak{s}_n) 
	\subseteq [0,T]\times O_r$
it holds that 
	$
	\P(\forall s\in[t,T]\colon (\Exists y\in O_r\colon\norm{\X^{n,t,x}_s - y }\leq \nicefrac{1}{m}))
	= 
	1$. 
This yields that for all 
	$n\in\N$, 
	$r\in (0,\infty)$, 
	$t\in [0,T]$,
	$x\in O_r$ 
with 
	$ 
	\operatorname{supp}(\mathfrak{m}_n) \cup \operatorname{supp}(\mathfrak{s}_n) 
	\subseteq [0,T]\times O_r$
it holds that 
	\begin{equation} \label{stochastic_continuity_lemma:in_Or}
			\P(\forall s\in [t,T]\colon \X^{n,t,x}_s\in O_r)=1. 
	\end{equation}
The fact that for every
	$n\in\N$
there exists 
	$r\in (0,\infty)$ 
such that
	$
	\operatorname{supp}(\mathfrak{m}_n) \cup \operatorname{supp}(\mathfrak{s}_n) 
	\subseteq [0,T]\times O_r
	$ 
and  
\eqref{stochastic_continuity_lemma:outside_Or} hence demonstrate that for every 
	$n\in\N$ 
there exists 
	$r\in (0,\infty)$ 
such that 
\begin{enumerate}[(I)]
	\item 
it holds for all 
	$t\in [0,T]$, 
	$x\in O_r$ 
that 
	$\P(\forall s\in [t,T]\colon \X^{n,t,x}_s\in O_r)=1$
and 
	\item 
it holds for all 
	$t\in [0,T]$, 
	$x\in \cO\setminus O_r$ 
that 
	$\P(\forall s\in [t,T]\colon \X^{n,t,x}_s=x)=1$. 
\end{enumerate}
Therefore, we obtain that for every 
	$n\in\N$, 
	$t\in [0,T]$, 
	$x\in \cO$ 
there exists an $(\F_s)_{s\in [t,T]}$-adapted stochastic process with continuous sample paths 
	$\mathfrak{X}^{n,t,x}=(\mathfrak{X}^{n,t,x}_s)_{s\in [t,T]}\colon [t,T]\times\Omega \to \cO$ 
which satisfies that for all 
	$s\in [t,T]$ 
it holds $\P$-a.s.~that 
	\begin{equation} \label{stochastic_continuity_lemma:approximation_sde}
	\mathfrak{X}^{n,t,x}_s 
	= 
	x 
	+ 
	\int_t^s \mathfrak{m}_n(r,\mathfrak{X}^{n,t,x}_r)\,dr 
	+ 
	\int_t^s \mathfrak{s}_n(r,\mathfrak{X}^{n,t,x}_r)\,dW_r.
	\end{equation} 
In the next step let 
	$\tau^{n,t,x}\colon \Omega\to [t,T]$, 
	$n\in \N$,
	$t\in [0,T]$, 
	$x\in \cO$,
satisfy for every 
	$n\in \N$,
	$t\in [0,T]$, 
	$x\in \cO$,
	$\omega\in\Omega$ 
that 
	$
	\tau^{n,t,x}(\omega) = 
	\inf(
	\{
	s\in [t,T]\colon \max\{V(s,\mathfrak{X}^{n,t,x}_{s}(\omega)),
	V(s,X^{t,x}_{s}(\omega))\} 
	> n
	\} \cup \{T\}
	)
	$. 
Note that for every 
	$n\in \N$,
	$t\in [0,T]$, 
	$x\in \cO$
it holds that $\tau^{n,t,x}\colon\Omega\to [t,T]$ is an  $(\mathbb{F}_s)_{s\in [t,T]}$-stopping time. Next observe that the fact that for every 
	$n\in\N$ 
it holds that  
	$\{V\leq n\}$
is a compact set,  Item~\eqref{stochastic_continuity_lemma:lipschitz_approximations_itemB}, and \cref{coinciding_until_stopping_times} (with 
	$T=T-t$, 
	$\cC=\{(s,y) \in [0,T-t]\times\cO\colon V(t+s,y) \leq n\}$, 
	$\mu_1=([0,T-t]\times\cO \ni (s,y) \mapsto \mu(t+s,y) \in \R^d)$, 
	$\mu_2=([0,T-t]\times\cO \ni (s,y) \mapsto \mathfrak{m}_n(t+s,y) \in \R^d)$, 
	$\sigma_1=([0,T-t]\times\cO \ni (s,y) \mapsto \sigma(t+s,y) \in \R^{d\times m})$, 
	$\sigma_2=([0,T-t]\times\cO \ni (s,y) \mapsto \mathfrak{s}_n(t+s,y) \in \R^{d\times m})$, 
	$\F=(\F_{t+s})_{s\in [0,T-t]}$, 
	$W=([0,T-t]\times\Omega\ni (s,\omega) \mapsto W_{t+s}(\omega)-W_t(\omega)\in\R^m)$, 
	$X^{(1)}=([0,T-t]\times\Omega\ni (s,\omega) \mapsto X^{t,x}_{t+s}(\omega) \in \cO)$, 
	$X^{(2)}=([0,T-t]\times\Omega\ni (s,\omega) \mapsto 
	\mathfrak{X}^{n,t,x}_{t+s}(\omega)\in\cO)$,
	$\tau=\tau^{n,t,x}-t$
for 
	$n\in\N$, 
	$t\in [0,T]$, 
	$x\in \cO$
in the notation of \cref{coinciding_until_stopping_times}) 
ensure for all 
	$n\in \N$, 
	$t\in [0,T]$,
	$x\in \cO$ 
that
	\begin{equation}
	\P\!
	\left(
	\Forall s\in [t,T]\colon
	\mathbbm{1}_{\{ s\leq\tau^{n,t,x} \}}
	\norm{
	\mathfrak{X}^{n,t,x}_{s} - X^{t,x}_s 
	}
	=
	0 
	\right) 
	= 1.
	\end{equation}
This, Markov's inequality, and \cref{integral_estimate_through_lyapunov_function} 
(with 
	$T=T-t$, 
	$\mu=([0,T-t]\times\cO\ni (s,y) \mapsto \mu(t+s,y)\in\R^d)$, 
	$\sigma=([0,T-t]\times\cO\ni (s,y) \mapsto \sigma(t+s,y)\in\R^{d\times m})$, 
	$V=([0,T-t]\times\cO \ni (s,y) \mapsto V(t+s,y) \in [0,\infty))$, 
	$\F=(\F_{t+s})_{s\in [0,T-t]}$, 
	$W=([0,T-t]\times\Omega \ni (s,\omega)\mapsto W_{t+s}(\omega)-W_t(\omega)\in\R^m)$, 
	$X=([0,T-t]\times\Omega \ni (s,\omega) \mapsto X^{t,x}_{t+s}(\omega) \in \cO)$, 
	$\tau=\tau^{k,t,x}-t$
for 
	$k\in\N$,
	$t\in [0,T]$, 
	$x\in \cO$ 
in the notation of \cref{integral_estimate_through_lyapunov_function}) imply for all 
	$\varepsilon\in (0,\infty)$, 
	$k\in\N$, 
	$t\in [0,T]$, 
	$x\in \cO$, 
	$s\in [t,T]$
that 
	\begin{equation} 
	\label{stochastic_continuity_lemma:discrepancy_between_mathfrak_X_and_X_processes}
	\begin{split}
		&
		\P\!\left( 
		  \norm{\mathfrak{X}^{k,t,x}_s-X^{t,x}_s} \geq \varepsilon 
		\right) 
			\\
		&
		\leq 
		\P\!\left( 
		  \tau^{k,t,x} < s
		\right)
		\leq
		\P( 
		V(\tau^{k,t,x}, X^{t,x}_{\tau^{k,t,x}} ) \geq k)
		\leq 
		 \frac{1}{k}\Exp{V(\tau^{k,t,x},X^{t,x}_{\tau^{k,t,x}})}
		 \leq 
		 \frac{1}{k}V(t,x). 
	\end{split}
	\end{equation}
Furthermore, observe that \cref{stability_for_sdes} ensures that there exist real numbers 
	$c_k \in [0,\infty)$, $k\in\N$, 
which satisfy that for every 		
	$k,n\in\N$, 
	$s\in [\mathfrak{t}_0,T]$  
it holds that 
\begin{equation} \label{stochastic_continuity_lemma:application_of_a_priori_estimates}
 \Exp{\norm{\mathfrak{X}^{k,\mathfrak{t}_n,\mathfrak{x}_n}_{\max\{s,\mathfrak{t}_n\}}-\mathfrak{X}^{k,\mathfrak{t}_0,\mathfrak{x}_0}_{\max\{s,\mathfrak{t}_n\}}}^2} 
 \leq c_k \left[ |\mathfrak{t}_n-\mathfrak{t}_0| + \norm{\mathfrak{x}_n-\mathfrak{x}_0}^2\right]\!.
\end{equation} 
Moreover, observe that \eqref{stochastic_continuity_lemma:approximation_sde} ensures that for all 
	$k,n\in\N$, 
	$s\in [\mathfrak{t}_0,T]$ 
it holds $\P$-a.s.~that 
	\begin{equation} 
	\begin{split}
	\mathfrak{X}^{k,\mathfrak{t}_0,\mathfrak{x}_0}_{\max\{s,\mathfrak{t}_n\}} 
	- 
	\mathfrak{X}^{k,\mathfrak{t}_0,\mathfrak{x}_0}_{\max\{s,\mathfrak{t}_0\}} 
	& = 
	\mathfrak{X}^{k,\mathfrak{t}_0,\mathfrak{x}_0}_{\max\{s,\mathfrak{t}_n\}} 
	- 
	\mathfrak{X}^{k,\mathfrak{t}_0,\mathfrak{x}_0}_{s} 
	\\
	& = 
	\int_{s}^{\max\{s,\mathfrak{t}_n\}} 
	\mathfrak{m}_k(r,\mathfrak{X}^{k,\mathfrak{t}_0,\mathfrak{x}_0}_r)\,dr 
	+ 
	\int_{s}^{\max\{s,\mathfrak{t}_n\}} 
	\mathfrak{s}_k(r,\mathfrak{X}^{k,\mathfrak{t}_0,\mathfrak{x}_0}_r)\,dW_r. 
	\end{split}
	\end{equation}
Minkowski's inequality, the fact that 
	$\mathfrak{m}_k\colon [0,T]\times\R^d \to \R^d$, $k\in\N$, 
and 
	$\mathfrak{s}_k\colon [0,T]\times\R^d \to \R^{d\times m}$, $k\in\N$, 
are compactly supported continuous functions, and It\^o's isometry hence imply that for all 
	$k,n\in\N$, 
	$s \in [\mathfrak{t}_{0},T]$ 
it holds that 
	\begin{equation} 
	\begin{split}
    & 
    \left(
    \Exp{\norm{\mathfrak{X}^{k,\mathfrak{t}_0,\mathfrak{x}_0}_{\max\{s,\mathfrak{t}_n\}} 
    - 
    \mathfrak{X}^{k,\mathfrak{t}_0,\mathfrak{x}_0}_{\max\{s,\mathfrak{t}_0\}}}^2}\right)^{\!\nicefrac12} 
	\\
	& \leq 
	\int_{s}^{\max\{s,\mathfrak{t}_n\}} 
	\left(\Exp{ \norm{\mathfrak{m}_k(r,\mathfrak{X}^{k,\mathfrak{t}_0,\mathfrak{x}_0}_r)}^2}\right)^{\!\nicefrac12}\,dr
	+ 
	\left(
	\int_{s}^{\max\{s,\mathfrak{t}_n\}} 
	\Exp{\HSnorm{\mathfrak{s}_k(r,\mathfrak{X}^{k,\mathfrak{t}_0,\mathfrak{x}_0}_r)}^2}\,dr\right)^{\!\nicefrac12}
	\\
	& \leq 
	\left|\max\{0,\mathfrak{t}_n-s\}\right|^{\nicefrac12}
	\left[ 
	\sqrt{T} 
	\left(
	\sup_{t\in [0,T]} 
	\sup_{x\in \cO} 
	\norm{\mathfrak{m}_k(t,x)} 
	\right)
	+ 
	\left(
	\sup_{t\in [0,T]} 
	\sup_{x\in \cO} 
	\HSNorm{\mathfrak{s}_k(t,x)}
	\right)
	\right]\!. 
	\end{split}
	\end{equation} 
This, the fact that for all 
	$a,b\in \R$ 
it holds that $(a+b)^2 \leq 2 a^2 + 2 b^2 $, 
and \eqref{stochastic_continuity_lemma:application_of_a_priori_estimates}
ensure that there exist real numbers 
	$\mathfrak{c}_k \in [0,\infty)$, $k\in\N$, 
which satisfy for every 
	$k,n\in\N$, 
	$s\in [\mathfrak{t}_0,T]$ 
that 
	\begin{equation} \label{stochastic_continuity_lemma:case1}
	\begin{split}
	\Exp{\norm{\mathfrak{X}^{k,\mathfrak{t}_n,\mathfrak{x}_n}_{\max\{s,\mathfrak{t}_n \}} - \mathfrak{X}^{k,\mathfrak{t}_0,\mathfrak{x}_0}_{\max\{s,\mathfrak{t}_0\}} }^2 }
	& 
	\leq 
	2 \Exp{\norm{\mathfrak{X}^{k,\mathfrak{t}_n,\mathfrak{x}_n}_{\max\{s,\mathfrak{t}_n\}}-\mathfrak{X}^{k,\mathfrak{t}_0,\mathfrak{x}_0}_{\max\{s,\mathfrak{t}_n\}}}^2} 
	+ 
	2 \Exp{\norm{\mathfrak{X}^{k,\mathfrak{t}_0,\mathfrak{x}_0}_{\max\{s,\mathfrak{t}_n\}} 
			- 
			\mathfrak{X}^{k,\mathfrak{t}_0,\mathfrak{x}_0}_{s}}^2}
	\\
	& 
	\leq 
	\mathfrak{c}_k 
	\left[ |\mathfrak{t}_n-\mathfrak{t}_0| + \norm{\mathfrak{x}_n-\mathfrak{x}_0}^2 + 
	\max\{0,\mathfrak{t}_n - s\} \right]\!.
	\end{split}
	\end{equation} 
In addition, observe that \eqref{stochastic_continuity_lemma:approximation_sde} ensures that for all 
	$k,n\in\N$, 
	$s\in [0,\mathfrak{t}_0]$ 
it holds $\P$-a.s.~that 
	\begin{equation} 
	\begin{split}
	\mathfrak{X}^{k,\mathfrak{t}_n,\mathfrak{x}_n}_{\max\{s,\mathfrak{t}_n\}} 
	- 
	\mathfrak{X}^{k,\mathfrak{t}_0,\mathfrak{x}_0}_{\max\{s,\mathfrak{t}_0\}} 
	& = 
	\mathfrak{X}^{k,\mathfrak{t}_n,\mathfrak{x}_n}_{\max\{s,\mathfrak{t}_n\}} 
	- 
	\mathfrak{x}_0 
	\\
	& = 
	\mathfrak{x}_n 
	- 
	\mathfrak{x}_0 
	+ 
	\int_{\mathfrak{t}_n}^{\max\{s,\mathfrak{t}_n\}}
	\mathfrak{m}_k(r,\mathfrak{X}_{r}^{k,\mathfrak{t_n},\mathfrak{x}_n})\,dr
	+ 	
	\int_{\mathfrak{t}_n}^{\max\{s,\mathfrak{t}_n\}}
	\mathfrak{s}_k(r,\mathfrak{X}_{r}^{k,\mathfrak{t_n},\mathfrak{x}_n})\,dW_r. 
	\end{split}
	\end{equation} 
Minkowski's inequality, the fact that 
$\mathfrak{m}_k\colon [0,T]\times\R^d \to \R^d$, $k\in\N$, 
and 
$\mathfrak{s}_k\colon [0,T]\times\R^d \to \R^{d\times m}$, $k\in\N$, 
are compactly supported continuous functions, and It\^o's isometry hence imply that for all 
	$k,n\in\N$, 
	$s \in [0,\mathfrak{t}_{0}]$ 
it holds that 
	\begin{equation} 
	\begin{split}
	& 
	\left(  \Exp{\norm{\mathfrak{X}^{k,\mathfrak{t}_n,\mathfrak{x}_n}_{\max\{s,\mathfrak{t}_n\}} 
			- 
			\mathfrak{X}^{k,\mathfrak{t}_0,\mathfrak{x}_0}_{\max\{s,\mathfrak{t}_0\}}}^2}\right)^{\!\nicefrac12} 
	\leq  
	\norm{
		\mathfrak{x}_n 
		- 
		\mathfrak{x}_0} 
	\\
	& + 
	\int_{\mathfrak{t}_n}^{\max\{s,\mathfrak{t}_n\}}
	\left(\Exp{	\norm{\mathfrak{m}_k(r,\mathfrak{X}_{r}^{k,\mathfrak{t_n},\mathfrak{x}_n})}^2}\right)^{\!\nicefrac12}
	\!\,dr
	+ 	
	\left( 
	\int_{\mathfrak{t}_n}^{\max\{s,\mathfrak{t}_n\}}
	\Exp{\HSnorm{\mathfrak{s}_k(r,\mathfrak{X}_{r}^{k,\mathfrak{t_n},\mathfrak{x}_n})}^2}\!\,dr\right)^{\!\nicefrac12}
	\\
	& \leq 
	\norm{\mathfrak{x}_n-\mathfrak{x}_0} 
	+ 
	|\!\max\{0,s-\mathfrak{t}_n\}|^{\nicefrac12}
	\left[ 
	\sqrt{T} 
	\left(\sup_{t\in[0,T]}\sup_{x\in\cO} \norm{\mathfrak{m}_k(t,x)} 
	\right) 
	+ 
	\left(\sup_{t\in[0,T]}\sup_{x\in\cO}
	\HSnorm{\mathfrak{s}_k(t,x)}\right)
	\right]\!.
	\end{split}
	\end{equation} 
This and \eqref{stochastic_continuity_lemma:case1} ensure that there exist real numbers 
$\mathfrak{c}_k \in [0,\infty)$, $k\in\N$, 
which satisfy for every 
	$k,n\in\N$, 
	$s\in [0,T]$ 
that 
	\begin{equation} 
	\begin{split}
	& 
	\Exp{\norm{\mathfrak{X}^{k,\mathfrak{t}_n,\mathfrak{x}_n}_{\max\{s,\mathfrak{t}_n \}} - \mathfrak{X}^{k,\mathfrak{t}_0,\mathfrak{x}_0}_{\max\{s,\mathfrak{t}_0\}} }^2 }
	\\
	& \leq 
	\mathfrak{c}_k 
	\left[ |\mathfrak{t}_n-\mathfrak{t}_0| + \norm{\mathfrak{x}_n-\mathfrak{x}_0}^2 + 
	\mathbbm{1}_{[0,\mathfrak{t}_0]}(s)\max\{0,s-\mathfrak{t}_n\} + 
	\mathbbm{1}_{[\mathfrak{t}_0,T]}(s)\max\{0,\mathfrak{t}_n - s\} \right]\!.
	\end{split}
	\end{equation}
Combining this with Markov's inequality and  \eqref{stochastic_continuity_lemma:discrepancy_between_mathfrak_X_and_X_processes} demonstrates that for all 
	$\varepsilon\in (0,\infty)$, 
	$s\in [0,T]$
it holds that 
	\begin{equation}
	\begin{split}
	&
	\limsup_{n\to\infty} 
	\left[ 
	\P\!\left(\Normm{X^{\mathfrak{t}_n,\mathfrak{x}_n}_{\max\{s,\mathfrak{t}_n\}}-X^{\mathfrak{t}_0,\mathfrak{x}_0}_{\max\{s,\mathfrak{t}_0 \}}}\geq \varepsilon\right)\right] 
	\\
	& 
	\leq 
	\inf_{k\in\N} 
	\Bigg(
	\limsup_{n\to\infty} 
	\Bigg[    
	\P\!\left(\Normm{X^{\mathfrak{t}_n,\mathfrak{x}_n}_{\max\{s,\mathfrak{t}_n\}}-\mathfrak{X}^{k,\mathfrak{t}_n,\mathfrak{x}_n}_{\max\{s,\mathfrak{t}_n\}}}\geq \frac{\varepsilon}{3}\right)
	\\
   	& \qquad \qquad \qquad + 
	\P\!\left(\Normm{\mathfrak{X}^{k,\mathfrak{t}_n,\mathfrak{x}_n}_{\max\{s,\mathfrak{t}_n\}}-\mathfrak{X}^{k,\mathfrak{t}_0,\mathfrak{x}_0}_{\max\{s,\mathfrak{t}_0 \}} } \geq \frac{\varepsilon}{3}\right)
    + 
    \P\!\left(\Normm{\mathfrak{X}^{k,\mathfrak{t}_0,\mathfrak{x}_0}_{\max\{s,\mathfrak{t}_0 \}}-X^{\mathfrak{t}_0,\mathfrak{x}_0}_{\max\{s,\mathfrak{t}_0 \}}}\geq \frac{\varepsilon}{3}\right)\!
    \Bigg] 
 	\Bigg)
	\\
	&
 	\leq 
	\inf_{k\in\N} 
	\Bigg(
	\limsup_{n\to\infty} 
	\Bigg[  
	\frac{V(\mathfrak{t}_n,\mathfrak{x}_n)}{k} 
	+
	\frac{9}{\varepsilon^2}\Exp{\norm{\mathfrak{X}^{k,\mathfrak{t}_n,\mathfrak{x}_n}_{\max\{s,\mathfrak{t}_n\}}-\mathfrak{X}^{k,\mathfrak{t}_0,\mathfrak{x}_0}_{\max\{s,\mathfrak{t}_0 \}}}^2}
	+
	\frac{V(\mathfrak{t}_0,\mathfrak{x}_0)}{k}  
	\Bigg] 
	\Bigg) 
	\\
	& 
	\leq 
	\inf_{k\in\N} 
	\Bigg(
	\limsup_{n\to\infty} 
	\Bigg[  
	\frac{V(\mathfrak{t}_n,\mathfrak{x}_n)}{k} 
 	+
 	\frac{9\mathfrak{c}_k }{\varepsilon^2}
 	\bigg( |\mathfrak{t}_n-\mathfrak{t}_0| + \norm{\mathfrak{x}_n-\mathfrak{x}_0}^2 
 	\\
 	& \qquad \qquad \qquad 
 	+ \mathbbm{1}_{[0,\mathfrak{t}_0]}(s)\max\{0,s-\mathfrak{t}_n\}  
	+	\mathbbm{1}_{[\mathfrak{t}_0,T]}(s)\max\{0,\mathfrak{t}_n - s\} \bigg)
 	+
 	\frac{V(\mathfrak{t}_0,\mathfrak{x}_0)}{k}  
 	\Bigg] 
 	\Bigg) 
 	\\
 	& 
 	= 
 	\inf_{k\in\N} 
 	\left( 
   	\frac{2V(\mathfrak{t}_0,\mathfrak{x}_0)}{k} 
  	\right)
 	= 0. 
\end{split}  
\end{equation} 
This demonstrates \eqref{stochastic_continuity_lemma:suffices_to_show}. 
The proof of \cref{stochastic_continuity_lemma} is thus completed. 
\end{proof}

The next result, \cref{existence_sde_setting} below, is the main result of this article. It is an application of \cref{abstract_existence}.  \cref{integral_estimate_through_lyapunov_function,stochastic_continuity_lemma} above ensure that the crucial hypotheses of \cref{abstract_existence} are satisfied in the setting of \cref{existence_sde_setting}.  

\begin{theorem}
	\label{existence_sde_setting}
Let 
	$d,m\in\N$, 
	$L,T\in (0,\infty)$, 
let 
	$\langle\cdot,\cdot\rangle\colon\R^d\times\R^d\to\R$ 
be the standard scalar product on $\R^d$, 
let 
	$\norm{\cdot}\colon\R^d\to [0,\infty)$ 
be the standard norm on $\R^d$, 
let 
	$\HSnorm{\cdot}\colon\R^{d\times m}\to [0,\infty)$ 
be the Frobenius norm on 
	$\R^{d\times m}$, 
let 
	$\cO \subseteq  \R^d$ be a non-empty open set,
for every 
	$r\in (0,\infty)$ 
let 
	$O_r \subseteq \cO$ 
satisfy  
	$
	O_r = \{x\in\cO\colon \norm{x}\leq r~\text{and}~\{y\in\R^d\colon 
	\norm{y-x} < \nicefrac{1}{r}\}
	\subseteq \cO\}
	$,
let 
	$\mu \in C([0,T]\times\cO,\R^d)$, 
	$\sigma\in C([0,T]\times \cO, 
	\R^{d\times m})$ 
satisfy for all 
	$r\in (0,\infty)$ 
that 
	\begin{equation}
	\sup
	\left(
	\left\{
	\frac{
		\norm{ \mu(t,x) - \mu(t,y) } 
		+ 
		\HSnorm{ \sigma(t,x) - \sigma(t,y) }}
	{
		\norm{ x - y }}
	\colon 
	t\in [0,T], 
	x,y\in O_r, 
	x\neq y
	\right\}
	\cup \{ 0 \} 
	\right)
	<  
	\infty, 
	\end{equation}
let 
	$f\in C([0,T]\times\cO\times\R,\R)$, 
	$g\in C(\cO,\R)$,  
	$V\in C^{1,2}([0,T]\times\cO,(0,\infty))$, 
assume for all 
$t\in [0,T]$, 
$x\in \cO$, 
$v,w\in\R$
that 
$
|f(t,x,v) - f(t,x,w)| 
\leq 
L | v - w | 
$
and 
\begin{equation}
(\tfrac{\partial V}{\partial t})(t,x) 
+ 
\tfrac12 
\operatorname{Trace}\!\left( \sigma(t,x)
[\sigma(t,x)]^{*}(\operatorname{Hess}_x V)(t,x)\right) 
+ 
\langle 
\mu(t,x),
(\nabla_x V)(t,x) 
\rangle
\leq 0,  
\end{equation}
assume that 
	$
	\sup_{r\in (0,\infty)}
	[ 
	\inf_{t\in [0,T]}
	\inf_{x\in\cO\setminus O_r}
	V(t,x)
	] 
	= 
	\infty
	$ 
and 
	$
	\inf_{r\in (0,\infty)} 
	[ 
	\sup_{t\in [0,T]}
	\sup_{x\in \cO\setminus O_r} 
	( 
	\frac{|f(t,x,0)|}{V(t,x)}+
	\frac{|g(x)|}{V(T,x)}
	) 
	] 
	= 
	0
	$,
let 
	$(\Omega,\mathcal{F},\P,
	(\mathbb{F}_t)_{t\in [0,T]})$ 
	be a filtered probability space which satisfies the 
	usual conditions, 
	let 
	$W\colon [0,T]\times\Omega\to\R^m$ be a standard 
	$(\mathbb{F}_t)_{t\in [0,T]}$-Brownian motion, 
	and for every 
	$t\in [0,T]$, 
	$x\in\cO$ 
let 
	$
	X^{t,x}=(X^{t,x}_s)_{s\in [t,T]} 
	\colon [t,T]\times\Omega\to\cO
	$
	be an 
	$(\mathbb{F}_s)_{s\in [t,T]}$-adapted stochastic process with continuous sample paths which satisfies that for all 
	$s\in [t,T]$ 
	it holds $\P$-a.s.~that 
	\begin{equation}
	X^{t,x}_s
	= 
	x 
	+
	\int_t^s 
	\mu( r, X^{t,x}_r ) 
	\,dr 
	+ 
	\int_t^s 
	\sigma( r, X^{t,x}_r ) \,dW_r.   
	\end{equation}
Then there exists a unique $u\in C([0,T]\times\cO,\R)$ such that 
	\begin{enumerate}[(i)]
		\item
		\label{existence_sde_setting:item1} it holds that 
		\begin{equation}
		\limsup_{r\to\infty} 
		\left[ 
		\sup_{t\in [0,T]}
		\sup_{x\in \cO\setminus O_r}
		\left( 
		\frac{|u(t,x)|}{V(t,x)}
		\right) 
		\right] 
		= 
		0
		\end{equation}
		and
		\item		\label{existence_sde_setting:item2}
		 it holds for all 
			$t\in [0,T]$, 
			$x\in\cO$ 
		that 
		\begin{equation}
		u(t,x) 
		= 
		\Exp{
			g(X^{t,x}_T) 
			+ 
			\int_t^T f\big( s, X^{t,x}_s, u(s,X^{t,x}_s)\big) \,ds}\!. 
		\end{equation}
	\end{enumerate}
\end{theorem}

\begin{proof}[Proof of 
\cref{existence_sde_setting}] 
First, note that \cref{integral_estimate_through_lyapunov_function} 
(with 
	$T=T-t$,
	$\mu=([0,T-t]\times\cO \ni (s,y) \mapsto \mu(t+s,y)\in\R^d)$, 
	$\sigma = ([0,T-t]\times\cO \ni (s,y) \mapsto \sigma(t+s,y) \in \R^{d\times m})$, 
	$V=([0,T-t]\times\cO \ni (s,y) \mapsto V(t+s,y)\in [0,\infty))$, 
	$\F = (\F_{t+s})_{s\in [0,T-t]}$, 
	$W=([0,T-t]\times\Omega \ni (s,\omega)\mapsto 
	W_{t+s}(\omega)-W_t(\omega)\in \R^m)$, 
	$X=([0,T-t]\times\Omega \ni (s,\omega)\mapsto X^{t,x}_{t+s}(\omega)\in\cO)$	 
for 
	$t\in [0,T]$, 
	$x\in \cO$
in the notation of \cref{integral_estimate_through_lyapunov_function})
ensures that for all 
	$t\in [0,T]$, 
	$s\in [t,T]$, 
	$x\in \cO$ 
it holds that 
	\begin{equation}
	\label{existence_sde_setting:supermartingale_type_inequality}
	\Exp{V(s,X^{t,x}_s)} \leq V(t,x).
	\end{equation}
Next observe that \cref{stochastic_continuity_lemma} ensures that for all 
	$\varepsilon\in (0,\infty)$,
	$s\in [0,T]$ 
and all
	$(t_n,x_n)\in [0,T]\times\cO$, $n\in\N_0$, 
with 
	$\limsup_{n\to\infty} 
	[|t_n-t_0|+\norm{x_n-x_0}]
	= 
	0
	$
it holds that 
	$
	\limsup_{n\to\infty}
	[\P(\Norm{X^{t_n,x_n}_{\max\{s,t_n\}}-X^{t_0,x_0}_{\max\{s,t_0\} }}\geq \varepsilon)] = 0$. 
Combining this with \eqref{existence_sde_setting:supermartingale_type_inequality} and \cref{abstract_existence} demonstrates that there exists a unique $u\in C([0,T]\times\cO,\R)$ which satisfies that for all 
	$t\in [0,T]$, 
	$x\in \cO$ 
it holds that 
	$\limsup_{r\to\infty} [\sup_{s\in [0,T]}\sup_{y\in\cO\setminus O_r} (\frac{|u(s,y)|}{V(s,y)})] = 0$ 
and 
\begin{equation}
 u(t,x) = \Exp{g(X^{t,x}_T) + \int_t^T f\big(s,X^{t,x}_s, u(s,X^{t,x}_s) \big)\,ds}\!.
\end{equation}
This establishes Items~\eqref{existence_sde_setting:item1} and \eqref{existence_sde_setting:item2}. The proof of \cref{existence_sde_setting} is thus completed.
\end{proof}

\cref{moment_estimate_special_lyapunov} implies the following corollary of \cref{existence_sde_setting}  in the situation in which the drift and diffusion coefficients $\mu\colon [0,T]\times\cO\to\R^d$ and $\sigma\colon [0,T]\times\cO\to\R^{d\times m}$ depend only on the spatial variable $x\in\cO$ and are independent of the time variable $t\in [0,T]$. For the sake of simplicity we take the spatial domain $\cO$ to be $\R^d$ in \cref{existence_of_fixpoint_x_dependence_full_space_lyapunov} below.
 
\begin{cor}
\label{existence_of_fixpoint_x_dependence_full_space_lyapunov}
Let 	
 	$d,m\in\N$, 
 	$L,T\in (0,\infty)$,
 	$\rho\in\R$, 
let 
 	$\langle\cdot,\cdot\rangle\colon\R^d\times\R^d\to\R$ be the standard scalar product on $\R^d$, 
let 
 	$\norm{\cdot}\colon \R^d \to [0,\infty)$ 
 	be the standard norm on $\R^d$, 
let 
	$\mu \colon \R^d \to \R^d$ 
and 
	$\sigma \colon \R^d \to \R^{d\times m}$ 
be locally Lipschitz continuous, 	
let 
 	$f\in C([0,T]\times\R^d\times\R,\R)$, 
 	$g\in C(\R^d,\R)$, 
 	$V\in C^{2}(\R^d,(0,\infty))$,
assume for all 
 	$t\in [0,T]$, 
 	$x\in \R^d$, 
 	$v,w\in \R$ 
that 
 	$|f(t,x,v)-f(t,x,w)| \leq L|v-w|$ 
and 
 	\begin{equation} 
 	\tfrac12 \operatorname{Trace}\!\left(
 	\sigma(x)[\sigma(x)]^{*}(\operatorname{Hess} V)(x)
 	\right) 
 	+ 
 	\langle \mu(x), (\nabla V)(x) \rangle 
 	\leq \rho V(x), 
 	\end{equation} 
 assume that 
 $
 \sup_{r\in (0,\infty)} 
 [
 \inf_{x\in \R^d,\norm{x} > r} 
 V(x)
 ]
 = 
 \infty
 $
 and 	
 $
 \inf_{r\in (0,\infty)} 
 [ 
 \sup_{t\in [0,T]}
 \sup_{x\in \R^d, \norm{x} > r}
 ( 
 \frac{|f(t,x,0)|+|g(x)|}{V(x)} 
 )
 ]
 = 0, 
 $	
 let 
 	$(\Omega,\cF,\P,(\F_t)_{t\in [0,T]})$ be a filtered probability space which satisfies the usual conditions, 
 let 
 	$W\colon [0,T]\times\Omega\to\R^m$ be a standard $(\F_t)_{t\in [0,T]}$-Brownian motion, 
 and for every 
 	$t\in [0,T]$,
 	$x\in \R^d$ 
 let 
 	$X^{t,x} = (X^{t,x}_s)_{s\in [t,T]}\colon [t,T]\times\Omega \to \R^d$ 
 be an $(\F_s)_{s\in [t,T]}$-adapted stochastic process with continuous sample paths which satisfies that for all 
	$s\in [t,T]$ 
 it holds $\P$-a.s.~that 
 \begin{equation} 
 X^{t,x}_s = x + \int_t^s \mu(X^{t,x}_r)\,dr + \int_t^s \sigma(X^{t,x}_r)\,dW_r. 
 \end{equation} 
 Then there exists a unique $u\in C([0,T]\times\R^d,\R)$ such that 
 \begin{enumerate}[(i)]
 	\item
 	\label{existence_of_fixpoint_x_dependence_full_space_lyapunov:item1} it holds that 
 	\begin{equation} 
 	\limsup_{r\to\infty} 
 	\left[
 	\sup_{t\in [0,T]}
 	\sup_{\substack{x\in \R^d\\ \norm{x}>r}} 
	\left( \frac{|u(t,x)|}{V(x)} \right)
 	\right]
 	= 0
 	\end{equation}
 	and  	
 	\item
 	\label{existence_of_fixpoint_x_dependence_full_space_lyapunov:item2} for all 
 	$t\in [0,T]$, 
 	$x\in \R^d$ 
 	it holds that 
 	\begin{equation} 
 	u(t,x) 
 	= 
 	\Exp{g(X^{t,x}_{T}) + \int_t^T f\big(s,X^{t,x}_{s},u(s,X^{t,x}_{s})\big)\,ds}
 	\!.
 	\end{equation}  
 \end{enumerate}
\end{cor}

\begin{proof}[Proof of \cref{existence_of_fixpoint_x_dependence_full_space_lyapunov}] 
Throughout this proof let 
 	$\V\colon [0,T]\times\R^d \to (0,\infty)$ 
satisfy for all 
 	$t\in [0,T]$, 
 	$x\in \R^d$ 
that 
 	$\V(t,x) = e^{-\rho t} V(x)$. 
Observe that \cref{moment_estimate_special_lyapunov} (with 
	$\cO = \R^d$, 
	$\mu = ([0,T]\times\R^d \ni (t,x) \mapsto \mu(x) \in \R^d)$, 
	$\sigma = ([0,T]\times\R^d \ni (t,x) \mapsto \sigma(x) \in \R^{d\times m})$ 
in the notation of \cref{moment_estimate_special_lyapunov})
ensures that for all 
 	$t\in [0,T]$, 
 	$x\in \R^d$ 
it holds that 
 	$\V\in C^{1,2}([0,T]\times\R^d,(0,\infty))$ 
and 
 	\begin{equation}
 	(\tfrac{\partial\V}{\partial t})(t,x) 
 	+ 
 	\tfrac12\operatorname{Trace}\!\left( 
 	\sigma(x)[\sigma(x)]^{*}(\operatorname{Hess}_x\V)(t,x)
 	\right)
 	+ 
 	\langle \mu(x), (\nabla_x\V)(t,x) \rangle 
 	\leq 0 . 
 	\end{equation}
Next, observe that the hypothesis that
 	$\sup_{r\in (0,\infty)} [\inf_{x\in\R^d, \norm{x} > r} V(x)] = \infty$
implies that 
 	\begin{equation} \label{existence_of_fixed_point_x_dependence_full_space_lyapunov:boundary_growth}
 	\sup_{r\in (0,\infty)} \left[\inf_{t\in [0,T]} 
 	\inf_{x\in \R^d, \norm{x} > r} \V(t,x)\right] = \infty.
 	\end{equation} 
Furthermore, observe that the hypothesis that   
 	$\inf_{r\in (0,\infty)} [
 	\sup_{t\in [0,T]}
 	\sup_{x\in \R^d, \norm{x} > r} (\frac{|f(t,x,0)|+|g(x)|}{V(x)}) ] = 0$ 
demonstrates that 
 	\begin{equation} 
 	\inf_{r\in (0,\infty)} \left[\sup_{t\in [0,T]} \sup_{x\in \R^d, \norm{x} > r} \left(\frac{|f(t,x,0)|}{\V(t,x)}+\frac{|g(x)|}{\V(T,x)}\right) \right] = 0.
 	\end{equation} 
\cref{existence_sde_setting} (with $\cO=\R^d$, $\mu=([0,T]\times\R^d\ni(t,x)\mapsto \mu(x)\in\R^d)$, $\sigma=([0,T]\times\R^d\ni(t,x)\mapsto \sigma(x)\in\R^{d\times m})$, $V=\V$ in the notation of \cref{existence_sde_setting}) and \eqref{existence_of_fixed_point_x_dependence_full_space_lyapunov:boundary_growth} hence ensure that there exists a unique $u\in C([0,T]\times\R^d,\R)$ which satisfies that 
 \begin{enumerate}[(I)] 
  \item
  \label{existence_of_fixpoint_x_dependence_full_space_lyapunov:itemI} it holds that 
  \begin{equation}
   \limsup_{r\to\infty} \left[ 
    \sup_{t\in [0,T]} 
    \sup_{\substack{x\in \R^d, \\ \norm{x} > r}} 
    \left(
     \frac{|u(t,x)|}{\V(t,x)}
    \right)
   \right] = 0
  \end{equation}
  and   
  \item
  \label{existence_of_fixpoint_x_dependence_full_space_lyapunov:itemII} it holds for all 
  	$t\in [0,T]$, 
  	$x\in \R^d$ 
  that 
  \begin{equation} 
   u(t,x) = \Exp{g(X^{t,x}_T) + \int_t^T f\big(s,X^{t,x}_s,u(s,X^{t,x}_s)\big)\,ds}\!. 
  \end{equation} 
 \end{enumerate}
Next note that Item~\eqref{existence_of_fixpoint_x_dependence_full_space_lyapunov:itemI} implies that 
\begin{equation}
 \limsup_{r\to\infty} 
 \left[ 
  \sup_{t\in [0,T]} 
  \sup_{\substack{x\in \R^d \\ \norm{x} > r}} 
  \left(
   \frac{|u(t,x)|}{V(x)} 
  \right) 
 \right] 
 = 0. 
\end{equation} 
This establishes Item~\eqref{existence_of_fixpoint_x_dependence_full_space_lyapunov:item1}. Moreover, note that Item~\eqref{existence_of_fixpoint_x_dependence_full_space_lyapunov:itemII} establishes Item~\eqref{existence_of_fixpoint_x_dependence_full_space_lyapunov:item2}. The proof of \cref{existence_of_fixpoint_x_dependence_full_space_lyapunov} is thus completed. 	
\end{proof}

Finally, in \cref{existence_of_fixpoint_polynomial_growth} below, we specialize in the setting of \cref{existence_of_fixpoint_x_dependence_full_space_lyapunov} to the situation in which the drift and diffusion coefficients $\mu\colon\R^d\to\R^d$ and $\sigma\colon\R^d\to\R^{d\times m}$ satisfy a coercivity condition and the nonlinearity $f\colon [0,T]\times\R^d\times\R\to\R$ as well as the terminal condition $g\colon \R^d\to\R$ are at most polynomially growing with respect to the spatial variable $x\in\R^d$. Suitable choices for the Lyapunov-type function $V$ are provided by \cref{polynomials_lyapunov}. 

\begin{cor}[Existence and uniqueness of at most polynomially growing solutions of SFPEs]
\label{existence_of_fixpoint_polynomial_growth}
 Let 
 	$d,m\in\N$, 
 	$L,T\in (0,\infty)$, 
let 
 	$\langle\cdot,\cdot\rangle\colon \R^d\times\R^d\to\R$ 
be the standard scalar product on $\R^d$, 
let 
 	$\norm{\cdot}\colon \R^d \to [0,\infty)$ 
be the standard norm on $\R^d$, 
let 
 	$\HSnorm{\cdot}\colon \R^{d\times m}\to [0,\infty)$ be the Frobenius norm on $\R^{d\times m}$, 
let 
 	$\mu\colon\R^d\to\R^d$ and 
 	$\sigma\colon \R^d\to\R^{d\times m}$
be locally Lipschitz continuous,  
let 
 	$f\in C([0,T]\times\R^d\times\R,\R)$, 
 	$g\in C(\R^d,\R)$ 
be at most polynomially growing,  
assume for all 
	$t\in [0,T]$, 
	$x\in \R^d$, 
	$v,w\in \R$ 
that
	$\max\{\langle x,\mu(x) \rangle, \HSNorm{\sigma(x)}^2\} \leq L (1+\norm{x}^2)$
and  
	$|f(t,x,v)-f(t,x,w)| \leq L|v-w|$, 
let 
 	$(\Omega,\cF,\P,(\F_t)_{t\in [0,T]})$ be a filtered probability space which satisfies the usual conditions, 
let 
 	$W\colon [0,T]\times\Omega\to\R^m$ be a standard $(\F_t)_{t\in [0,T]}$-Brownian motion, 
and for every 
 	$t\in [0,T]$, 
 	$x\in \R^d$ 
let 
 	$X^{t,x} = (X^{t,x}_s)_{s\in [t,T]}\colon [t,T]\times\Omega \to \R^d$ 
be an $(\F_s)_{s\in [t,T]}$-adapted stochastic process with continuous sample paths which satisfies that for all 
 	$s\in [t,T]$ 
it holds $\P$-a.s.~that 
 \begin{equation} 
  X^{t,x}_s = x + \int_t^s \mu(X^{t,x}_r)\,dr + \int_t^s \sigma(X^{t,x}_r)\,dW_r. 
 \end{equation} 
Then there exists a unique $u\in C([0,T]\times\R^d,\R)$ such that 
\begin{enumerate}[(i)]
 	\item
 	\label{existence_of_fixpoint_polynomial_growth:item1} it holds that $u$ is at most polynomially growing and  	
 	\item
 	\label{existence_of_fixpoint_polynomial_growth:item2} it holds for all 
 		$t\in [0,T]$, 
 		$x\in \R^d$ 
 	that 
 	\begin{equation} 
 	 u(t,x) 
 	 = 
 	 \Exp{g(X^{t,x}_{T}) + \int_t^T f\big(s,X^{t,x}_{s},u(s,X^{t,x}_{s})\big)\,ds}
 	 \!.
 	\end{equation}  
\end{enumerate}
\end{cor}

\begin{proof}[Proof of \cref{existence_of_fixpoint_polynomial_growth}]  Throughout this proof let 
 	$\rho_q\in (0,\infty)$, 
 	$q\in (0,\infty)$, 
satisfy for every 
 	$q\in (0,\infty)$
that 
 	$\rho_q = \frac{qL}{2}\max\{q+1,3\}$, 
let 
	$p\in (0,\infty)$
satisfy that
	$\sup_{t\in [0,T]}
	\sup_{y\in \R^d} [ \frac{|f(t,y,0)| + |g(y)|}{1+\norm{y}^p} ]
	< \infty$, 
and let 
	$V_q\colon \R^d \to \R$, 
	$q\in (0,\infty)$, 
satisfy for all 
	$q\in (0,\infty)$,
	$x\in \R^d$ 
that 
 	$V_q(x) = [1+\norm{x}^2]^{\nicefrac{q}{2}}$. 
Note that the fact that 
 	$
 	\sup_{t\in [0,T]}\sup_{x\in\R^d} [\frac{|f(t,x,0)|+|g(x)|}{1+\norm{x}^p}]
 	< \infty 
 	$
implies that for all 
 	$q\in (p,\infty)$ 
it holds that 
 	\begin{equation}
 	\label{existence_of_fixpoint_polynomial_growth:growth_towards_boundary}
	\begin{split}
 	& 
	\limsup_{r \to \infty}  
 	\left[ 
 	\sup_{t\in [0,T]}
 	\sup_{\substack{x\in\R^d, \\ \norm{x} > r}}
 	\left(
 		\frac{|f(t,x,0)|+|g(x)|}{V_q(x)}
 	\right) 
 	\right] 
    \\
	& = 	
	\limsup_{r \to \infty} 
	\left[ 
	\sup_{t\in [0,T]} 
	\sup_{\substack{x\in\R^d, \\ \norm{x} > r}}
	\left( 
	\left[
		\frac{|f(t,x,0)|+|g(x)|}{1+\norm{x}^p} 
	\right]
	\left[ 
		\frac{1+\norm{x}^p}{V_q(x)} 
	\right]
	\right)
	\right] 
	\\
	& \leq 
	\left[ 
	\sup_{t\in [0,T]} 
	\sup_{x\in\R^d}
	\left(
	\frac{|f(t,x,0)|+|g(x)|}{1+\norm{x}^p}
	\right)
	\right]
	\left[ 
 	\limsup_{r\to\infty} 
	\left(
	\sup_{t\in [0,T]}
	\sup_{\substack{x\in\R^d, \\ \norm{x}> r}} 
	\left( 
	\frac{1+\norm{x}^p}{V_q(x)} 
	\right) 
	\right)
	\right]
	= 0. 
	\end{split}
 	\end{equation}
Moreover, observe that for all 
 	$q\in (p,\infty)$ 
it holds that 
 	\begin{equation}
 	\label{existence_of_fixpoint_polynomial_growth:V_blows_up_towards_boundary}
 	 \sup_{r\in (0,\infty)} 
 	 \left[ 
 	  \inf_{\substack{x\in\R^d, \\ \norm{x}>r}} 
 	  V_q(x)
 	 \right] 
 	 = \infty. 
 	\end{equation}
Next note that \cref{polynomials_lyapunov} ensures for all
 	$q\in (0,\infty)$, 
 	$x\in \R^d$ 
that 
 \begin{equation} 
  \tfrac12 \operatorname{Trace}\!\left( 
   \sigma(x)[\sigma(x)]^{*}(\operatorname{Hess} V_q)(x)
  \right)
  + 
  \langle \mu(x), (\nabla V_q)(x)\rangle 
  \leq \rho_q V_q(x).  
 \end{equation} 
Combining this with  \eqref{existence_of_fixpoint_polynomial_growth:growth_towards_boundary}, \eqref{existence_of_fixpoint_polynomial_growth:V_blows_up_towards_boundary}, and \cref{existence_of_fixpoint_x_dependence_full_space_lyapunov} (with $V=V_{2p}$ in the notation of \cref{existence_of_fixpoint_x_dependence_full_space_lyapunov}) yields that there exists a unique $u\in C([0,T]\times\R^d,\R)$ which satisfies for all 
 	$t\in [0,T]$, 
 	$x\in\R^d$
that 
	$
	\limsup_{r\to\infty} 
  	[
  	\sup_{s\in [0,T]}\sup_{y\in\R^d, \norm{y}>r}(\frac{|u(s,y)|}{V_{2p}(y)})
  	]
  	= 0
  	$
and
 	$
  	u(t,x) = \EXP{g(X^{t,x}_T)+\int_t^T f\big(s,X^{t,x}_s,u(s,X^{t,x}_s)\big)\,ds}. 
 	$
In particular, this ensures that $u\colon [0,T]\times\R^d\to\R$ is at most polynomially growing. This establishes that $u\in C([0,T]\times\R^d,\R)$ satisfies Items~\eqref{existence_of_fixpoint_polynomial_growth:item1} and \eqref{existence_of_fixpoint_polynomial_growth:item2}.  
It remains to prove that $u\colon [0,T]\times\R^d\to\R$ is the only continuous function which satisfies Items~\eqref{existence_of_fixpoint_polynomial_growth:item1} and \eqref{existence_of_fixpoint_polynomial_growth:item2}. For this, let 
 	$v\in C([0,T]\times\R^d,\R)$ 
be an at most polynomially growing function which satisfies for all 
 	$t\in [0,T]$, 
 	$x\in \R^d$ 
that 
 	$v(t,x)=\EXP{g(X^{t,x}_T)+\int_t^T f(s,X^{t,x}_s,v(s,X^{t,x}_s))\,ds}$. 
The fact that $v\colon [0,T]\times\R^d\to\R$ is at most polynomially growing ensures that there exists 
 	$q\in (0,\infty)$ 
which satisfies that 
 	$
 	\sup_{t\in [0,T]}\sup_{x\in\R^d} [\frac{|v(t,x)|}{1+\norm{x}^q}]
 	< 
 	\infty
 	$. 
This implies that $u,v\in C([0,T]\times\R^d,\R)$ satisfy for all 
 	$t\in [0,T]$, 
 	$x\in \R^d$ 
that 
	$
	\limsup_{r\to\infty} 
	[ 
	\sup_{s\in [0,T]}
	\sup_{y\in \R^d,\norm{y}>r}
	(
	\frac{|u(s,y)| + |v(s,y)|}{V_{\max\{2q,2p\}}(y)} 
	)
	] 
	= 
	0
	$, 
	$
	u(t,x) = \EXP{g(X^{t,x}_T)+\int_t^T f\big(s,X^{t,x}_s,u(s,X^{t,x}_s)\big)\,ds}
	$, 
and 
  \begin{equation} 
   v(t,x) = \Exp{g(X^{t,x}_T)+\int_t^T f\big(s,X^{t,x}_s,v(s,X^{t,x}_s)\big)\,ds}\!. 
  \end{equation} 
\cref{existence_of_fixpoint_x_dependence_full_space_lyapunov} (with $V=V_{\max\{2q,2p\}}$ in the notation of \cref{existence_of_fixpoint_x_dependence_full_space_lyapunov}) hence ensures that $u=v$. 
This establishes that $u\colon [0,T]\times\R^d\to\R$ is the unique continuous function which satisfies Items~\eqref{existence_of_fixpoint_polynomial_growth:item1} and \eqref{existence_of_fixpoint_polynomial_growth:item2}. The proof of \cref{existence_of_fixpoint_polynomial_growth} is thus completed.
\end{proof}

\bibliographystyle{acm}
\bibliography{PDE_approximation_bibfile}
\end{document}